\documentclass[a4paper,11pt]{amsart}

\usepackage[foot]{amsaddr}

\usepackage{a4wide}
\usepackage{pgfplots}
\usetikzlibrary{patterns}
\usepackage{subcaption}
\usepackage{rotating}

\usepackage[breaklinks,bookmarks=false]{hyperref}
\hypersetup{colorlinks, linkcolor=blue, citecolor=blue,urlcolor=blue,plainpages=false}

\usepackage{xcolor}

\newtheorem{theorem}{Theorem}
\newtheorem{lemma}[theorem]{Lemma}

\newtheorem{assumption}[theorem]{Assumption}
\newtheorem{remark}[theorem]{Remark}

\numberwithin{theorem}{section}
\numberwithin{equation}{section}

\usepackage{amssymb}
\usepackage{mathrsfs}

\renewcommand{\Re}{\operatorname{Re}}

\newcommand{\eq}{:=}

\newcommand{\pd}[2]{\frac{\partial #1}{\partial #2}}

\newcommand{\grad}{\boldsymbol \nabla}
\renewcommand{\div}{\grad \cdot}
\newcommand{\curl}{\grad \times}

\newcommand{\ccurl}{\boldsymbol{\operatorname{curl}}}
\newcommand{\ddiv}{\operatorname{div}}

\newcommand{\jmp}[1]{\,[\![#1]\!]}
\newcommand{\avg}[1]{\{\!\!\{#1\}\!\!\}}

\newcommand{\BE}{\boldsymbol E}

\newcommand{\BH}{\boldsymbol H}
\newcommand{\BI}{\boldsymbol I}
\newcommand{\BJ}{\boldsymbol J}

\newcommand{\BW}{\boldsymbol W}

\newcommand{\bd}{\boldsymbol d}
\newcommand{\be}{\boldsymbol e}

\newcommand{\bh}{\boldsymbol h}

\newcommand{\bj}{\boldsymbol j}

\newcommand{\bl}{\boldsymbol l}

\newcommand{\bn}{\boldsymbol n}
\newcommand{\bo}{\boldsymbol o}
\newcommand{\bp}{\boldsymbol p}
\newcommand{\bq}{\boldsymbol q}

\newcommand{\bu}{\boldsymbol u}
\newcommand{\bv}{\boldsymbol v}
\newcommand{\bw}{\boldsymbol w}
\newcommand{\bx}{\boldsymbol x}
\newcommand{\by}{\boldsymbol y}

\newcommand{\CF}{\mathcal F}

\newcommand{\CH}{\mathcal H}

\newcommand{\CP}{\mathcal P}
\newcommand{\CQ}{\mathcal Q}
\newcommand{\CR}{\mathcal R}
\newcommand{\CS}{\mathcal S}
\newcommand{\CT}{\mathcal T}

\newcommand{\LC}{\mathscr C}
\newcommand{\LD}{\mathscr D}
\newcommand{\LE}{\mathscr E}

\newcommand{\LL}{\mathscr L}

\newcommand{\LP}{\mathscr P}

\newcommand{\LV}{\mathscr V}

\newcommand{\TD}{\textup D}

\newcommand{\TS}{\textup S}

\newcommand{\TU}{\textup U}

\newcommand{\BCH}{\boldsymbol{\CH}}

\newcommand{\BCP}{\boldsymbol{\CP}}

\newcommand{\BLL}{\pmb{\LL}}

\newcommand{\BTD}{\pmb{\TD}}

\newcommand{\BTS}{\pmb{\TS}}

\newcommand{\BTU}{\pmb{\TU}}

\newcommand{\SlopeTriangle}[6]
{

    \pgfplotsextra
    {
        \pgfkeysgetvalue{/pgfplots/xmin}{\xmin}
        \pgfkeysgetvalue{/pgfplots/xmax}{\xmax}
        \pgfkeysgetvalue{/pgfplots/ymin}{\ymin}
        \pgfkeysgetvalue{/pgfplots/ymax}{\ymax}

        \pgfmathsetmacro{\xArel}{#1}
        \pgfmathsetmacro{\yArel}{#3}
        \pgfmathsetmacro{\xBrel}{#1-#2}
        \pgfmathsetmacro{\yBrel}{\yArel}
        \pgfmathsetmacro{\xCrel}{\xArel}

        \pgfmathsetmacro{\lnxB}{\xmin*(1-(#1-#2))+\xmax*(#1-#2)} 
        \pgfmathsetmacro{\lnxA}{\xmin*(1-#1)+\xmax*#1} 
        \pgfmathsetmacro{\lnyA}{\ymin*(1-#3)+\ymax*#3} 
        \pgfmathsetmacro{\lnyC}{\lnyA+#4*(\lnxA-\lnxB)}
        \pgfmathsetmacro{\yCrel}{\lnyC-\ymin)/(\ymax-\ymin)} 

        \coordinate (A) at (rel axis cs:\xArel,\yArel);
        \coordinate (B) at (rel axis cs:\xBrel,\yBrel);
        \coordinate (C) at (rel axis cs:\xCrel,\yCrel);

        \draw[#6]   (A)-- node[anchor=north] {#5}
                    (B)--
                    (C)--
                    cycle;
    }
}

\newcommand{\eps}{\varepsilon}

\newcommand{\ee}{{\boldsymbol \eps}}
\newcommand{\mm}{{\boldsymbol \mu}}
\newcommand{\cc}{{\boldsymbol \chi}}

\newcommand{\zero}{\bo}
\newcommand{\norm}[1]{|\!|\!|#1|\!|\!|}

\newcommand{\enorm}[1]{\norm{#1}_{\ccurl,\omega,\Omega}}
\newcommand{\enormE}[2]{\norm{#1}_{\ccurl,\omega,\ee,\cc,#2}}
\newcommand{\enormH}[2]{\norm{#1}_{\ccurl,\omega,\mm,\bzet,#2}}

\newcommand{\enormEH}[2]{\norm{#1}_{\ccurl,\omega,#2}}

\newcommand{\lO}{\ell_\Omega}

\newcommand{\bthe}{\boldsymbol \theta}
\newcommand{\bphi}{\boldsymbol \phi}
\newcommand{\bpsi}{\boldsymbol \psi}
\newcommand{\bxi}{\boldsymbol \xi}
\newcommand{\bvthe}{\boldsymbol \vartheta}

\newcommand{\bzet}{\boldsymbol \zeta}

\newcommand{\tbe}{\widetilde \be}
\newcommand{\tbh}{\widetilde \bh}

\newcommand{\tbj}{\widetilde \bj}
\newcommand{\tbl}{\widetilde \bl}

\newcommand{\cR}{\LC_{\rm r}}
\newcommand{\cI}{\LC_{\rm i}}

\newcommand{\cb}{\LC_{\rm b}}

\newcommand{\cP}{\LC_{\rm pol}}

\newcommand{\cSE}{\gamma_{\rm st,E}}
\newcommand{\gbaE}{\gamma_{\rm ba,E}}

\newcommand{\cSH}{\gamma_{\rm st,H}}
\newcommand{\gbaH}{\gamma_{\rm ba,H}}

\newcommand{\cS}{\gamma_{\rm st}}
\newcommand{\gba}{\gamma_{\rm ba}}

\newcommand{\tK}{\widetilde K}
\newcommand{\tF}{\widetilde F}

\newcommand{\eemO}{\varepsilon_{\Omega,\min}}
\newcommand{\eeMO}{\varepsilon_{\Omega,\max}}

\newcommand{\mmmO}{\mu_{\Omega,\min}}
\newcommand{\mmMO}{\mu_{\Omega,\max}}

\newcommand{\eemK}{\varepsilon_{\tK,\min}}
\newcommand{\mmMK}{\mu_{\tK,\max}}

\newcommand{\eeMK}{\varepsilon_{\tK,\max}}
\newcommand{\mmmK}{\mu_{\tK,\min}}

\newcommand{\essinf}[1]{\underset{\substack{#1}}{\operatorname{ess} \operatorname{inf}}\;}
\newcommand{\esssup}[1]{\underset{\substack{#1}}{\operatorname{ess} \operatorname{sup}}\;}

\newcommand{\pol}{\mathcal{P}}

\newcommand{\tV}{\widetilde V}
\newcommand{\tBW}{\widetilde{\BW}}

\newcommand{\tCQ}{\widetilde \CQ}
\newcommand{\tCR}{\widetilde \CR}

\newcommand{\osc}{\operatorname{osc}}

\newcommand{\se}{b_{\rm E}}
\newcommand{\sh}{b_{\rm H}}

\newcommand{\Gi}{\Gamma_{\rm i}}
\newcommand{\bni}{\bn_{\rm i}}
\newcommand{\OTF}{\Omega_{\rm tf}}
\newcommand{\BEi}{\BE^{\rm inc}}
\newcommand{\BHi}{\BH^{\rm inc}}

\title[Frequency-explicit a posteriori error estimates for Maxwell's equations]{Frequency-explicit a posteriori error estimates for discontinuous Galerkin discretizations of Maxwell's equations}

\author{T. Chaumont-Frelet$^{\star,\dagger}$}
\author{P. Vega$^{\ddagger}$}

\address{\vspace{-.5cm}}
\address{\noindent \tiny \textup{$^\star$Inria, 2004 Route des Lucioles, 06902 Valbonne, France}}
\address{\noindent \tiny \textup{$^\dagger$Laboratoire J.A. Dieudonn\'e, Parc Valrose, 28 Avenue Valrose, 06108 Nice Cedex 02, 06000 Nice, France}}
\address{\noindent \tiny \textup{$^\ddagger$Instituto de Matem\'aticas, Pontificia Universidad Cat\'olica de Valpara\'{\i}so, Blanco Viel 596, Valpara\'{\i}so, Chile}}

\begin{document}

\maketitle

\begin{abstract}
We propose a new residual-based a posteriori error estimator for discontinuous
Galerkin discretizations of time-harmonic Maxwell's equations in first-order form.
We establish that the estimator is reliable and efficient, and the dependency of the
reliability and efficiency constants on the frequency is analyzed and discussed.
The proposed estimates generalize similar results previously obtained for the
Helmholtz equation and conforming finite element discretization of Maxwell's equations.
In addition, for the discontinuous Galerkin scheme considered here, we also show that
the proposed estimator is asymptotically constant-free for smooth solutions. We also present
two-dimensional numerical examples that highlight our key theoretical findings and
suggest that the proposed estimator is suited to drive $h$- and $hp$-adaptive iterative
refinements.

\vspace{.5cm}
\end{abstract}


\section{Introduction}

Time-harmonic Maxwell's equations are a central model in a variety
of applications involving electromagnetic fields \cite{dorf_2006a,griffiths_1999a,viquerat_2015a}.
Maxwell's equations cannot be analytically solved in complex settings frequently encountered in
applications. Instead, numerical simulation tools based on
discretization methods are commonly employed in practice, including boundary element
\cite{sauter_schwab_2010a}, finite difference \cite{taflove_hagness_2005a}, finite element
\cite{chaumontfrelet_vega_2020,monk_2003a}, and discontinuous Galerkin
\cite{feng_lu_xu_2016a,hesthaven_warburton_2002a,li_lanteri_perrussel_2014a,viquerat_2015a}
methods. Irrespective of the discretization technique, there is a need to make the simulation as
efficient as possible, either to treat problems faster or to tackle larger problems.

In this work, we focus on discontinuous Galerkin (DG) discretizations of Maxwell's equations
in the first-order form, where we look for $\BE,\BH: \Omega \to \mathbb C^3$ such that
\begin{equation}
\label{eq_maxwell_strong_first_order}
\left \{
\begin{array}{rcll}
i\omega \ee \BE - \curl \BH &=& \BJ & \text{ in } \Omega,
\\
i\omega \mm \BH + \curl \BE &=& \zero& \text{ in } \Omega,
\\
\BE \times \bn &=& \zero & \text{ on } \partial \Omega,
\\
\mm \BH \cdot \bn &=& 0 & \text{ on } \partial \Omega,
\end{array}
\right .
\end{equation}
where $\ee,\mm: \Omega \to \mathbb C^{3 \times 3}$ are given coefficients
representing the electric permittivity and magnetic permeability of the materials contained
in $\Omega$, $\BJ: \Omega \to \mathbb C^3$ is the (known) current density,
and $\omega > 0$ is the frequency. As we elaborate below, this setting is rather general
and, in particular, unbounded propagation media fit the framework of
\eqref{eq_maxwell_strong_first_order} if the coefficients $\ee,\mm$ are suitably modified
using perfectly matched layers \cite{berenger_1994,berenger_1996,monk_2003a}.

An attractive solution to reduce computational costs is to use
a posteriori error estimators coupled with adaptive strategies
\cite{bespalov_haberl_praetorius_2017a,bonito_nochetto_2010a,verfurth_1994}.
In this framework, after a discrete solution $(\BE_h,\BH_h)$ has been computed on
a given mesh $\CT_h$, an error estimator $\eta_K$ is associated with each element
$K \in \CT_h$. These estimators can then be used to decide where to locally refine
the mesh, leading to an adaptive loop procedure. Besides, the cumulated estimator
$\eta^2 \eq \sum_{K \in \CT_h} \eta_K^2$ can be used to stop the adaptive loop
and, more generally, to assess the error level reliably.

In this context, DG methods are especially interesting, as they allow for
an easy implementation of hanging nodes and $p$-adaptivity \cite{bonito_nochetto_2010a}.
Another appealing aspect of these methods that we highlight in Section \ref{section_tfsf}
is their ability to impose prescribed jumps in the solution easily.
For wave propagation problems, this is especially useful to inject incident fields
\cite[\S4.2.2]{viquerat_2015a}.
Besides, DG methods are sometimes more stable than conforming finite elements on
coarse meshes (compare, e.g. \cite[Theorem 3.5]{feng_lu_xu_2016a} and
\cite[Lemma 3.7]{bernkopf_sauter_torres_veit_2022a}),
which may be relevant in adaptive processes starting with coarse discretizations.

\begin{subequations}
\label{eq_intro_estimator}
The key contribution of this work is the design of a new residual-based a posteriori error
estimator for DG discretizations of \eqref{eq_maxwell_strong_first_order}. Specifically, our
proposed estimator is reliable
\begin{equation}
\label{eq_intro_reliability}
\norm{(\BE-\BE_h,\BH-\BH_h)}_\Omega
\leq
C \!\left (
1 + \max_{K \in \CT_h} \frac{\omega h_K}{p_K \vartheta_K}
\right ) \!(1+\gba) \eta
\end{equation}
and locally efficient
\begin{equation}
\label{eq_intro_efficiency}
\eta_K
\leq
C p_K^{3/2}\!
\left (
1 + \frac{\omega h_K}{p_K \vartheta_K}
\right )\!
\norm{(\BE-\BE_h,\BH-\BH_h)}_{\widetilde K}
+
\osc_K,
\end{equation}
where $\norm{\cdot}$ is a suited ``energy norm'', $h_K$ and $p_K$ are the size and
the polynomial degree associated with the element $K \in \CT_h$, and $\vartheta_K$ is
a measure of the wavespeed around $K$. The real number $\gba > 0$ in
\eqref{eq_intro_reliability} is the so-called ``approximation factor''
\cite{chaumontfrelet_ern_vohralik_2019,chaumontfrelet_vega_2020,dorfler_sauter_2013a}.
It generally grows with the frequency but tends to zero as
$\max_{K \in \CT_h} h_K/p_K \to 0$, see Section \ref{section_approximation_factor} below
as well as \cite{chaumontfrelet_vega_2020,chaumontfrelet_vega_2021a} for more details.
The term $\osc_K$ in \eqref{eq_intro_efficiency} is a ``data oscillation term'' customary in
efficiency estimates. We refer the reader to Theorems \ref{rel_LDG} and
\ref{eff_LDG} below, where the estimates in \eqref{eq_intro_estimator} are established.
Interestingly, these results generalize similar findings for the Helmholtz equation
\cite{dorfler_sauter_2013a,sauter_zech_2015a} and conforming N\'ed\'elec approximations
of Maxwell's equations \cite{chaumontfrelet_vega_2020}.
\end{subequations}

On the analysis side, although the estimates we obtain for DG discretization
in \eqref{eq_intro_estimator} are rather similar to the ones established in
\cite{chaumontfrelet_vega_2020} for conforming N\'ed\'elec elements, the arguments
employed are fairly different. In particular, following
\cite{chaumontfrelet_vega_2020,dorfler_sauter_2013a}, our analysis relies
on duality arguments that turn out to be substantially complicated in the first-order
setting we consider here in \eqref{eq_maxwell_strong_first_order}, as can be seen in Lemma
\ref{lemma_approximation_factor_first_order} below. DG schemes also have
some unique properties as compared to N\'ed\'elec discretizations. For instance, we
also show that for smooth solutions (or locally refined meshes), our estimator is
asymptotically constant-free, meaning that
\begin{equation}
\label{eq_intro_exact}
\norm{(\BE-\BE_h,\BH-\BH_h)}_\Omega
\leq \left (1+\theta\left(\max_{K \in \CT_h} h_K/p_K\right)\!\right )\!\eta
\end{equation}
with $\lim_{t \to 0} \theta(t) = 0$.
This is detailed in Remark \ref{remark_asymptotic_exactness}. Another interesting
fact is that for divergence-free right-hand sides ($\div \BJ = 0$), the estimator
is ``oscillation free''.

We also present a series of detailed two-dimensional numerical experiments. On the one hand, we 
showcase a few academic benchmarks with known analytical solutions that we use to highlight
the main features of our results. In particular, the interplay between the frequency,
the mesh size and polynomial degree, and the constants appearing in the estimates
\eqref{eq_intro_reliability}, \eqref{eq_intro_efficiency}, and \eqref{eq_intro_exact} is
largely illustrated and thoroughly discussed. On the other hand, we also consider more
realistic benchmarks where the analytic solution is unavailable. In these cases, we couple
the estimator with D\"orfler's marking \cite{dorfler_1996a} and newest vertex bisection
\cite{binev_dahmen_devore_2004a} to drive adaptive mesh refinements. We also consider
a simple $hp$-adaptive procedure based on \cite{melenk_wohlmuth_2001a}. For these examples,
we observe optimal convergence rates in all cases, which indicates that the estimator is suited
to drive this kind of adaptive refinements.

The remainder of our work is organized as follows.
In Section \ref{sec_settings}, we precise the setting and main assumptions,
and recall useful standard tools. Section \ref{sec_IPDG} presents the range of DG
schemes for which our analysis applies and gives the construction of the
estimator. Section \ref{sec_IPDG} contains the proof of our main results.
Finally, we report on numerical examples in Section \ref{sec_numerics} before
providing some concluding remarks in Section \ref{sec_conclusion}.

\section{Preliminaries}
\label{sec_settings}

We start by providing key notations and preliminary results.

\subsection{Domain and coefficients}
\label{section_coefficients}

We consider Maxwell's equations \eqref{eq_maxwell_strong_first_order}
in a Lipschitz polyhedral (not necessarily simply connected) domain
$\Omega \subset \mathbb R^3$. We denote by $\lO \eq \sup_{\bx,\by \in \Omega} |\bx-\by|$
the diameter of $\Omega$.

The coefficients $\ee,\mm: \Omega \to \CS(\mathbb C^3)$ are
two symmetric (but not necessarily self-adjoint) tensor-valued
functions describing the electromagnetic properties of the materials
contained inside $\Omega$. For the sake of simplicity, we assume that
$\Omega$ can be partitioned into a set $\LP$ of non-overlapping polyhedral
subdomains $P$ such that $\ee|_P$ and $\mm|_P$ are constant for all $P \in \LP$.
The short-hand notations $\bzet \eq \ee^{-1}$ and $\cc \eq \mm^{-1}$ will also be useful.

If $\bphi$ is any of the tensor fields mentioned above, we introduce the notations
\begin{equation*}
\phi_{\min}(\bx)
\eq
\min_{\substack{\bu \in \mathbb C^{3} \\ |\bu| = 1}} \Re \bphi(\bx) \bu \cdot \overline{\bu},
\qquad
\phi_{\max}(\bx)
\eq
\max_{\substack{\bu \in \mathbb C^{3} \\ |\bu| = 1}}
\max_{\substack{\bv \in \mathbb C^{3} \\ |\bv| = 1}}
\Re \bphi(\bx) \bu \cdot \overline{\bv},
\end{equation*}
as well as
\begin{equation*}
\phi_{D,\min} \eq \essinf{\bx \in D} \phi_{\min}(\bx),
\qquad
\phi_{D,\max} \eq \esssup{\bx \in D} \phi_{\max}(\bx)
\end{equation*}
for any open set $D \subset \Omega$, and we assume that $\phi_{\Omega,\min} > 0$.
Finally, for $D \subset \Omega$, the notation
\begin{align*}
c_{D,\min}\eq(\varepsilon_{D,\max}\mu_{D,\max})^{-1/2},
\qquad
c_{D,\max}\eq(\varepsilon_{D,\min}\mu_{D,\min})^{-1/2}
\end{align*}
are employed for the smallest and highest wavespeeds in $D$.

As described in \cite[Remark 2.1]{chaumontfrelet_vega_2020}, these assumptions
cover most interesting application scenarios, including dissipative materials
and perfectly matched layers \cite{berenger_1994,berenger_1996,monk_2003a}.

\subsection{Functional spaces}
\label{functional_spaces}

In the following, if $D \subset \Omega$, $\LL^2(D)$ denotes the space of square-integrable
complex-valued function over $D$, see e.g. \cite{adams_fournier_2003a}, and
$\BLL^2(D) \eq \left (\LL^2(D) \right )^3$. We equip $\BLL^2(D)$ with the following
(equivalent) norms
\begin{equation*}
	\|\bw\|_D^2 \eq \int_D |\bw|^2,
	\qquad
	\|\bw\|_{\bphi,D}^2 \eq \Re \int_D \bphi \bw \cdot \overline{\bw},
	\qquad
	\bw \in \BLL^2(D)
\end{equation*}
where $\bphi \in \{\ee,\mm,\bzet,\cc\}$.
We denote by $(\cdot,\cdot)_D$ the inner-product of $\BLL^2(D)$,
and we drop the subscript when $D = \Omega$.
If $F \subset \overline{\Omega}$ is a two-dimensional measurable planar subset,
$\|\cdot\|_F$ and $\langle \cdot,\cdot \rangle_F$
respectively denote the natural norm and inner-product of
both $\LL^2(F)$ and $\BLL^2(F) \eq \left (\LL(F)\right )^3$.

Classically \cite{adams_fournier_2003a}, we employ the notation $\CH^1(D)$
for the usual Sobolev space of functions $w \in \LL^2(D)$ such that
$\grad w \in \BLL^2(D)$. We also set $\BCH^1(D) \eq \left (\CH^1(D)\right )^3$
and introduce the semi-norms
\begin{equation*}
	\|\grad \bw\|_D^2
	\eq
	\sum_{j,k = 1}^3 \int_D \left |\pd{\bw_j}{\bx_k}\right |^2,
	\qquad
	\|\grad \bw\|_{\bphi,D}^2
	\eq
	\sum_{j,k = 1}^3 \int_D \phi_{\max} \left |\pd{\bw_j}{\bx_k}\right |^2,
\end{equation*}
for $\bw \in \BCH^1(\Omega)$ and $\bphi \in \{\ee,\mm,\bzet,\cc\}$.

We shall also need Sobolev spaces of vector-valued functions with well-defined
divergence and rotation \cite{girault_raviart_1986a}. Specifically, we denote
by $\BCH(\ccurl,D)$ the space of functions $\bw \in \BLL^2(D)$
with $\curl \bw \in \BLL^2(D)$, that we equip with the norms
\begin{equation*}
	\norm{\bv}_{\ccurl,\omega,\bphi,\bpsi,D}^2
	\eq
	\omega^2 \|\bv\|_{\bphi,D}^2 + \|\curl \bv\|_{\bpsi,D}^2,
	\qquad \bv \in \BCH(\ccurl,D),
\end{equation*}
for $\bphi,\bpsi \in \{\ee,\mm,\bzet,\cc\}$. We also introduce the following
``energy'' norm
\begin{equation*}
	\enormEH{(\be,\bh)}{D}^2
	\eq
	\enormE{\be}{D}^2 + \enormH{\bh}{D}^2,
	\qquad
	(\be,\bh) \in \BCH_0(\ccurl,D) \times \BCH(\ccurl,D),
\end{equation*}
for electromagnetic fields on the product space.

In addition, if $\bxi: D \to \mathbb C^3$ is a measurable tensor-valued
function, we will use the notation $\BCH(\ddiv,\bxi,D)$ for the
set of functions $\bw \in \BLL^2(D)$ with $\div (\bxi \bw) \in L^2(D)$,
and we will write $\BCH(\ddiv^0,\bxi,D)$ for the set of fields
$\bw \in \BCH(\ddiv,\bxi,D)$ such that $\div (\bxi \bw) = 0$
in $D$. When $\bxi = \BI$, the identity tensor,
we simply write $\BCH(\ddiv,D)$ and $\BCH(\ddiv^0,D)$.

For any of the aforementioned spaces $\LV$, the notation
$\LV_0$ denotes the closure of smooth, compactly supported functions
into $\LL^2(D)$ (or $\BLL^2(D))$ with respect to the norm of
$\LV$. These spaces also correspond to the kernel of the naturally associated
trace operators \cite{adams_fournier_2003a,girault_raviart_1986a}.

Finally, if $\LD$ is a collection of disjoint sets $D \subset \Omega$
and $\LV(D)$ is any of the aforementioned spaces, $\LV(\LD)$ stands for the
``broken'' space of functions in $\bv \in \BLL^2(\Omega)$ (or $\LL^2(\Omega)$)
such that $\bv|_D \in \LV(D)$ for all $D \in \LD$. We employ the same
notation for the inner-products, norms, and semi-norms of $\LV(\LD)$ and $\LV(D)$,
with the subscript $\LD$ instead of $D$.

\subsection{TF-SF formulation}
\label{section_tfsf}

We consider a TF-SF interface $\Gi$ that is either empty (there are no incident fields)
or the boundary of a Lipschitz polyhedral subdomain $\OTF \subset \Omega$.

We then consider incident fields
$\BEi,\BHi \in \BCH(\ccurl,\OTF) \cap \BCH(\ddiv,\OTF) \cap \BLL^2(\Gi)$ such that
\begin{equation}
\label{eq_assumption_incident_fields}
\left \{
\begin{array}{rcl}
i\omega \varepsilon_0 \BEi-\curl \BHi &=& \bo
\\
i\omega \mu_0 \BHi+\curl \BEi &=& \bo
\end{array}
\right .
\end{equation}
in $\OTF$, where $\varepsilon_0,\mu_0 > 0$ are arbitrary values that typically correspond
to the vacuum electric permittivity and magnetic permeability in applications. In addition,
for the sake of simplicity, we assume that $\ee = \varepsilon_0 \BI$ and $\mm = \mu_0 \BI$
in a neighborhood of $\Gi$.

The TF-SF interface $\Gi$ is then employed to inject the incident fields $\BEi$ and $\BHi$
into the computational domain via jump conditions, see e.g. \cite[\S4.2.2]{viquerat_2015a}.
	
\subsection{Variational formulations}

In the remainder of this work, we assume that $\BJ \in \BCH(\ddiv,\Omega)$. Then,
we may recast \eqref{eq_maxwell_strong_first_order} into a weak formulation, which consists in
finding a pair $(\BE,\BH)\in \BLL^2(\Omega) \times \BLL^2(\Omega)$ satisfying
\begin{equation}
\label{eq_maxwell_weak_first_order}
b((\BE,\BH),(\bv,\bw)) = i\omega (\BJ,\bv) + i\omega \ell(\bv,\bw)
\qquad
\forall (\bv,\bw) \in \BCH_0(\ccurl,\Omega) \times \BCH(\ccurl,\Omega)
\end{equation}
where
\begin{equation*}
b((\be,\bh),(\bv,\bw))
\eq
-\omega^2\left(\mm\bh,\bw\right)
+
i\omega\left(\be,\curl\bw\right)
-
i\omega\left(\bh,\curl\bv\right)
-
\omega^2\left(\ee\be,\bv\right)
\end{equation*}
and
\begin{equation*}
\ell(\bv,\bw)
\eq
\langle \BHi,\bv \times \bn_{\rm i} \rangle_{\Gi}
-
\langle \BEi,\bw \times \bn_{\rm i} \rangle_{\Gi},
\end{equation*}
for all $\be,\bv \in \BLL^2(\Omega)$ and
$(\bv,\bw) \in \BCH_0(\ccurl,\Omega) \times \BCH(\ccurl,\Omega)$.
Notice that the duality pairing in the definition of $\ell$
is well-defined since $\BEi,\BHi \in \BCH(\ccurl,\OTF)$ if we
understand $\langle\cdot,\cdot\rangle_{\Gi}$ as the duality pairing
introduced in \cite{buffa_ciarlet_2001a}.

The variational formulations associated with 
second-order forms of Maxwell's equations will also be useful.
As a result, we introduce
\begin{equation*}
\se(\be,\bv) \eq -\omega^2 (\ee\be,\bv) + (\cc\curl\be,\curl\bv)
\qquad
\forall \be,\bv \in \BCH_0(\ccurl,\Omega)
\end{equation*}
and
\begin{equation*}
\sh(\bh,\bw) \eq
-\omega^2 (\mm\bh,\bw) + (\bzet \curl \bh,\curl \bw)
\qquad
\forall \bh,\bw \in \BCH(\ccurl,\Omega),
\end{equation*}
and observe that the G\aa rding inequalities
\begin{subequations}
\label{eq_garding_inequalities}
\begin{equation}
\label{eq_garding_inequality_E}
\enormE{\be}{\Omega}^2 = \Re\se(\be,\be) + 2\omega^2\|\be\|_{\ee,\Omega}^2
\qquad \forall \be \in \BCH_0(\ccurl,\Omega)
\end{equation}
and
\begin{equation}
\label{eq_garding_inequality_H}
\enormH{\bh}{\Omega}^2 = \Re\sh(\bh,\bh) + 2\omega^2\|\bh\|_{\mm,\Omega}^2
\qquad
\forall \bh \in \BCH(\ccurl,\Omega)
\end{equation}
are satisfied.
\end{subequations}
	
\subsection{Computational mesh}
\label{section_mesh}

The computational mesh $\CT_h$ is a partition of $\Omega$ into non-overlapping
(closed) simplicial elements $K$. We denote by $\CF_h^{\rm e}$ the set of exterior
faces lying on the boundary $\partial \Omega$, by $\CF_h^{\rm i}$ the remaining
(interior) faces, and set $\CF_h \eq \CF_h^{\rm e} \cup \CF_h^{\rm i}$. We associate
with each face $F \in \CF_h^{\rm i}$ a unit normal vector $\bn_F$ whose orientation
is arbitrary but fixed. For $F \in \CF_h^{\rm e}$, $\bn_F \eq \bn$ is
the outward unit vector normal to $\partial \Omega$.
For $K \in \CT_h$, $\CF_K \subset \CF_h$ denotes the faces of $K$. The notations
\begin{equation*}
h_K \eq \sup_{\bx,\by \in K} |\bx-\by|,
\qquad
\rho_K \eq \sup \left \{ r > 0 \; | \; \exists \bx \in K: \; B(\bx,r) \subset K \right \},
\end{equation*}
stand for the diameter of $K$ and the radius of the largest ball contained in $\overline{K}$,
and $\beta_K \eq h_K/\rho_K$ is its shape-regularity parameter. The (global) mesh size, and
shape-regularity parameters are respectively defined as $h \eq \max_{K \in \CT_h} h_K$
and $\beta \eq \max_{K \in \CT_h} \beta_K$.

We assume that $\CT_h$ is conforming in the sense of \cite{ciarlet_2002a}, that is, the
intersection $\overline{K_+} \cap \overline{K_-}$ of two distinct elements
$K_\pm \in \CT_h$ is either empty, or a single vertex, edge, or face of both $K_-$ and $K_+$. 
We further require that the mesh $\CT_h$ is conforming with the physical
partition $\LP$. Namely, we assume that for each $K \in \CT_h$, there exists
$P \in \LP$ such that $K \subset P$, which ensures that coefficients $\ee$ and $\mm$
are constant in each element.

\begin{remark}[Hanging nodes]
Discontinuous Galerkin methods allow hanging nodes that
violate the above assumption and can be especially beneficial in mesh adaptivity techniques \cite{bonito_nochetto_2010a}.
We believe that the present analysis could extend to meshes featuring hanging nodes,
but at the price of increased technicalities in the definition of the quasi-interpolation
operators described in Section \ref{section_quasi_interpolation} below.
\end{remark}

We follow the standard convention for jumps and averages of functions. Namely,
if $\bv \in \BCH^1(\CT_h)$, we define
\begin{equation*}
\jmp{\bv}|_F \eq \bv_+ (\bn_+ \cdot \bn_F) + \bv_- (\bn_- \cdot \bn_F),
\qquad
\avg{\bv}|_F \eq \frac{1}{2}(\bv_+ + \bv_-),
\end{equation*}
for all interior faces $F \eq \partial K_- \cap \partial K_+ \in \CF_h^{\rm i}$ with
$\bv_\pm$ the trace of $\bv$ on $F$ from the interior of $K_{\pm}$ and $\bn_\pm$ the
unit normal pointing outward of $K_\pm$. For exterior faces $F \in \CF_h^{\rm e}$,
the definitions $\jmp{\bv}|_F \eq \avg{\bv}|_F \eq \bv|_F$ are convenient.

If $K \in \CT_h$ and $F \in \CF_h$, the associated mesh patches are defined by
\begin{equation*}
\CT_{K,h}
\eq
\left \{
K' \in \CT_h \; | \; \overline{K} \cap \overline{K'} \neq \emptyset
\right \},
\quad
\CT_{F,h}
\eq
\left \{
K' \in \CT_h \; | \;
F \subset \partial K'
\right \},
\end{equation*}
and we respectively use the notations $\tK$ and $\tF$ for the open domain covered
by the elements of $\CT_{K,h}$ and $\CT_{F,h}$.

For collections of elements and faces $\CT \subset \CT_h$ and $\CF \subset \CF_h$,
the following broken inner-product will be useful:
\begin{equation*}
(\cdot,\cdot)_\CT
\eq
\sum_{K \in \CT} (\cdot,\cdot)_K,
\quad
\langle \cdot,\cdot \rangle_{\partial \CT}
\eq
\sum_{K \in \CT} \langle \cdot,\cdot \rangle_{\partial K},
\quad
\langle \cdot,\cdot \rangle_{\CF}
\eq
\sum_{F \in \CF} \langle \cdot,\cdot \rangle_{F}.
\end{equation*}

\subsection{Polynomial spaces}

In the following, for all $K \in \CT_h$ and $q \geq 0$, $\CP_q(K)$ stands for the space of
(complex-valued) polynomials defined over $K$ and $\BCP_q(K) \eq \left (\CP_q(K)\right )^3$. 
If $\CT \subset \CT_h$ and $\bq \eq \{q_K\}_{K\in\CT}$, then $\CP_{\bq}(\CT)$ and
$\BCP_{\bq}(\CT)$ respectively stand for the space of functions that are piecewise in
$\CP_{q_K}(K)$ and $\BCP_{q_K}(K)$ for all $K \in \CT$. 

In the remaining, we associated with each element $K \in \CT_h$ a polynomial degree $p_K \geq 1$,
and we set $\bp \eq \{ p_K \}_{K \in \CT_h}$. For the sake of simplicity, we will assume
that there exists a constant $\cP > 0$ such that
\begin{equation}
\label{eq:quasiunif}
\cP^{-1} p_K\leq p_{K'}\leq \cP p_K
\end{equation}
for all neighboring elements $K$ and $K'$ in $\CT_h$. We also set $p_F \eq \max(p_K,p_{K'})$
for $F \eq \partial K_- \cap \partial K_+ \in \CF_h^{\rm i}$ and $p_F\eq p_K$
for exterior faces $F \in \CF_h^{\rm e}\cap\CF_K$. Then,
\begin{equation*}
  V_h \eq \CP_{\bp}(\CT_h) \cap H^1_0(\Omega), \qquad
\tV_h \eq \CP_{\bp}(\CT_h) \cap H^1(\Omega)
\end{equation*}
are the usual Lagrange finite element spaces (with and without essential boundary conditions) and
\begin{equation*}
	\BW_h \eq \BCP_{\bp}(\CT_h) \cap \BCH_0(\ccurl,\Omega), \qquad
	\tBW_h \eq \BCP_{\bp}(\CT_h) \cap \BCH(\ccurl,\Omega)
\end{equation*}
are the usual second-family of N\'ed\'elec spaces \cite{nedelec_1986a}.

\subsection{Quasi-interpolation}
\label{section_quasi_interpolation}

There exists two operators $\CQ_h: H^1_0(\Omega) \to V_h$ and
$\CR_h: \BCH_0(\ccurl,\Omega) \to \BW_h$ and a constant $\cI$ that
only depends on $\beta$ such that
\begin{equation}
\label{eq_quasi_interpolation_H1}
\frac{p_K}{h_K}\|w-\CQ_hw\|_K
+
\sqrt{\frac{p_K}{h_K}}\|w-\CQ_hw\|_{\partial K}
\leq
\cI
\|\grad w\|_{\tK}
\end{equation}
for all $w \in H^1_0(\Omega)$ and
\begin{equation}
\label{eq_quasi_interpolation_Hc}
\frac{p_K}{h_K}\|\bw-\CR_h\bw\|_K
+
\sqrt{\frac{p_K}{h_K}}\|(\bw-\CR_h\bw) \times \bn\|_{\partial K}
\leq
\cI
\|\grad \bw\|_{\tK}
\end{equation}
for all $\bw \in \BCH^1(\CT_h) \cap \BCH_0(\ccurl,\Omega)$.
We will also use quasi-interpolation operators that operate on spaces without essential boundary
conditions, namely, $\tCQ_h: H^1(\Omega) \to \tV_h$ and $\tCR_h: \BCH(\ccurl,\Omega) \to \tBW_h$.
These operators also satisfy \eqref{eq_quasi_interpolation_H1} and \eqref{eq_quasi_interpolation_Hc}
for all $w \in H^1(\Omega)$ and $\bw \in \BCH^1(\CT_h) \cap \BCH(\ccurl,\Omega)$.
We refer the reader to, e.g, \cite{hiptmair_pechstein_2019,melenk_2005a} for
the construction of $\CQ_h$ and $\tCQ_h$. $\CR_h$ and $\tCR_h$ are then respectively
defined by using $\CQ_h$ and $\tCQ_h$ componentwise.

\subsection{Bubble functions and inverse inequalities}
\label{bubbles}

\begin{subequations}
Bubble functions constitute a standard tool that we will use to prove efficiency
estimates \cite{dorfler_sauter_2013a,melenk_wohlmuth_2001a}. For all elements $K \in \CT_h$
and faces $F \in \CF_h$, there exists ``bubble'' functions
$b_K\in C^0(\overline{K})$ and $b_F\in C^0(\overline{F})$ such that the following holds.%
\footnote{The results in \cite{melenk_wohlmuth_2001a} are rigorously stated for the
two-dimensional case. However, as observed in \cite[Theorem 4.12]{dorfler_sauter_2013a},
these results naturally extend to the three-dimensional case.}
We have \cite[Lemma 2.5]{melenk_wohlmuth_2001a}
\begin{equation}
\label{eq_norm_bubble}
\|w\|_K \leq \cb p_K\|b_K^{1/2} w\|_K, \qquad \|v\|_F \leq \cb  p_F \|b_F^{1/2} v\|_F
\end{equation}
for all $w \in \CP_p(K)$ and $v \in \CP_p(F)$. Here, $\cb>0$ is a constant depending on the shape regularity parameter $\beta$. Besides, \cite[Lemma 2.5]{melenk_wohlmuth_2001a}
shows that
\begin{equation}
\label{eq_inv_bubble}
\|\grad(b_K w)\|_K \leq \cb \frac{p_K}{h_K}\|b_K^{1/2}w\|_K
\qquad
\forall w \in \CP_p(K).
\end{equation}
Finally, \cite[Lemma 2.6]{melenk_wohlmuth_2001a} guarantees the existence of an extension operator
$\LE: \pol_p(F) \to H^1_0(\tF)$ such that $\LE(v)|_F= b_F v$ and
\begin{equation}
	\label{eq_ext_bubble}
	p_K h_K^{-1/2}\|\LE(v)\|_{\CT_{F,h}} + p_K^{-1}h_K^{1/2}\|\grad \LE(v)\|_{\CT_{F,h}}
	\leq
	\cb \|b_F^{1/2}v\|_F \qquad \forall v \in \CP_p(F).
\end{equation}
\end{subequations}
Identical results hold for vector-valued functions,
applying the above estimates componentwise.

\subsection{Data oscillation}
\label{section_oscillation}

Our estimates include a ``data oscillation'' term which is standard in a posteriori
error estimation. Perhaps surprisingly, this oscillation term only involves the divergence
of the right-hand side $\BJ$, and not its actual values. In particular, there are
no oscillation terms in the common case where $\div \BJ = 0$ and $\BEi = \BHi = \bo$.
We thus set
\begin{align*}
\varrho_h
&\eq
\arg \min_{q_h \in \CP_{\bp}(\CT_h)} \|\div \BJ-q_h\|_\Omega,
\\
\BE^{\rm inc}_h
&\eq
\arg \min_{\bv_h \in \BCP_{\bp}(\CF_h^{\rm i})}
\|\BE^{\rm inc}-\bv_h\|_{\Gi},
\\
\BH^{\rm inc}_h
&\eq
\arg \min_{\bw_h \in \BCP_{\bp}(\CF_h^{\rm i})}
\|\BH^{\rm inc}-\bw_h\|_{\Gi},
\end{align*}
and
\begin{align*}
\osc_K
\eq
\osc_K(\BJ)+\osc_K(\BE^{\rm inc})+\osc_K(\BH^{\rm inc}),
\qquad
\osc_\CT^2 \eq \sum_{K \in \CT} \osc_K
\end{align*}
for all $K \in \CT_h$ and $\CT \subset \CT_h$, where
\begin{align*}
\osc_K(\BJ)
&\eq
p^{3/2}_K \frac{1}{\sqrt{\eps_{\tK,\min}}}
\frac{h_K}{p_K} \|\div \BJ-\varrho_h\|_K,
\\
\osc_K(\BE^{\rm inc})
&\eq
p_K^2\sqrt{\varepsilon_{\tK,\max}}\omega
\sqrt{\frac{h_K}{p_K}}\|\BE^{\rm inc}-\BE_h^{\rm inc}\|_{\partial K\cap\Gi},
\\
\osc_K(\BH^{\rm inc})
&\eq
p_K^2\sqrt{\mu_{\tK,\max}}\omega
\sqrt{\frac{h_K}{p_K}}\|\BHi-\BH_h^{\rm inc}\|_{\partial K\cap\Gi}.
\end{align*}
	
\subsection{Regular decomposition}
\label{sec_gradient_extraction}

Regular decompositions play an essential role in the derivation of reliability
estimates for a posteriori error estimators in the context of $\BCH_0(\ccurl,\Omega)$
problems
\cite{%
beck_hiptmair_hoppe_wohlmuth_2000a,%
chen_xu_zou_2011,%
demlow_hirani_2014,%
nicaise_creuse_2003a,%
schoberl_2008},
and we refer the reader to 
\cite{costabel_dauge_nicaise_1999a,girault_raviart_1986a,hiptmair_pechstein_2019}
for a thorough discussion of this topic.

\begin{subequations}
\label{eq_regularity_estimate}
The results we need follow from \cite[Theorem 2.1]{hiptmair_pechstein_2019} and read as follow:
for all $\bthe \in \BCH_0(\ccurl,\Omega)$, there exist $\bphi \in \BCH^1_0(\Omega)$ and
$r \in \CH^1_0(\Omega)$ such that $\bthe = \bphi + \grad r$ with
\begin{equation}
\|\grad \bphi\|_{\cc,\Omega}
\leq
\cR
\enormE{\bthe}{\Omega},
\qquad
\|\grad r\|_{\ee,\Omega} \leq \cR \|\bthe\|_{\ee,\Omega},
\end{equation}
and similarly, for all $\bvthe \in \BCH(\ccurl,\Omega)$ there exist
$\bpsi \in \BCH^1(\Omega)$ and $s \in \CH^1(\Omega)$ such that $\bvthe = \bpsi + \grad s$ with
\begin{equation}
\|\grad \bpsi\|_{\bzet,\Omega}
\leq
\cR
\enormH{\bvthe}{\Omega},
\qquad
\|\grad s\|_{\mm,\Omega} \leq \cR \|\bvthe\|_{\mm,\Omega}.
\end{equation}
\end{subequations}

In \eqref{eq_regularity_estimate}, $\cR$ is a constant possibly depending on the geometry
of $\Omega$, and, since the result in \cite{hiptmair_pechstein_2019} is enunciated in
non-weighted norms, the material contrasts $\eeMO/\eemO$ and $\mmMO/\mmmO$. In addition,
it may also depend on $(\omega\ell_\Omega/c_{\Omega,{\rm max}})^{-1}$ if the domain is not
simply-connected, but the constant can only blow up in the low-frequency regime and
remains well-behaved in the high-frequency regime on which we focus here (see the
discussion in \cite{chaumontfrelet_vega_2020}).

\subsection{Well-posedness}

We will work under the assumption that the (adjoint) problem under
consideration is well-posed for the fixed frequency $ \omega $.

\begin{assumption}[Well-posedness]
\label{assumption_well_posedness}
For all $\bj \in \BLL^2(\Omega)$, there exists a unique
$\be^\star(\bj) \in \BCH_0(\ccurl,\Omega)$ such that
\begin{equation}
\label{eq_def_be_bj}
\se(\bw,\be^\star(\bj))
=
\omega(\bw,\ee\bj)
\qquad
\forall \bw \in \BCH_0(\ccurl,\Omega).
\end{equation}
\end{assumption}

Classically, we can infer from Assumption \ref{assumption_well_posedness}
that the first-order variational formulation is well-posed (see, e.g.
\cite[Lemma 3.1]{chaumontfrelet_2019a}, or the proof of Lemma \ref{lemma_aubin_nitsche}
below). Similarly, for all $\bl \in \BLL^2(\Omega)$, there exists a unique element
$\bh^\star(\bl) \in \BCH(\ccurl,\Omega)$ such that
\begin{equation}
	\label{eq_def_bh_bl}
	\sh(\bw,\bh^\star(\bl)) = \omega (\bw,\mm\bl)
	\qquad
	\forall \bw \in \BCH(\ccurl.\Omega).
\end{equation}
Thus, we can introduce the notations
\begin{equation}
	\label{eq_cS}
	\cSE
	\eq
	\sup_{\substack{\bj \in \BCH(\ddiv^0,\ee,\Omega) \\ \|\bj\|_{\ee,\Omega} = 1}}
	\enormE{\be^\star(\bj)}{\Omega},
	\qquad
	\cSH
	\eq
	\sup_{\substack{\bl \in \BCH_0(\ddiv^0,\mm,\Omega) \\ \|\bl\|_{\mm,\Omega} = 1}}
	\enormH{\bh^\star(\bl)}{\Omega},
\end{equation}
and set $\cS \eq \cSE + \cSH$.

\subsection{Approximation factors}
\label{section_approximation_factor}

As is now standard for high-frequency wave propagation problems
\cite{%
chaumontfrelet_ern_vohralik_2019,%
chaumontfrelet_vega_2020,%
dorfler_sauter_2013a,%
sauter_zech_2015a},
our analysis will rely on some ``approximation factor'' that is employed in the context of a
duality argument. Because we study the problem in first-order form, the definition is slightly
different than the one proposed earlier for problems in second-order form.

\begin{subequations}
\label{eq_gba}
We start by introducing two approximation factors that respectively describe
the ability of the discrete spaces $\BW_h$ and $\tBW_h$ to approximate
solutions to \eqref{eq_def_be_bj} and \eqref{eq_def_bh_bl} and are defined by
\begin{equation}
\label{eq_gba_E}
\gbaE \eq
\sup_{\substack{\bj \in \BCH(\ddiv^0,\ee,\Omega) \\ \|\bj\|_{\ee,\Omega} = 1}}
\inf_{\be_h \in \BW_h} \enormE{\be^\star(\bj)-\be_h}{\Omega}
\end{equation}
and
\begin{equation}
\label{eq_gba_H}
\gbaH \eq
\sup_{\substack{\bl \in \BCH_0(\ddiv^0,\mm,\Omega) \\ \|\bl\|_{\mm,\Omega} = 1}}
\inf_{\bh_h \in \tBW_h} \enormH{\bh^\star(\bl)-\bh_h}{\Omega},
\end{equation}
and we set $\gba \eq \gbaE + \gbaH$.
\end{subequations}

Notice that $\gba \leq \cS$, showing that the approximation factor is controlled independently
of the mesh size $h$ and the approximation orders $\bp$. It does, however, in general, depend
on the wavenumber $\omega \ell_\Omega/c_{\Omega,\min}$, the geometry of $\Omega$, and
the coefficients $\ee$ and $\mm$. Besides, since our definition of $\gba$ immediately hinges
on the ``standard'' approximation factors $\gbaE$ and $\gbaH$, it automatically recovers their
key properties. In particular
\begin{equation}
\label{eq_gbaE_gbaH_implicit}
\gba \leq C(\omega,\Omega,\ee,\mm,s) \left (\frac{h}{p}\right )^s,
\end{equation}
for some $s \in (0,1]$, where $p \eq \min_{K \in \CT_h}p_K$. Sharper estimates
are available, and we refer the reader to
\cite{%
chaumontfrelet_nicaise_2019a,%
chaumontfrelet_vega_2021a,%
melenk_sauter_2011a,%
melenk_sauter_2021a,%
melenk_sauter_2022,%
nicaise_tomezyk_2019a},
for an in-depth discussion on the dependence of $\gba$ on $\omega$, $h$, and $p$.

\subsection{Coefficient contrasts}

For $K \in \CT_h$ and $\bphi \in \{\ee,\mm,\bzet,\cc\}$, we employ
the notation
\begin{equation*}
\LC_{{\rm c},\bphi,K}
\eq
\frac{\max_{\tK} \phi_{\rm max}}{\min_{\tK} \phi_{\rm min}}
\end{equation*}
for the ``contrast'' of the coefficient $\bphi$ in the patch $\CT_{K,h}$.
We also set $\LC_{{\rm c},\bphi} \eq \max_{K \in \CT_h} \LC_{{\rm c},\bphi,K}$
and $\LC_{{\rm c}} \eq \max_{\bphi \in \{\ee,\mm,\bzet,\cc\}} \LC_{{\rm c},\bphi}$.
Notice that these quantities are actually independent of the mesh $\CT_h$, as long as
it fits the partition $\LP$ and is only affected by the definition of the coefficients
$\ee$, $\mm$.

\subsection{Notation for generic constants}

In the remaining of this document, if $A,B \geq 0$ are two positive real values,
we employ the notation $A \lesssim B$ if there exists a constant $C$ that only
depends on $\cI$, $\cP$, $\cb$, $\cR$, and $\LC_{{\rm c}}$
such that $A \leq CB$. Importantly, $C$ is independent of $\omega$, $h$ and $p$.
However, $C$ may depend on $\Omega$ through $\cR$, and it may also depend on $\beta$
through $\cI$ and $\cb$. We also employ the notation $A \gtrsim B$ if $B \lesssim A$
and $ A \sim B $ if $A \lesssim B$ and $A \gtrsim B$.
	
\section{DG discretization and a posteriori error estimator}
\label{sec_IPDG}

\subsection{Numerical scheme}

Following
\cite{arnold_brezzi_cockburn_marini_2002a,li_lanteri_perrussel_2014a,perugia_schotzau_2003},
the discrete problem consists in finding
$(\BE_h,\BH_h) \in \BCP_{k+1}(\CT_h)\times\BCP_{k+1}(\CT_h)$
such that
\begin{equation}
\label{eq_LDG}
b_h((\BE_h,\BH_h),(\bv_h,\bw_h)) = i\omega(\BJ,\bv_h) + \ell_h(\bv_h,\bw_h)
\end{equation}
for all $\bv_h, \bw_h \in \BCP_{k+1}(\CT_h)$, where
\begin{equation*}
b_h((\BE_h,\BH_h),(\bv_h,\bw_h))
\eq
b((\BE_h,\BH_h),(\bv_h,\bw_h)) + \beta_h((\BE_h,\BH_h),(\bv_h,\bw_h)),
\end{equation*}
$\beta_h$ is a sesquilinear form over $\BCP_{k+1}(\CT_h) \times \BCP_{k+1}(\CT_h)$
that we call the ``flux'' form and $\ell_h$ is an antilinear form over
$\BCP_{k+1}(\CT_h) \times \BCP_{k+1}(\CT_h)$ designed to impose the jump conditions.
We assume that if $\be_h,\bh_h,\bv_h^\dagger,\bw_h^\dagger \in \BCP_{k+1}(\CT_h)$, then
\begin{equation}
\label{eq_betah}
\beta_h((\be_h,\bh_h),(\bv_h^\dagger,\bw_h^\dagger)) = 0,
\qquad
\ell_h(\bv_h^\dagger,\bw_h^\dagger) = \ell(\bv_h^\dagger,\bw_h^\dagger),
\end{equation}
whenever
$\bv_h^\dagger \in \BCH_0(\ccurl,\Omega)$ and $\bw_h^\dagger \in \BCH(\ccurl,\Omega)$.
Essentially, we ask for the flux form to vanish for conforming test functions.

In practice, the sesquilinear form $\beta_h(\cdot,\cdot)$ is only
employed with discrete arguments to assemble the linear system
associated with \eqref{eq_LDG}. However, in the context of abstract
mathematical analysis, it is very convenient to extend the domain of definition
of $\beta_h(\cdot,\cdot)$ and to apply it to ``continuous'' arguments
as well. To simplify the discussion, we employ the notation
$\BTU \eq \BCH_0(\ccurl,\Omega) \times \BCH(\ccurl,\Omega)$ for the
``energy'' space of ``continuous'' functions,
$\BTD_h \eq \BCP_{k+1}(\CT_h) \times \BCP_{k+1}(\CT_h)$ for the
``discrete'' space of piecewise polynomial functions, and
$\BTS_h \eq \BTU + \BTD_h$. Because of assumption \eqref{eq_betah},
we can consistently extend the domain of definition of $\beta_h(\cdot,\cdot)$
to $\BTD_h \times \BTD_h + \BTS_h \times \BTU$ by simply setting
\begin{equation}
\label{eq_betah_ext}
\beta_h((\be,\bh),(\bv^\dagger,\bw^\dagger)) = 0,
\qquad
\ell_h(\bv^\dagger,\bw^\dagger)
=
\ell(\bv^\dagger,\bw^\dagger),
\end{equation}
for all $(\be,\bh) \in \BTS_h$ and $(\bv^\dagger,\bw^\dagger) \in \BTU$.

Then, an immediate consequence of \eqref{eq_betah_ext} is that the discrete
form is consistent in the sense that
\begin{equation}
\label{consistency}
b_h((\be,\bh),(\bv,\bw))=b((\be,\bh),(\bv,\bw))
\end{equation}
for all $(\be,\bh) \in \BTS_h$, and $(\bv,\bw) \in \BTU$.
In particular, observing that $\BW_h \times \tBW_h \subset \BTU$,
the Galerkin orthogonality property
\begin{equation}
\label{eq_galerkin_orthogonality_first_order}
b_h((\BE-\BE_h,\BH-\BH_h),(\bv_h^\dagger,\bw_h^\dagger))=0
\end{equation}
holds true for all $\bv_h^\dagger \in \BW_h $ and $\bw_h^\dagger \in \tBW_h$,
assuming $(\BE_h,\BH_h)$ solves \eqref{eq_LDG}.
	
\subsection{Examples of flux form}

After formally multiplying \eqref{eq_maxwell_strong_first_order}
by tests function $i\omega \overline{\bv}_h, i\omega \overline{\bw}_h \in \BCP_{k+1}(\CT_h)$
and integrating by parts locally in each element $ K $, one obtains
the formulation
\begin{equation}
\label{tmp_DG_formulation}
b((\BE_h,\BH_h),(\bv_h,\bw_h))
-i\omega
\sum_{K \in \CT_h}
\int_{\partial K}
\BH_h \cdot \overline{\bv}_h \times \bn_K
+
i\omega
\sum_{K \in \CT_h}
\int_{\partial K}
\BE_h \cdot \overline{\bw}_h \times \bn_K
=
i\omega (\BJ,\bv_h).
\end{equation}
Obviously, \eqref{tmp_DG_formulation} is not a satisfactory discrete formulation
since no communication between separate mesh elements occurs, all the
considered functions being discontinuous. Following \cite{arnold_brezzi_cockburn_marini_2002a},
the solution consists in replacing the traces of $\BE_h$ and $\BH_h$ by numerical
fluxes $\BE_h^\star$ and $\BH_h^\star$, computed from $\BE_h$ and $\BH_h$,
leading to
\begin{equation}
\label{tmp_DG_formulation_2}
b((\BE_h,\BH_h),(\bv_h,\bw_h))
-i\omega
\sum_{K \in \CT_h}
\int_{\partial K}
\BH_h^\star \cdot \overline{\bv}_h \times \bn_K
+i\omega
\sum_{K \in \CT_h}
\int_{\partial K}
\BE_h^\star \cdot \overline{\bw}_h \times \bn_K
=
i\omega (\BJ,\bv_h).
\end{equation}
If the fluxes are single-valued on every face of the mesh,
and if $\BH_h^\star = \bo$ on $\partial \Omega$, we may rewrite
\eqref{tmp_DG_formulation_2} with face-by-face integrals as
\begin{equation*}
b((\BE_h,\BH_h),(\bv_h,\bw_h))
+
\beta_h((\BE_h,\BH_h),(\bv_h,\bw_h))
=
i\omega (\BJ,\bv_h)
\end{equation*}
with
\begin{equation}
\label{eq_betah_flux}
\beta_h((\BE_h,\BH_h),(\bv_h,\bw_h))
\eq
-i\omega\sum_{F \in \CF_h^{\rm i}}
\int_F \BH_h^\star \cdot \jmp{\overline{\bv}_h} \times \bn_F
+i\omega
\sum_{F \in \CF_h}
\int_F \BE_h^\star \cdot \jmp{\overline{\bw}_h} \times \bn_F.
\end{equation}
One readily sees that for any single-valued flux, the stabilization
form $\beta_h(\cdot,\cdot)$ of \eqref{eq_betah_flux} satisfies
\eqref{eq_betah}. As a result, the forthcoming analysis applies to
a variety of DG schemes.

In particular, a rather general family of fluxes we cover reads
\begin{equation}
	\label{eq_example_flux}
	\BE_h^\star
	\eq
	\frac{1}{\avg{Y}} \left ( \avg{Y\BE_h} + \frac{\alpha}{2} \jmp{\BH_h} \times \bn \right ),
	\quad
	\BH_h^\star
	\eq
	\frac{1}{\avg{Z}} \left ( \avg{Z\BH_h} - \frac{\alpha}{2} \jmp{\BE_h} \times \bn \right ),
\end{equation}
where $Y,Z$ are piecewise constants weights, and $0 \leq \alpha \leq 1$,
see \cite[\S3.1.3]{viquerat_2015a}. We also refer the reader to
\cite{hesthaven_warburton_2002a,li_lanteri_perrussel_2014a,perugia_schotzau_2003}.
These numerical fluxes are called centered fluxes for $\alpha = 0$
and upwind fluxes for $\alpha = 1$.

\subsection{Hybridization}

One asset of the scheme associated with any of the fluxes
defined by \eqref{eq_example_flux} is that it is ``hybridizable'', which means that it can
be equivalently rewritten with a Lagrange multiplier living on the faces of the mesh
\cite{feng_lu_xu_2016a,li_lanteri_perrussel_2014a,nguyen_peraire_cockburn_2011a}.
Such hybridized form is usually called hybrid discontinuous Galerkin (HDG)
and exhibits fewer degrees of freedom than the ``naive''
formulation \eqref{eq_LDG}. It is thus well suited to speed up the solve of the associated
linear system. Here, for the sake of simplicity, we focus on the formulation \eqref{eq_LDG},
particularly for symmetry reasons with respect to the analysis of second-order
schemes. Notice, however, that since the hybridized system is an equivalent reformulation
of \eqref{eq_LDG}, the proposed estimators apply equally well to HDG discretizations.

\subsection{Numerical solution}

It is an open question whether the discrete problem \eqref{eq_LDG}
is well-posed for general meshes. However, the following analysis applies to any pair
$(\BE_h,\BH_h) \in \BCP_{k+1}(\CT_h) \times \BCP_{k+1}(\CT_h)$
satisfying \eqref{eq_LDG}, and unique solvability is not required.

We nevertheless mention \cite{feng_lu_xu_2016a} where the authors analyze
(the hybridized version of) the method with upwind fluxes
(\eqref{eq_example_flux} with $\alpha = 1$). They focus
on a homogeneous medium enclosed by impedance boundary conditions. In this
setting, a key feature of the scheme is that it is well-posed without any assumption
on the mesh size. While we work under slightly different assumptions here, we believe
that this stability result indicates that the method is interesting for adaptivity
techniques since a coarse mesh may be used to start the algorithm.
	
\subsection{Error estimators}

For each $K \in \CT_h$, the estimator is split into four parts
\begin{equation*}
\eta_K^2
\eq
\eta_{\ddiv,\ee,K}^2 + \eta_{\ddiv,\mm,K}^2 + \eta_{\ccurl,\ee,K}^2 + \eta_{\ccurl,\mm,K}^2,
\end{equation*}
where
\begingroup
\allowdisplaybreaks
\begin{align*}
\eta_{\ddiv,\ee,K}
&\eq
\frac{1}{\sqrt{\eemK}}
\left(
\frac{h_K}{p_K} \|\div(\BJ-i\omega\ee\BE_h)\|_K
+
\omega\sqrt{\frac{h_K}{p_K}}
\|\jmp{\ee(\BE_h-\BE^{\rm inc})} \cdot \bn\|_{\partial K \setminus \partial \Omega}
\right),\\
\eta_{\ddiv,\mm,K}
&\eq
\frac{1}{\sqrt{\mmmK}}
\left(
\omega\frac{h_K}{p_K}\|\div(\mm\BH_h)\|_K
+
\omega\sqrt{\frac{h_K}{p_K}}\|\jmp{\mm(\BH_h-\BH^{\rm inc})} \cdot \bn\|_{\partial K}
\right),\\
\eta_{\ccurl,\ee,K}
&\eq
\frac{1}{\sqrt{\eeMK}}\|\BJ-i\omega \ee \BE_h+\curl\BH_h\|_K
+
\sqrt{\mmMK}\omega \sqrt{\frac{h_K}{p_K}}
\|\jmp{(\BH_h-\BHi)}\times\bn\|_{\partial K \setminus \partial \Omega},
\\
\eta_{\ccurl,\mm,K}
&\eq
\frac{1}{\sqrt{\mmMK}}\|i\omega\mm\BH_h + \curl \BE_h\|_K
+
\sqrt{\eeMK}\omega \sqrt{\frac{h_K}{p_K}}
\|\jmp{(\BE_h-\BE^{\rm inc})}\times\bn\|_{\partial K}.
\end{align*}
\endgroup
	
We also set $\eta^2 \eq \eta_{\ddiv}^2 + \eta_{\ccurl}^2$ with
\begin{align*}
\eta_{\ddiv}^2
\eq
\sum_{K \in \CT_h} \left(\eta_{\ddiv,\ee,K}^2 + \eta_{\ddiv,\mm,K}^2\right),
\qquad
\eta_{\ccurl}^2
\eq
\sum_{K \in \CT_h} \left(\eta_{\ccurl,\ee,K}^2 + \eta_{\ccurl,\mm,K}^2\right).
\end{align*}

\section{Main results}
\label{sec_main}

This section introduces our theoretical results.
	
\subsection{Preliminary result for the TF-SF formulation}
	
We start with a preliminary result concerning the jump term in the right-hand side
of the DG formulation.

\begin{lemma}[Jump form with gradient arguments]\label{lemma_jump}
The identity
\begin{equation}
\label{eq_ell_grad}
\ell(\grad q,\grad r)
=
i\varepsilon_0\omega \langle\BEi \cdot \bni,q\rangle_{\Gi}
+
i\mu_0\omega \langle\BHi \cdot \bni,r\rangle_{\Gi}
\end{equation}
holds true for all $(q,r) \in H^1_0(\Omega) \times H^1(\Omega)$.
\end{lemma}

\begin{proof}
Let $q \in H^1_0(\Omega)$. We will focus on the
$\langle \BHi,\grad q \times \bn_{\rm i}\rangle_{\Gi}$ term.
The other term in the definition of $\ell$ is treated similarly. We have
\begin{equation*}
\langle \BHi,\grad q \times \bn_{\rm i} \rangle_{\Gi}
=
(\curl \BHi,\grad q)_{\OTF}
=
i\varepsilon_0\omega (\BEi,\grad q)_{\OTF}
=
i\varepsilon_0\omega \langle \BEi \cdot \bn_{\rm i},q \rangle_{\Gi},
\end{equation*}
and \eqref{eq_ell_grad} follows: the last duality pairing
may be simplified into a usual $\BLL^2(\Gi)$ inner-product, as
$\BEi \in \BLL^2(\Gi)$ by assumption.
\end{proof}

\begin{remark}[General jump conditions]
If general jump conditions are employed (i.e., $\curl \BHi \neq i\omega \BEi$),
then a result similar to \eqref{eq_ell_grad} can still be derived if $\BEi \cdot \bni$
and $\BHi \cdot \bni$ are replaced by $\ddiv_{\Gi} \BHi$ and $\ddiv_{\Gi} \BEi$,
where $\ddiv_{\Gi}$ denotes the surface divergence along
$\Gi$. A similar modification has to be performed in the estimator as well. Although we lose
some generality, we prefer to emphasize the presentation of incident fields solution to
free-space Maxwell's equations. Indeed, they are usually employed in practice, and
the associated results are easier to understand.
\end{remark}
	
\subsection{Reliability}

We start by two lemmas where we show that the residual is controlled by the estimator.

\begin{lemma}[Control of the residual]
\label{lemma_residual_control}
The estimates
\begin{equation}
\label{eq_upper_bound_b_grad_LDG}
|b_h((\BE-\BE_h,\BH-\BH_h),(\grad p, \grad q))|
\lesssim
\omega \eta_{\ddiv}
\left( \|\grad p\|_{\ee,\Omega}+\|\grad q\|_{\mm,\Omega}\right)
\end{equation}
and
\begin{equation}
\label{eq_upper_bound_b_H1_LDG}
|b_h((\BE-\BE_h,\BH-\BH_h),(\bphi,\bpsi))|
\lesssim
\left (1 + \max_{K\in\CT_h} \frac{\omega h_K}{p_K c_{\tK,\min}} \right )
\eta_{\ccurl}
\left (
\|\grad \bphi\|_{\cc,\Omega}+\|\grad \bpsi\|_{\bzet,\Omega}
\right )
\end{equation}
hold true for all $p \in \CH^1_0(\Omega)$, $q \in \CH^1(\Omega)$,
$\bphi \in \BCH^1_0(\Omega)$ and $\bpsi \in \BCH^1(\Omega)$.
\end{lemma}
	
\begin{proof}
We first establish \eqref{eq_upper_bound_b_grad_LDG}. We observe that for any
$v \in \CH^1_0(\Omega)$, $w \in \CH^1(\Omega)$, we have
\begin{equation*}
b_h((\BE-\BE_h,\BH-\BH_h),(\grad v,\grad w))
=
i\omega(\BJ-i\omega\ee\BE_h,\grad v)-i\omega(i\omega\mm\BH_h,\grad w)
+i\omega\ell_h(\grad v,\grad w)
\end{equation*}
so that, thanks to Lemma \ref{lemma_jump}, we get
\begin{align*}
&\frac{1}{i\omega}b_h((\BE-\BE_h,\BH-\BH_h),(\grad v,\grad w))
\\
&=
-i\omega\langle \ee \BE_h\cdot \bn,v \rangle_{\partial \CT_h}
-(\div(\BJ-i\omega\ee \BE_h),v)_{\CT_h}
-i\omega\langle\mm\BH_h\cdot \bn,w\rangle_{\partial \CT_h}
+i\omega(\div(\mm \BH_h),w)_{\CT_h}\\
&\quad\ +i\varepsilon_0\omega \langle\BEi \cdot \bni,v\rangle_{\Gi}+i\mu_0\omega \langle\BHi \cdot \bni,w\rangle_{\Gi}
\\
&=
-i\omega \langle\jmp{\ee(\BE_h-\BEi)}\cdot \bni,v\rangle_{\CF_h^{\rm i}}
-(\div(\BJ-i\omega\ee \BE_h),v)_{\CT_h}
\\
&\quad
-i\omega \langle\jmp{\mm(\BH_h-\BHi)}\cdot \bni,w\rangle_{\CF_h^{\rm i}}
+i\omega(\div(\mm \BH_h),w)_{\CT_h},
\end{align*}
and therefore
\begin{align}
\label{tmp_upper_bound_b_grad_LDG}
&\frac{1}{\omega}|b_h((\BE-\BE_h,\BH-\BH_h),(\grad v, \grad w))|
\\
\nonumber
&\leq
\omega \sum_{F \in \CF_h^{\rm i}} \|\jmp{\ee(\BE_h-\BEi)}\cdot \bn\|_F\|v\|_F
+
\omega\sum_{F \in \CF_h} \|\jmp{\mm(\BH_h-\BHi)}\cdot \bn\|_F\|w\|_F
\\
\nonumber
&\quad+\sum_{K \in \CT_h}\left(\|\div(\BJ-i\omega\ee\BE_h)\|_K\|v\|_K+\omega\|\div(\mm\BH_h)\|_K\|w\|_K\right)
\\
\nonumber
&\lesssim
\sum_{K \in \CT_h}
\left(\sqrt{\varepsilon_{\widetilde K,\min}} \eta_{\ddiv,\ee,K}
\left (
\frac{p_K}{h_K}\|v\|_K
+
\sqrt{\frac{p_K}{h_K}}
\|v\|_{\partial K \setminus \partial \Omega}
\right )
\right.\\
\nonumber
&\qquad\qquad+\left.
\sqrt{\mu_{\widetilde K,\min}} \eta_{\ddiv,\mm,K}
\left (
\frac{p_K}{h_K}\|w\|_K
+
\sqrt{\frac{p_K}{h_K}}\|w\|_{\partial K}
\right )\right ).
\end{align}
Now, let $p \in \CH^1_0(\Omega)$ and $q \in \CH^1(\Omega)$.
Since $\grad (\CQ_h p)\in \BW_h$ and $\grad (\tCQ_h q) \in \tBW_h$,
by Galerkin orthogonality \eqref{eq_galerkin_orthogonality_first_order},
we can apply \eqref{tmp_upper_bound_b_grad_LDG}
with $v = p - \CQ_h p$ and $w = q - \tCQ_h q$ to show that
\begin{align*}
&
\frac{1}{\omega}|b_h((\BE-\BE_h,\BH-\BH_h),(\grad p, \grad q))|
\\
&=
|b_h((\BE-\BE_h,\BH-\BH_h),(\grad (p-\CQ_h p),\grad (q-\tCQ_h q)))|
\\
&\lesssim
\sum_{K \in \CT_h}
\left(
\sqrt{\varepsilon_{\widetilde K,\min}}
\eta_{\ddiv,\ee,K}
\left (
\frac{p_K}{h_K}\|p-\CQ_h p\|_K
+
\sqrt{\frac{p_K}{h_K}}
\|p-\CQ_h p\|_{\partial K \setminus \partial \Omega}\right)\right.
\\
&
\qquad\qquad
+\left.\sqrt{\mu_{\widetilde K,\min}} \eta_{\ddiv,\mm,K}
\left (
\frac{p_K}{h_K}\|q-\tCQ_h q\|_K
+
\sqrt{\frac{p_K}{h_K}}\|q-\tCQ_h q\|_{\partial K}
\right )
\right )
\\
&\lesssim
\sum_{K \in \CT_h}
\left(
\sqrt{\varepsilon_{\widetilde K,\min}} \eta_{\ddiv,\ee,K}\|\grad p\|_{\tK}
+\sqrt{\mu_{\widetilde K,\min}} \eta_{\ddiv,\mm,K}\|\grad q\|_{\tK}\right)
\\
&\lesssim
\eta_{\ddiv}\left( \|\grad p\|_{\ee,\Omega}+\|\grad q\|_{\mm,\Omega}\right),
\end{align*}
where we additionally employed \eqref{eq_quasi_interpolation_H1}. This shows
\eqref{eq_upper_bound_b_grad_LDG}.

We now focus on \eqref{eq_upper_bound_b_H1_LDG}.
Similarly, we start with arbitrary elements
$\bv \in \BCH^1(\CT_h) \cap \BCH_0(\ccurl,\Omega)$
and $\bw \in \BCH^1(\CT_h) \cap \BCH(\ccurl,\Omega)$. We have
\begin{align*}
&b_h((\BE-\BE_h,\BH-\BH_h),(\bv,\bw))\\
&=
(i\omega\BJ,\bv)+(\omega^2\ee\BE_h,\bv)
+
(i\omega\BH_h,\curl \bv)+(\omega^2\mm\BH_h,\bw)-(i\omega\BE_h,\curl\bw)+i\omega\ell_h(\bv,\bw)
\\
&=
(i\omega\BJ + \omega^2\ee \BE_h + i\omega\curl \BH_h,\bv)_{\CT_h}
-
i\omega\langle \BH_h \times \bn,\bv \rangle_{\partial \CT_h}+\langle \BHi,\bv \times \bn_{\rm i} \rangle_{\Gi}
\\
&\quad
-i\omega(i\omega\mm\BH_h + \curl \BE_h,\bw)
+
i\omega\langle \BE_h \times \bn,\bw \rangle_{\partial \CT_h}-\langle \BEi,\bw \times \bn_{\rm i} \rangle_{\Gi}
\\
&=
i\omega(\BJ - i\omega\ee\BE_h + \curl \BH_h,\bv)_{\CT_h}
-
i\omega\langle \jmp{\BH_h-\BHi} \times \bn,\bv \rangle_{\CF_h^{\rm i}}
\\
&\quad-i\omega(i\omega\mm\BH_h + \curl \BE_h,\bw)
+
i\omega\langle \jmp{\BE_h-\BEi} \times \bn,\bw \rangle_{\CF_h}
\end{align*}
and therefore
\begin{align}
\label{tmp_upper_bound_b_H1_LDG}
&|b_h((\BE-\BE_h,\BH-\BH_h),(\bv,\bw))|
\\
\nonumber
&\leq
\sum_{K \in \CT_h} \left(\omega\|\BJ - i\omega \ee\BE_h + \curl \BH_h\|_K\|\bv\|_K+\omega\|i\omega\mm\BH_h + \curl \BE_h	\|_K\|\bw\|_K\right)
\\
\nonumber
&\quad+
\omega \sum_{F \in \CF_h^{\rm i}} \|\jmp{\BH_h-\BHi} \times \bn\|_F\|\bv\times \bn\|_F
+
\omega \sum_{F \in \CF_h} \|\jmp{\BE_h-\BEi} \times \bn\|_F\|\bw\times \bn\|_F
\\
\nonumber
&\lesssim
\sum_{K \in \CT_h}
\left (1 +
\frac{\omega h_K}{p_K c_{\tK,\min}}
\right )
\left(\frac{1}{\sqrt{\mu_{\widetilde K,\max}}} \eta_{\ccurl,\ee,K}
\left (
\frac{p_K}{h_K}\|\bv\|_K
+
\sqrt{\frac{p_K}{h_K}}\|\bv \times \bn\|_{\partial K \setminus \partial \Omega}
\right)\right.\\
\nonumber
&\qquad\qquad\qquad\qquad\qquad\qquad
+\left.
\frac{1}{\sqrt{\varepsilon_{\widetilde K,\max}}}
\eta_{\ccurl,\mm,K}
\left (
\frac{p_K}{h_K}\|\bw\|_K
+
\sqrt{\frac{p_K}{h_K}} \|\bw \times \bn\|_{\partial K}
\right ) \right ).
\end{align}
Let now $\bphi \in \BCH^1_0(\Omega)$ and $\bpsi \in \BCH^1(\Omega)$.
Since $\CR_h \bphi \in \BW_h$ and $\tCR_h \bpsi \in \tBW_h$,
by Galerkin orthogonality \eqref{eq_galerkin_orthogonality_first_order},
we may employ \eqref{tmp_upper_bound_b_H1_LDG} with
$\bv = \bphi - \CR_h \bphi$, $\bw = \bpsi - \tCR_h \bpsi$ and
\eqref{eq_quasi_interpolation_Hc}, showing that
\begingroup
\allowdisplaybreaks
\begin{align*}
	&|b_h((\BE-\BE_h,\BH-\BH_h),(\bphi,\bpsi))|
	=
	|b_h((\BE-\BE_h,\BH-\BH_h),(\bphi-\CR_h\bphi,\bpsi-\tCR_h\bpsi))|
	\\
	&\lesssim
	\sum_{K \in \CT_h}
	\left (1 +
	\frac{\omega h_K}{p_K c_{\tK,\min}}
	\right )
	\left(
	\frac{1}{\sqrt{\mu_{\widetilde K,\max}}} \eta_{\ccurl,\ee,K}
	\left(
	\frac{p_K}{h_K}\|\bphi-\CR_h\bphi\|_K
	+
	\sqrt{\frac{p_K}{h_K}}
	\|(\bphi-\CR_h\bphi) \times \bn\|_{\partial K \setminus \partial \Omega}\right)\right.
	\\
	&\qquad\qquad\qquad\qquad\qquad\qquad
	+\left. \frac{1}{\sqrt{\varepsilon_{\widetilde K,\max}}} \eta_{\ccurl,\mm,K}\left(
	\frac{p_K}{h_K}\|\bpsi-\tCR_h\bpsi\|_K
	+
	\sqrt{\frac{p_K}{h_K}}\|(\bpsi-\tCR_h\bpsi) \times \bn\|_{\partial K}
	\right)\right)
	\\
	&\lesssim
	\sum_{K \in \CT_h}
	\left (1 +
	\frac{\omega h_K}{p_K c_{\tK,\min}}
	\right )
	\left(
	\eta_{\ccurl,\ee,K}\|\grad \bphi\|_{\cc,\tK}+\eta_{\ccurl,\mm,K}\|\grad \bpsi\|_{\bzet,\tK}\right)
	\\
	&\lesssim
	\left (
	1 +
	\max_{K\in\CT_h} \frac{\omega h_K}{p_K c_{\tK,\min}}
	\right )
	\eta_{\ccurl}
	\left(
	\|\grad \bphi\|_{\cc,\Omega}
	+
	\|\grad \bpsi\|_{\bzet,\Omega}
	\right ).
\end{align*}
\endgroup
\end{proof}
	
\begin{lemma}[General control of the residual]
We have
\begin{equation}
\label{eq_upper_bound_b_div_free_LDG}
|b_h((\BE-\BE_h,\BH-\BH_h),(\bthe,\bvthe))|
\lesssim
\left(1+\max_{K\in\CT_h}\frac{\omega h_K}{p_K c_{\tK,\min}} \right)
\eta \enorm{(\bthe,\bvthe)}
\end{equation}
for all $\bthe \in \BCH_0(\ccurl,\Omega)$ and $\bvthe \in \BCH(\ccurl,\Omega)$.
\end{lemma}

\begin{proof}
Using the results from Section \ref{sec_gradient_extraction}, we can decompose,
any $\bthe \in \BCH_0(\ccurl,\Omega)$ and $\bvthe \in \BCH(\ccurl,\Omega)$ as
$\bthe = \bphi + \grad r$ and $\bvthe = \bpsi + \grad s$ with
$(\bphi,\bpsi) \in \BCH^1_0(\Omega) \times \BCH^1(\Omega)$ and
$(r,s) \in \CH^1_0(\Omega) \times \CH^1(\Omega)$. By employing
\eqref{eq_upper_bound_b_grad_LDG} and \eqref{eq_upper_bound_b_H1_LDG}
from Lemma \ref{lemma_residual_control}, we have
\begin{align*}
|b_h((\BE-\BE_h,\BH-\BH_h),(\bthe,\bvthe))|
&\leq
|b_h((\BE-\BE_h,\BH-\BH_h),(\bphi,\bpsi))|
\\
&\quad+
|b_h((\BE-\BE_h,\BH-\BH_h),(\grad r,\grad s))|
\\
&\lesssim
\eta_{\ccurl}
\left (1 + \max_{K \in \CT_h} \frac{\omega h_K}{p_K c_{\tK,\min}}\right )
\left (
\|\grad \bphi\|_{\cc,\Omega} + \|\grad \bpsi\|_{\bzet,\Omega}
\right )
\\
&\quad+
\eta_{\ddiv} \omega
\left (
\|\grad r\|_{\ee,\Omega} + \|\grad s\|_{\mm,\Omega}
\right ),
\end{align*}
and we conclude recalling that $\eta^2 \eq \eta_{\ccurl}^2 + \eta_{\ddiv}^2$ since
\eqref{eq_regularity_estimate} imply that
\begin{equation*}
\|\grad \bphi\|_{\cc,\Omega} + \|\grad \bpsi\|_{\bzet,\Omega}
\lesssim
\enorm{(\bthe,\bvthe)},
\qquad
\omega \left (\|\grad p\|_{\ee,\Omega} + \|\grad s\|_{\mm,\Omega}\right )
\lesssim
\enorm{(\bthe,\bvthe)}.
\end{equation*}
$ $
\end{proof}
	
	The next step is an Aubin-Nitsche type result that controls the $\BLL^2(\Omega)$-norm of
	the error to make up for the lack of coercivity of the sesquilinear form $b$. To this end,
	we first state a result concerning the approximation factor for first-order schemes.
	
	\begin{lemma}[Approximation factor]
		\label{lemma_approximation_factor_first_order}
		For all $\bj,\bl \in \BLL^2(\Omega)$, there exists a unique pair
		$(\be^\star,\bh^\star)(\bj,\bl) \in \BCH_0(\ccurl,\Omega) \times \BCH(\ccurl,\Omega)$
		such that
		\begin{equation*}
			b((\bv,\bw),(\be^\star,\bh^\star)(\bj,\bl))
			=
			\omega (\bv,\ee\bj) + \omega (\bw,\mm\bl).
		\end{equation*}
		
		In addition, if $\bj \in \BCH(\ddiv^0,\ee,\Omega)$ and $\bl \in \BCH_0(\ddiv^0,\mm,\Omega)$,
		we have
		\begin{equation}
			\label{eq_approximation_factor_first_order}
			\inf_{\substack{\be_h \in \BW_h \\ \bh_h \in \tBW_h}}
			\enormEH{(\be^\star,\bh^\star)(\bj,\bl)-(\be_h,\bh_h)}{\Omega}
			\lesssim
			(1+\gba)
			\left (
			\|\bj\|_{\ee,\Omega}
			+
			\|\bl\|_{\mm,\Omega}
			\right ).
		\end{equation}
	\end{lemma}
	
	\begin{proof}
		Let $\bj,\bl \in \BLL^2(\Omega)$ and set $(\be,\bh) \eq (\be^\star,\bh^\star)(\bj,\bl)$.
		We first observe that
		\begin{equation*}
			-i\omega\curl \bh = \omega \ee \bj + \omega^2\overline{\ee} \be,
			\qquad
			i\omega\curl \be = \omega \mm \bl + \omega^2\overline{\mm} \bh.
		\end{equation*}
		Then, considering $\bphi \in \BCH(\ccurl,\Omega)$ and $\bpsi \in \BCH_0(\ccurl,\Omega)$,
		selecting the test functions $\bv \eq \bzet \curl \bphi$ and $\bw \eq \cc \curl \bpsi$,
		and integrating by parts, we show that
		\begin{equation*}
			\se(\bpsi,\be) = \omega(\bpsi,\ee\bj) +i (\overline{\mm} \cc\curl \bpsi,\bl),
			\qquad
			\sh(\bphi,\bh) = \omega(\bphi,\mm\bl) -i (\overline{\ee} \bzet\curl \bphi,\bj),
		\end{equation*}
		for all $\bpsi \in \BCH_0(\ccurl,\Omega)$ and $\bphi \in \BCH(\ccurl,\Omega)$.
		At this point, it is tempting to use the approximation factors
		$\gbaE$ and $\gbaH$. However, recalling their definition in \eqref{eq_gba},
		it is not possible yet, since the right-hand sides are not in $\BLL^2(\Omega)$. The key idea
		then consists in ``lifting'' the last term in the above identities.
		To do so, we introduce $\be_0$ and $\bh_0$ as the unique elements
		of $\BCH_0(\ccurl,\Omega)$ and $\BCH(\ccurl,\Omega)$ such that
		\begin{equation*}
			2\omega^2 (\ee\bpsi,\be_0) + \se(\bpsi,\be_0)
			=
			i (\overline{\mm} \cc \curl \bpsi,\bl),
			\qquad
			2\omega^2 (\mm\bphi,\bh_0) + \sh(\bphi,\bh_0)
			=
			-i (\overline{\ee} \bzet \curl \bphi,\bj),
		\end{equation*}
		for all $\bpsi \in \BCH_0(\ccurl,\Omega)$ and $\bphi \in \BCH(\ccurl,\Omega)$.
		As can be seen from \eqref{eq_garding_inequalities}, the left-hand sides correspond
		to coercive sesquilinear forms, and we have
		\begin{align*}
			\enormE{\be_0}{\Omega}^2
			=
			\Re \left (
			2\omega^2 (\ee\be_0,\be_0) + \se(\be_0,\be_0)
			\right )
			&=
			\Re i (\overline{\mm}\cc
			\curl \be_0,\bl)
			\\
			&\lesssim
			\|\curl \be_0\|_{\cc,\Omega}\|\bl\|_{\mm,\Omega}
			\leq
			\|\bl\|_{\mm,\Omega} \enormE{\be_0}{\Omega}.
		\end{align*}
		As a result, we have $\enormE{\be_0}{\Omega} \lesssim \|\bl\|_{\mm,\Omega}$.
		Similar arguments show that $\enormH{\bh_0}{\Omega} \lesssim \|\bj\|_{\ee,\Omega}$,
		and therefore
		\begin{equation*}
			\enorm{(\be_0,\bh_0)}
			\lesssim
			\|\bj\|_{\ee,\Omega} + \|\bl\|_{\mm,\Omega}.
		\end{equation*}
		
		On the other hand, we see that
		\begin{align*}
			\se(\bpsi,\be_0)
			=
			i (\overline{\mm} \cc \curl \bpsi,\bl) - 2\omega^2 (\ee\bpsi,\be_0),
			\qquad
			\sh(\bphi,\bh_0)
			=
			-i (\overline{\ee} \bzet \curl \bphi,\bj) - 2\omega^2 (\mm\bphi,\bh_0),
		\end{align*}
		and therefore, letting $(\tbe,\tbh) \eq (\be,\bh)-(\be_0,\bh_0)$, we have
		\begin{align*}
			\se(\bpsi,\tbe)
			&=
			\omega (\bpsi,\ee\bj) + 2\omega^2 (\ee\bpsi,\be_0)
			=
			\omega (\bpsi,\ee\tbj),
			\\
			\sh(\bphi,\tbh)
			&=
			\omega (\bphi,\mm\bl) + 2\omega^2(\mm\bphi,\bh_0)
			=                
			\omega (\bphi,\mm\tbl),
		\end{align*}
		with
		$(\tbj,\tbl) \eq (\bj,\bl) + 2\omega(\ee^{-1}\overline{\ee}\be_0,\mm^{-1} \overline{\mm}\bh_0)$.
		
		Now, we observe that picking a gradient as a test function
		in the definition of $\be_0$ and $\bh_0$ reveals that
		$\be_0 \in \BCH(\ddiv^0,\overline{\ee},\Omega)$
		and
		$\bh_0 \in \BCH_0(\ddiv^0,\overline{\mm},\Omega)$.
		Hence
		$\ee^{-1} \overline{\ee} \be_0 \in \BCH(\ddiv^0,\ee,\Omega)$
		and
		$\mm^{-1} \overline{\mm} \bh_0 \in \BCH(\ddiv^0,\mm,\Omega)$.
		As a result, we have
		\begin{align*}
			\inf_{\be_h \in \BW_h}\enormE{\tbe-\be_h}{\Omega}
			&\leq
			\gbaE \|\tbj\|_{\ee,\Omega}
			\lesssim
			\gbaE \left (\|\bj\|_{\ee,\Omega} + 2\enormE{\be_0}{\Omega}\right )\\
			&\lesssim
			\gbaE \left (\|\bj\|_{\ee,\Omega} + \|\bl\|_{\mm,\Omega} \right ),
		\end{align*}
		and similarly
		\begin{equation*}
			\inf_{\bh_h \in \tBW_h} \enormH{\tbh-\bh_h}{\Omega}
			\lesssim
			\gbaH \left (\|\bj\|_{\ee,\Omega} + \|\bl\|_{\mm,\Omega} \right ).
		\end{equation*}
		Now, \eqref{eq_approximation_factor_first_order} follows since
		\begin{align*}
			&\inf_{\substack{\be_h \in \BW_h \\ \bh_h \in \tBW_h}}
			\enorm{(\be,\bh)-(\be_h,\bh_h)}
			=
			\inf_{\substack{\be_h \in \BW_h \\ \bh_h \in \tBW_h}}
			\enorm{(\be_0,\bh_0)+(\tbe,\tbh)-(\be_h,\bh_h)}
			\\
			&\leq
			\enorm{(\be_0,\bh_0)} +
			\inf_{\substack{\be_h \in \BW_h \\ \bh_h \in \tBW_h}} \enorm{(\tbe,\tbh)-(\be_h,\bh_h)}
			\\
			&=
			\enorm{(\be_0,\bh_0)}
			+
			\inf_{\be_h \in \BW_h} \enormE{\tbe-\be_h}{\Omega}
			+
			\inf_{\bh_h \in \tBW_h} \enormH{\tbh-\bh_h}{\Omega}
			\\
			&\lesssim
			(1+\gbaE+\gbaH) \left ( \|\bj\|_{\ee,\Omega} + \|\bl\|_{\mm,\Omega} \right )
			=
			(1+\gba) \left ( \|\bj\|_{\ee,\Omega} + \|\bl\|_{\mm,\Omega} \right ).
		\end{align*}
	\end{proof}
	
	\begin{lemma}[Aubin-Nitsche]
		\label{lemma_aubin_nitsche} 
		We have
		\begin{equation*}
			\omega\|\BE-\BE_h\|_{\ee,\Omega} + \omega\|\BH-\BH_h\|_{\mm,\Omega}
			\lesssim
			\left(1+\max_{K\in\CT_h}\frac{\omega h_K}{p_K c_{\tK,\min}}\right)
			\!\left (1 + \gba\right ) \eta.
		\end{equation*}
	\end{lemma}
	
	\begin{proof}
		The proof relies on the Helmholtz decomposition of the error. We thus define
		$p \in \CH^1_0(\Omega)$ and $q \in \CH^1(\Omega)$ such that
		\begin{equation*}
			(\ee\grad p,\grad v) = (\ee(\BE-\BE_h),\grad v),
			\qquad
			(\mm\grad q,\grad w) = (\mm(\BH-\BH_h),\grad w),
		\end{equation*}
		for all $v \in \CH^1_0(\Omega)$ and $w \in \CH^1(\Omega)$. Notice that $p$
		is uniquely defined and that $q$ is defined up to a constant, that does
		not contribute to its gradient. Then, we have
		\begin{equation*}
			\BE-\BE_h = \grad p + \bthe, \qquad \BH-\BH_h = \grad q + \bvthe,
		\end{equation*}
		with $p \in\CH^1_0(\Omega)$, $q \in\CH^1(\Omega)$, $\bthe \in \BCH(\ddiv^0,\ee,\Omega)$ and
		$\bvthe \in \BCH_0(\ddiv^0,\mm,\Omega)$. 
		For the gradient terms, we have
		\begin{align*}
			\omega^2\|\grad p\|_{\ee,\Omega}^2 + \omega^2\|\grad q\|_{\mm,\Omega}^2
			&=
			\omega^2\Re(\ee(\BE-\BE_h),\grad p) + \omega^2\Re(\mm(\BH-\BH_h),\grad q)
			\\
			&=
			-\Re b_h((\BE-\BE_h,\BH-\BH_h),(\grad p,\grad q))
			\lesssim
			\omega\eta_{\ddiv}(\|\grad p\|_{\ee,\Omega} + \|\grad q\|_{\mm,\Omega}),
		\end{align*}
		so that
		\begin{equation}
			\label{tmp_aubin_nitsche_grad_p_grad_q}
			\omega\|\grad p\|_{\ee,\Omega} + \omega\|\grad q\|_{\mm,\Omega}
			\lesssim
			\eta_{\ddiv}.
		\end{equation}
		
		For the remaining terms, we observe that
		\begin{align*}
			\omega \|\bthe\|_{\ee,\Omega}^2 + \omega \|\bvthe\|_{\mm,\Omega}^2
			&=
			\Re \left (\omega(\ee\bthe,\bthe) + \omega (\mm\bvthe,\bvthe)\right )
			=
			\Re \left (
			\omega(\bthe,\ee\bthe) + \omega (\bvthe,\mm\bvthe)
			\right )
			\\
			&=
			\Re \left (
			\omega (\BE-\BE_h,\ee\bthe) + \omega (\BH-\BH_h,\mm \bvthe)
			\right ).
		\end{align*}
		Then, by Lemma \ref{lemma_approximation_factor_first_order}, we may define $(\bxi,\bzet)$
		as the unique element of $\BCH_0(\ccurl,\Omega) \times \BCH(\ccurl,\Omega)$ such that
		\begin{equation*}
			b((\bw,\bv),(\bxi,\bzet)) = \omega(\bw,\ee\bthe) + \omega(\bv,\mm\bvthe)
		\end{equation*}
		for all $\bw,\bv \in \BLL^2(\Omega)$.
		Using consistency property \eqref{consistency}, Galerkin orthogonality
		\eqref{eq_galerkin_orthogonality_first_order} and \eqref{eq_upper_bound_b_div_free_LDG},
		we have
		\begin{align*}
			\Re (\omega ((\BE-\BE_h),\ee\bthe) + \omega ((\BH-\BH_h),\mm \bvthe))
			&=
			\Re b((\BE-\BE_h,\BH-\BH_h),(\bxi,\bzet))\\
			=
			\Re b_h((\BE-\BE_h,\BH-\BH_h),(\bxi,\bzet))
			&=
			\Re b_h((\BE-\BE_h,\BH-\BH_h),(\bxi-\bxi_h,\bzet-\bzet_h))\\
			&\lesssim
			\left(1+\max_{K\in\CT_h}\frac{\omega h_K}{p_K c_{\tK,\min}} \right)
			\!\eta
			\enorm{(\bxi-\bxi_h,\bzet-\bzet_h)}
		\end{align*}
		for all $\bxi_h\in\BW_h$ and $\bzet_h\in\tBW_h$. Then, recalling
		\eqref{eq_approximation_factor_first_order},
		we deduce that
		\begin{multline*}
			\Re (\omega ((\BE-\BE_h),\ee\bthe) + \omega ((\BH-\BH_h),\mm \bvthe))
			\\
			\lesssim
			\!\left(1+\max_{K\in\CT_h}\frac{\omega h_K}{p_K c_{\tK,\min}}\right)\!
			\left (1 + \gba \right )
			\eta
			\ \!(\|\bthe\|_{\ee,\Omega}+\|\bvthe\|_{\mm,\Omega}).
		\end{multline*}
		Hence, 
		\begin{equation*}
			\omega \|\bthe\|_{\ee,\Omega}^2 + \omega \|\bvthe\|_{\mm,\Omega}^2
			\lesssim
			\left (1+\max_{K\in\CT_h}\frac{\omega h_K}{p_K c_{\tK,\min}}\right )
			\!\left (1 + \gba \right )
			\eta
			\ \!(\|\bthe\|_{\ee,\Omega}+\|\bvthe\|_{\mm,\Omega}),
		\end{equation*}
		and the result follows from \eqref{tmp_aubin_nitsche_grad_p_grad_q}.
	\end{proof}
	
	We are now ready to establish the main result of this section.
	Notice that in contrast to second-order schemes \cite{chaumontfrelet_vega_2020},
	the estimate does not stem from a continuous-level G\aa rding inequality. The
	``electric-magnetic mismatch'' part of the estimator is employed instead.
	
	\begin{theorem}[Reliability]
		\label{rel_LDG}
		The estimate
		\begin{equation}
			\label{eq_reliability_LDG}
			\enormEH{(\BE-\BE_h,\BH-\BH_h)}{\CT_h}
			\lesssim
			\left (
			1 + \max_{K\in\CT_h}\frac{\omega h_K}{p_K c_{\tK,\min}}
			\right )
			\!\left (
			1 + \gba
			\right )
			\eta
		\end{equation}
		holds true.
	\end{theorem}
	
	\begin{proof}
		We start by the observation that
		\begin{align*}
			\curl(\BE-\BE_h) &= -i\omega\mm(\BH-\BH_h)-(i\omega\mm\BH_h+\curl\BE_h),
			\\
			\curl(\BH-\BH_h) &= i\omega\ee(\BE-\BE_h)-(\BJ-i\omega\ee\BE_h+\curl\BH_h),
		\end{align*}
		which immediately yields the estimates
		\begin{align*}
			\|\curl(\BE-\BE_h)\|_{\cc,\Omega}
			&\lesssim
			\omega\|\BH-\BH_h\|_{\mm,\Omega}
			+
			\sum_{K\in\CT_h}\frac{1}{\sqrt{\mmMK}}\|i\omega\mm\BH_h+\curl\BE_h\|_K,
			\\
			\|\curl(\BH-\BH_h)\|_{\bzet,\Omega}
			&
			\lesssim
			\omega\|\BE-\BE_h\|_{\ee,\Omega}
			+
			\sum_{K\in\CT_h}\frac{1}{\sqrt{\eeMK}}\|\BJ-i\omega\ee\BE_h+\curl\BH_h\|_K.
		\end{align*}
		Adding the last two inequalities, we have
		\begin{equation*}
			\|\curl(\BE-\BE_h)\|_{\cc,\Omega}+\|\curl(\BH-\BH_h)\|_{\bzet,\Omega}
			\lesssim
			\omega\|\BE-\BE_h\|_{\ee,\Omega}+\omega\|\BH-\BH_h\|_{\mm,\Omega}+\eta_{\ccurl},
		\end{equation*}
		and \eqref{eq_reliability_LDG} follows since we already estimated
		the $\BLL^2(\Omega)$ terms in Lemma \ref{lemma_aubin_nitsche}.
	\end{proof}
	
\begin{remark}[Asymptotic estimate for smooth solutions]
\label{remark_asymptotic_exactness}
If $\ee$ and $\mm$ are real scalars, we can actually rewrite that
last line of the above proof as
\begin{equation*}
\|\curl(\BE-\BE_h)\|_{\cc,\Omega}+\|\curl(\BH-\BH_h)\|_{\bzet,\Omega}
\leq
\omega\|\BE-\BE_h\|_{\ee,\Omega}+\omega\|\BH-\BH_h\|_{\mm,\Omega}+\eta_{\ccurl},
\end{equation*}
without any hidden constant.
As a result, assuming that the solution is sufficiently
smooth (or that the mesh is locally refined) so that the convergence in $\BLL^2$ norm
happens faster than in the energy norm, we asymptotically have
\begin{equation*}
\enorm{(\BE-\BE_h,\BH-\BH_h)} \leq \eta_{\ccurl} = \eta.
\end{equation*}
This behavior is observed several times in the numerical examples reported hereafter.
\end{remark}

\subsection{Efficiency}

We now show that the estimator proposed for DG discretizations
is efficient. Classically, the proofs of this section hinge on
the ``bubble'' functions introduced in Section \ref{bubbles}.

We start by showing an upper bound for the ``divergence'' parts of the estimator,
namely $\eta_{\ddiv,\ee,K}$ and $\eta_{\ddiv,\mm,K}$.

\begin{lemma}
\label{eff_eta_div_LDG}
We have
\begin{equation*}
\eta_{\ddiv,\ee,K}
\lesssim
p_K^{3/2} \omega\|\BE-\BE_h\|_{\ee,\tK}
+
\osc_{\CT_{K,h}},
\qquad
\eta_{\ddiv,\mm,K}
\lesssim
p_K^{3/2}\omega\|\BH-\BH_h\|_{\mm,\tK}
+\osc_{\CT_{K,h}}
\end{equation*}
for all $K \in \CT_h$.
\end{lemma}

\begin{proof}
For the sake of readability, we make a slight abuse of notation in the proof, and set
$\jmp{\BEi_h}_F \eq \BEi_h$ if $F \subset \Gi$ and $\jmp{\BEi_h}_F = 0$ for the remaining faces.

Let $K \in \CT_h$ and $v_K \eq b_K\div(\BJ_h-i\omega\ee\BE_h)$. After integration by parts,
we have
\begin{align*}
\|b_K^{1/2}\div(\BJ_h-i\omega\ee\BE_h)\|_K^2&=(\div(\BJ_h-i\omega\ee\BE_h),v_K)_K
\\
&=
-(i\omega\ee(\BE-\BE_h),\grad v_K)_K-(\div(\BJ-\BJ_h),v_K)_K
\\
&\leq
\omega\|\ee(\BE-\BE_h)\|_K\|\grad v_K\|_K
+
\|b_K^{1/2}\div(\BJ-\BJ_h)\|_K\|b_K^{1/2}\div(\BJ_h-i\omega\ee\BE_h)\|_K.
\end{align*}
Using \eqref{eq_norm_bubble} and \eqref{eq_inv_bubble}, it follows that
\begin{align}\label{eta_div_K}
\frac{1}{\sqrt{\varepsilon_{\tK,\min}}}
\frac{h_K}{p_K}\|\div(\BJ-i\omega\ee\BE_h)\|_K
\lesssim
p_K\omega\|\BE-\BE_h\|_{\ee,K}
+
\frac{p_K}{\sqrt{\varepsilon_{\tK,\min}}}\frac{h_K}{p_K}\|\div(\BJ-\BJ_h)\|_K.
\end{align}

On the other hand, for $F \in \CF_h^{\rm i} \cap \CF_K$, we set
$v_F \eq \LE(\omega\jmp{\ee(\BE_h-\BEi_h)} \cdot \bn_F)$. Since
$i\omega\div(\ee\BE)=\div\BJ$ and $b_F = 0$ on $\partial \tF$, we have
\begin{align*}
\omega^2\|b_F^{1/2}\jmp{\ee(\BE_h-\BE_h^{\rm inc})} \cdot \bn_F\|_F^2
&=
\omega(\jmp{\ee(\BE_h-\BE)}\cdot\bn_F,v_F)_F
+
\omega(\jmp{\ee(\BEi-\BE_h^{\rm inc})}\cdot\bn_F,v_F)_F,
\end{align*}
and
\begin{align*}
\omega(\jmp{\ee(\BE_h-\BE)}\cdot\bn_F,v_F)_F
&=
|(\div(i\omega\ee\BE_h),v_F)_{\CT_{F,h}}+(i\omega\ee\BE_h,\grad v_F)_{\CT_{F,h}}|
\\
&=
|-(\div(\BJ-i\omega\ee\BE_h),v_F)_{\CT_{F,h}}-(i\omega\ee(\BE-\BE_h),\grad v_F)_{\CT_{F,h}}|
\\
&\leq
\|\div(\BJ-i\omega\ee\BE_h)\|_{\CT_{F,h}}\|v_F\|_{\tF}
+
\omega\|\ee(\BE-\BE_h)\|_{\tF}\|\grad v_F\|_{\tF}.
\end{align*}
Then, thanks to \eqref{eq_ext_bubble}, we bound the terms depending on $v_F$,
and using both \eqref{eq_norm_bubble} and \eqref{eq:quasiunif}, we get
\begin{align}
\label{eta_div_F}
&\frac{\omega}{\sqrt{\varepsilon_{\tK,\min}}}
\sqrt{\frac{h_K}{p_K}}
\|\jmp{\ee(\BE_h-\BEi_h)} \cdot \bn_F\|_F\\
&\lesssim
\frac{\omega}{\sqrt{\varepsilon_{\tK,\min}}}
\sqrt{\frac{h_K}{p_K}}p_F\|b_F^{1/2}\jmp{\ee(\BE_h-\BEi_h)} \cdot \bn_F\|_F
\nonumber
\\
&\lesssim
\frac{p_K^{1/2}}{\sqrt{\varepsilon_{\tK,\min}}}
\frac{h_K}{p_K}
\|\div(\BJ-i\omega\ee\BE_h)\|_{\CT_{F,h}}
+
p_K^{1/2}\omega\|\BE-\BE_h\|_{\ee,\tF}+\osc_{\CT_{F,h}}(\BEi).
\nonumber
\end{align}

The bound associated with $\eta_{\ddiv,\ee,K}$ follows from \eqref{eta_div_K} and
\eqref{eta_div_F}. For the sake of shortness, we do not write the proofs for
$\eta_{\ddiv,\mm,K}$, as it follows from the same arguments.
\end{proof}
	
We now turn to the ``rotation'' parts of the estimator, which require increased attention.
	
\begin{lemma}
\label{eff_eta_curl_LDG}
We have
\begin{subequations}
\begin{equation}
\label{eq_eff_eta_curl_e_LDG}
\eta_{\ccurl,\ee,K}
\lesssim
\omega\|\BE-\BE_h\|_{\ee,K}
+
p_K\!\left(1+\frac{\omega h_{K}}{p_K c_{\tK,\min}}\right)\!\enormH{\BH-\BH_h}{\CT_{K,h}}
+\osc_{\CT_{F,h}}(\BHi)
\end{equation}
and
\begin{equation}
\label{eq_eff_eta_curl_m_LDG}
\eta_{\ccurl,\mm,K}
\lesssim
p_K\!\left(1+\frac{\omega h_K}{p_K c_{\tK,\min}}\right)
\!\enormE{\BE-\BE_h}{\CT_{K,h}}
+
\omega\|\BH-\BH_h\|_{\mm,K}
+\osc_{\CT_{F,h}}(\BEi)
\end{equation}
\end{subequations}
for all $ K\in\CT_h $.
\end{lemma}
	
\begin{proof}
We employ the same notation for $\jmp{\BEi_h}$ as in the previous proof, and we
only detail the proof of \eqref{eq_eff_eta_curl_m_LDG} since 
\eqref{eq_eff_eta_curl_e_LDG} is established similarly, 
given the ``symmetry" of the formulation with respect to the electric and magnetic fields.
We have
\begin{equation*}
i\omega\mm\BH_h+\curl\BE_h=-i\omega\mm(\BH-\BH_h)-\mm\cc\curl(\BE-\BE_h).
\end{equation*}
As a result, it holds that
\begin{equation*}
\|i\omega\mm\BH_h+\curl\BE_h\|_K
\leq
\sqrt{\mu_{K,\max}}\omega\|\BH-\BH_h\|_{\mm,K}
+
\frac{\mu_{K,\max}}{\sqrt{\mu_{K,\min}}}\|\curl(\BE-\BE_h)\|_{\cc,K}
\end{equation*}
and
\begin{equation}
\label{tmp_eff_LDG_volume}
\frac{1}{\sqrt{\mu_{\tK,\max}}}\|i\omega\mm\BH_h+\curl\BE_h\|_K
\lesssim
\omega\|\BH-\BH_h\|_{\mm,K}+\|\curl(\BE-\BE_h)\|_{\cc,K}.
\end{equation}
		
On the other hand, for a face $F\in\CF_K$, we set
$\bw_F \eq \LE(\omega\jmp{\BE_h-\BEi_h}\times\bn_F)$.
With this notation, since
$b_F = 0$ on $\partial \tF$ and $\BE \in \BCH_0(\ccurl,\Omega)$, we have
\begin{align*}
&\omega^2\|b_F^{1/2}\jmp{\BE_h-\BEi_h}\times\bn_F\|_F^2\\
&=
|\omega(\jmp{\BE_h-\BEi}\times\bn_F,\bw_F)_F
+
\omega(\jmp{\BEi-\BEi_h}\times\bn_F,\bw_F)_F|
\\
&=
|\omega
(\BE-\BE_h,\curl\bw_F)_{\CT_{F,h}}+(\curl(\BE-\BE_h),\bw_F)_{\CT_{F,h}}
+
\omega(\jmp{\BEi-\BEi_h}\times\bn_F,\bw_F)_F|
\\
&\leq
\omega
|(\BE-\BE_h,\curl\bw_F)_{\CT_{F,h}}
+
(\curl(\BE-\BE_h),\bw_F)_{\CT_{F,h}}|
+
\osc_{\CT_{F,h}}(\BEi).
\end{align*}
Then, it follows from \eqref{eq_norm_bubble}, \eqref{eq_ext_bubble} and \eqref{eq:quasiunif} that
\begin{align}
\sqrt{\varepsilon_{\tK,\max}}\omega\sqrt{\frac{h_K}{p_K}}
\|\jmp{\BE_h-\BEi_h}\times\bn_F\|_F
&\lesssim
\sqrt{\varepsilon_{\tK,\max}}
\omega\sqrt{\frac{h_K}{p_K}}p_F\|b_F^{1/2}\jmp{\BE_h-\BEi_h}\times\bn_F\|_F\nonumber
\\
&\lesssim
p_K\!\left(1+\frac{\omega h_K}{p_K c_{\tF,\min}}\right)
\!\enormE{\BE-\BE_h}{\CT_{F,h}}
+\osc_{\CT_{F,h}}(\BEi).\label{tmp_eff_LDG_surface}
\end{align}

Finally, \eqref{eq_eff_eta_curl_m_LDG} follows from the definition of
$\eta_{\ccurl,\mm,K}$, \eqref{tmp_eff_LDG_volume} and \eqref{tmp_eff_LDG_surface}.
\end{proof}
	
Our key efficiency estimate is a direct consequence of Lemmas
\ref{eff_eta_div_LDG} and \ref{eff_eta_curl_LDG}.

\begin{theorem}[Efficiency]
\label{eff_LDG}
The estimate
\begin{equation*}
\eta_K
\lesssim
p_K^{3/2}\!\left(1+\frac{\omega h_{K}}{p_K c_{\tK,\min}}\right)
\!\enormEH{(\BE-\BE_h,\BH-\BH_h)}{\CT_{K,h}}
+
\osc_{\CT_{K,h}}
\end{equation*}
holds true for all $K \in \CT_h$.
\end{theorem}

\section{Numerical experiments}
\label{sec_numerics}

\subsection{Settings}

For the sake of simplicity, we introduce the frequency $\nu \eq \omega/2\pi$,
and we will assume that $\varepsilon_0 = \mu_0 = 1$ in all the examples below.
$\widetilde \xi$ and $\widetilde \eta$ respectively denote the relative error and estimators
that have been scaled by the norm of the reference solution, i.e.
\begin{equation*}
\widetilde \xi
\eq
\frac{\enorm{(E-E_h,\BH-\BH_h)}}{\enorm{(E,\BH)}},
\qquad
\widetilde \eta \eq \frac{\eta}{\enorm{(E,\BH)}}.
\end{equation*}

\subsubsection{Structured meshes}
\label{section_structured_mesh}

We often employ ``structured'' meshes in the following numerical examples.
By this, we mean that the domain is (up to a translation) a square $\Omega \eq (-\ell,\ell)^2$
and that it is first partitioned into $N \times N$ identical squares. Each of these squares is then
subdivided into four triangles by joining the barycenter with each face. The resulting mesh
counts $4N^2$ triangular elements, with mesh size $h = 2\ell/N$.

\subsubsection{Unstructured meshes}
\label{section_unstructured_mesh}

We also use ``unstructured'' meshes generated with a software package.
Specifically, we employ {\tt mmg2D} \cite{mmg3d} with the options
{\tt -optim}, {\tt -ar=0} and {\tt -hmax=}$h$ to generate a mesh of size $h$.

\subsubsection{Perfectly matched layers}
\label{section_pml}

Most of our experiments employ perfectly matched layers (PML)
to mimic an infinite propagation medium. This approach is standard
\cite{berenger_1994,monk_2003a}, and we proceed as follows. The computational domain will always
be (up to a translation) a square $(-\ell,\ell)^2$, and the coefficients
$\varepsilon$ and $\mm$ will take the values $1$ and $\BI$ in a neighborhood
of the square's boundary. This original square is extended into a larger square
$(-(\ell+\ell_{\rm PML}),\ell+\ell_{\rm PML})^2$ for a fixed $\ell_{\rm PML} > 0$,
and the coefficients $\varepsilon$ and $\mm$ are artificially modified outside
$(-\ell,\ell)^2$ as
\begin{equation*}
\varepsilon \eq 1/\nu_1\nu_2,
\qquad
\mm \eq \left (
\begin{array}{cc}
\nu_2/\nu_1 & 0
\\
0 & \nu_1/\nu_2,
\end{array}
\right ),
\end{equation*}
where $\nu_j(\bx) = 1+i\chi_{|\bx_j| > \ell}$ and $\chi_{|\bx_j| > \ell}$
is the characteristic function of the set $\{\bx \in \mathbb R^2; \; |\bx_j| > \ell\}$,
for $j=1$ or $2$.

\subsubsection{$h$-adaptivity}
\label{section_h_adaptivity}

In several experiments, we consider $h$-adaptive iterative refinements,
and our strategy combines D\"orfler's marking \cite{dorfler_1996a} with the newest-vertex
bisection \cite{binev_dahmen_devore_2004a}. Specifically, once the estimator
$\{\eta_K\}_{K \in \CT_h}$ has been computed, we order the elements $\{K_j\}_{j=1}^{|\CT_h|}$
in such a way that $\eta_{K_{j+1}} \leq \eta_{K_j}$, and we then select the smallest number
$n$ of elements such that
\begin{equation*}
\sum_{j=1}^n \eta_{K_{j}}^2 \geq \theta \eta^2,
\end{equation*}
where $\theta \eq 0.05$. The elements $\{K_j\}_{j=1}^n$ are then refined using
the newest-vertex bisection, starting from $K_1$ and finishing with $K_n$.

\subsubsection{$hp$-adaptivity}
\label{section_hp_adaptivity}

We will also consider $hp$-adaptive refinements. In this case, the elements are
still marked using D\"orfler marking as above, and we employ an algorithm based on
\cite{melenk_wohlmuth_2001a} to decide between $h$ and $p$ refinements. Specifically,
if an element $K$ with diameter $h_{\rm old}$, order $p_{\rm old}$ and estimator
$\eta_{\rm old}$ has been refined into new elements $\{\kappa_j\}_{j=1}^n$ with diameters
$\{h_{\kappa_j}\}_{j=1}^n$ order $\{p_{\kappa_j}\}_{j=1}^n$ and estimators
$\{\eta_{\kappa_j}\}_{j=1}^n$, we define
the ``ideal'' error reduction to be
\begin{equation*}
{\rm red}
=
\left (
\frac{h_{\rm new}}{p_{\rm new}}
\right )^{p_{\rm new}}
\left (\frac{h_{\rm old}}{p_{\rm old}}\right )^{p_{\rm old}},
\end{equation*}
where $h_{\rm new} = \max_j h_{\kappa_j}$ and $p_{\rm new} = \min_j p_{\kappa_j}$.
Then, letting
\begin{equation*}
\eta_{\rm new}^2 \eq \sum_{j=1}^{n} \eta_{\kappa_j}^2,
\end{equation*}
we perform a $p$-refinement if $\eta_{\rm new} \leq {\rm red} \eta_{\rm old}$
and an $h$-refinement otherwise. Once the new $p$-distribution has been obtained,
it is ``smoothed'' to ensure that the polynomial degree of two neighboring only varies
by one through an iterative increase in the degree of neighboring elements when
required. For the first iteration (where we have no history to compute the ideal reduction),
we only employ $p$-refinements.

\subsection{Planewave in free space}

For our first example, we consider the computational domain $\Omega_0 \eq (-1,1)^2$,
that we surround with a PML layer of thickness $\ell_{\rm PML} \eq 0.25$.
The entire domain is thus $\Omega \eq (-(1+\ell_{\rm PML}),1+\ell_{\rm PML})^2$.
We also set the TF region $\Omega_{\rm TF} \eq (-\ell_{\rm TF},\ell_{\rm TF})^2$
with $\ell_{\rm TF} \eq 0.75$. $\varepsilon = 1$ and $\mm = \BI$ in $\Omega_0$
and are modified as explained in Section \ref{section_pml} in $\Omega \setminus \Omega_0$.
We then set $J \eq 0$, and
\begin{equation*}
E_{\rm inc} \eq e^{-i\omega\bd \cdot \bx},
\qquad
\BH_{\rm inc} \eq (i\omega)^{-1} \ccurl E_{\rm inc}
\end{equation*}
where $\bd \eq (\cos \phi, \sin \phi)$ with $\phi \eq \pi/3$. The analytic solution is then
simply given by $E \eq E_{\rm inc} \chi_{\Omega_{\rm TF}}$ and
$\BH \eq \BH_{\rm inc} \chi_{\Omega_{\rm TF}}$, where $\chi_{\Omega_{\rm TF}}$ is
the characteristic function of $\Omega_{\rm TF}$. This experiment employs structured meshes
as defined in Section \ref{section_structured_mesh}. Our goal is to illustrate the behavior
of the estimator as the frequency increases.

Figures \ref{figure_PW_errors} and \ref{figure_PW_effectivities} respectively
report the errors and effectivity indices for different frequencies $\omega$
and polynomial degrees $p$. The observed results are exactly in line with our
theoretical prediction: the error is underestimated for coarse meshes, but
this effect disappears asymptotically. The asymptotic regime is achieved faster
for higher-order methods. Also, the underestimation is more pronounced
for higher frequencies. We finally note that the effectivity indices approach one
as the meshes are refined, which is expected since the solution is regular here.

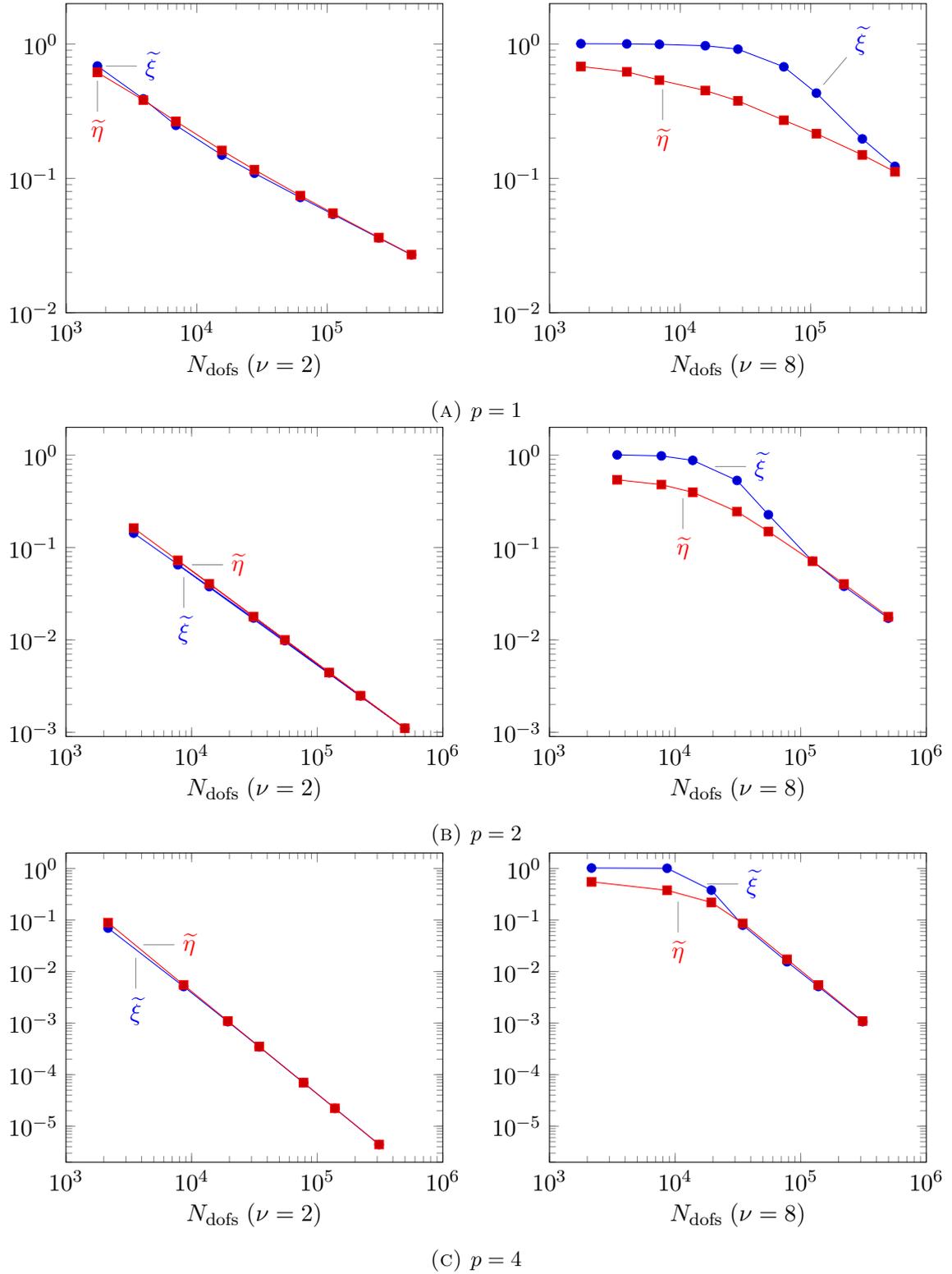
\begin{figure}
\subcaptionbox{$p=1$}{
\begin{minipage}{.49\linewidth}
\begin{tikzpicture}
\begin{axis}%
[
	width=\linewidth,
	xlabel={$N_{\rm dofs} \; (\nu=2)$},
	xmode=log,
	ymode=log,
	ymax=2,
	ymin=1.e-2
]
\addplot table[x=ndofs,y expr=\thisrow{err}/\thisrow{nor}]{figures/PW/upwind/P0/curve_0002.txt}
	node[pos=0.0,pin=0:{$\widetilde \xi$}] {};
\addplot table[x=ndofs,y expr=\thisrow{est}/\thisrow{nor}]{figures/PW/upwind/P0/curve_0002.txt}
	node[pos=0.0,pin=-90:{$\widetilde \eta$}] {};
\end{axis}
\end{tikzpicture}
\end{minipage}
\begin{minipage}{.49\linewidth}
\begin{tikzpicture}
\begin{axis}%
[
	width=\linewidth,
	xlabel={$N_{\rm dofs} \; (\nu=8)$},
	xmode=log,
	ymode=log,
	ymax=2,
	ymin=1.e-2
]
\addplot table[x=ndofs,y expr=\thisrow{err}/\thisrow{nor}]{figures/PW/upwind/P0/curve_0008.txt}
	node[pos=0.7,pin=45:{$\widetilde \xi$}] {};
\addplot table[x=ndofs,y expr=\thisrow{est}/\thisrow{nor}]{figures/PW/upwind/P0/curve_0008.txt}
	node[pos=0.25,pin=-90:{$\widetilde \eta$}] {};
\end{axis}
\end{tikzpicture}
\end{minipage}
}

\subcaptionbox{$p=2$}{
\begin{minipage}{.49\linewidth}
\begin{tikzpicture}
\begin{axis}%
[
	width=\linewidth,
	xlabel={$N_{\rm dofs} \; (\nu=2)$},
	xmode=log,
	ymode=log,
	ymax=2,
	ymin=9.e-4,
	xmin=1.e3,
	xmax=1.e6
]
\addplot table[x=ndofs,y expr=\thisrow{err}/\thisrow{nor}]{figures/PW/upwind/P1/curve_0002.txt}
	node[pos=0.1,pin=-90:{$\widetilde \xi$}] {};
\addplot table[x=ndofs,y expr=\thisrow{est}/\thisrow{nor}]{figures/PW/upwind/P1/curve_0002.txt}
	node[pos=0.1,pin=0:{$\widetilde \eta$}] {};
\end{axis}
\end{tikzpicture}
\end{minipage}
\begin{minipage}{.49\linewidth}
\begin{tikzpicture}
\begin{axis}%
[
	width=\linewidth,
	xlabel={$N_{\rm dofs} \; (\nu=8)$},
	xmode=log,
	ymode=log,
	ymax=2,
	ymin=9.e-4,
	xmin=1.e3,
	xmax=1.e6
]
\addplot table[x=ndofs,y expr=\thisrow{err}/\thisrow{nor}]{figures/PW/upwind/P1/curve_0008.txt}
	node[pos=0.25,pin=0:{$\widetilde \xi$}] {};
\addplot table[x=ndofs,y expr=\thisrow{est}/\thisrow{nor}]{figures/PW/upwind/P1/curve_0008.txt}
	node[pos=0.20,pin=-90:{$\widetilde \eta$}] {};
\end{axis}
\end{tikzpicture}
\end{minipage}
}

\subcaptionbox{$p=4$}{
\begin{minipage}{.49\linewidth}
\begin{tikzpicture}
\begin{axis}%
[
	width=\linewidth,
	xlabel={$N_{\rm dofs} \; (\nu=2)$},
	xmode=log,
	ymode=log,
	ymax=2,
	ymin=2.e-6,
	xmin=1.e3,
	xmax=1.e6
]
\addplot table[x=ndofs,y expr=\thisrow{err}/\thisrow{nor}]{figures/PW/upwind/P3/curve_0002.txt}
	node[pos=0.1,pin=-90:{$\widetilde \xi$}] {};
\addplot table[x=ndofs,y expr=\thisrow{est}/\thisrow{nor}]{figures/PW/upwind/P3/curve_0002.txt}
	node[pos=0.1,pin=0:{$\widetilde \eta$}] {};
\end{axis}
\end{tikzpicture}
\end{minipage}
\begin{minipage}{.49\linewidth}
\begin{tikzpicture}
\begin{axis}%
[
	width=\linewidth,
	xlabel={$N_{\rm dofs} \; (\nu=8)$},
	xmode=log,
	ymode=log,
	ymax=2,
	ymin=2.e-6,
	xmin=1.e3,
	xmax=1.e6
]
\addplot table[x=ndofs,y expr=\thisrow{err}/\thisrow{nor}]{figures/PW/upwind/P3/curve_0008.txt}
	node[pos=0.25,pin=0:{$\widetilde \xi$}] {};
\addplot table[x=ndofs,y expr=\thisrow{est}/\thisrow{nor}]{figures/PW/upwind/P3/curve_0008.txt}
	node[pos=0.20,pin=-90:{$\widetilde \eta$}] {};
\end{axis}
\end{tikzpicture}
\end{minipage}
}

\caption{Planewave example: error and estimator}
\label{figure_PW_errors}
\end{figure}

\begin{figure}
\begin{minipage}{.49\linewidth}
\begin{tikzpicture}
\begin{axis}%
[
	width=\linewidth,
	xlabel={$N_{\rm dofs} \; (p=1)$},
	xmode=log,
	ymax=2,
	ymin=0,
	xmin=1.e3,
	xmax=1.e6
]
\addplot table[x=ndofs,y expr=\thisrow{est}/\thisrow{err}]{figures/PW/upwind/P0/curve_0001.txt}
	node[pos=0.1,pin=90:{$\nu = 1$}] {};

\addplot table[x=ndofs,y expr=\thisrow{est}/\thisrow{err}]{figures/PW/upwind/P0/curve_0002.txt};
\addplot table[x=ndofs,y expr=\thisrow{est}/\thisrow{err}]{figures/PW/upwind/P0/curve_0004.txt};
\addplot table[x=ndofs,y expr=\thisrow{est}/\thisrow{err}]{figures/PW/upwind/P0/curve_0008.txt};

\addplot table[x=ndofs,y expr=\thisrow{est}/\thisrow{err}]{figures/PW/upwind/P0/curve_0016.txt}
	node[pos=0.3,pin=-90:{$\nu=16$}] {};
\end{axis}
\end{tikzpicture}
\end{minipage}
\begin{minipage}{.49\linewidth}
\begin{tikzpicture}
\begin{axis}%
[
	width=\linewidth,
	xlabel={$N_{\rm dofs} \; (p=2)$},
	xmode=log,
	ymax=2,
	ymin=0,
	xmin=1.e3,
	xmax=1.e6
]
\addplot table[x=ndofs,y expr=\thisrow{est}/\thisrow{err}]{figures/PW/upwind/P1/curve_0001.txt}
	node[pos=0.1,pin=90:{$\nu = 1$}] {};

\addplot table[x=ndofs,y expr=\thisrow{est}/\thisrow{err}]{figures/PW/upwind/P1/curve_0002.txt};
\addplot table[x=ndofs,y expr=\thisrow{est}/\thisrow{err}]{figures/PW/upwind/P1/curve_0004.txt};
\addplot table[x=ndofs,y expr=\thisrow{est}/\thisrow{err}]{figures/PW/upwind/P1/curve_0008.txt};

\addplot table[x=ndofs,y expr=\thisrow{est}/\thisrow{err}]{figures/PW/upwind/P1/curve_0016.txt}
	node[pos=0.9,pin=-90:{$\nu=16$}] {};
\end{axis}
\end{tikzpicture}
\end{minipage}

\begin{minipage}{.49\linewidth}
\begin{tikzpicture}
\begin{axis}%
[
	width=\linewidth,
	xlabel={$N_{\rm dofs} \; (p=3)$},
	xmode=log,
	ymax=2,
	ymin=0,
	xmin=1.e3,
	xmax=1.e6
]
\addplot table[x=ndofs,y expr=\thisrow{est}/\thisrow{err}]{figures/PW/upwind/P2/curve_0001.txt}
	node[pos=0.1,pin=90:{$\nu = 1$}] {};

\addplot table[x=ndofs,y expr=\thisrow{est}/\thisrow{err}]{figures/PW/upwind/P2/curve_0002.txt};
\addplot table[x=ndofs,y expr=\thisrow{est}/\thisrow{err}]{figures/PW/upwind/P2/curve_0004.txt};
\addplot table[x=ndofs,y expr=\thisrow{est}/\thisrow{err}]{figures/PW/upwind/P2/curve_0008.txt};

\addplot table[x=ndofs,y expr=\thisrow{est}/\thisrow{err}]{figures/PW/upwind/P2/curve_0016.txt}
	node[pos=0.85,pin=-75:{$\nu=16$}] {};
\end{axis}
\end{tikzpicture}
\end{minipage}
\begin{minipage}{.49\linewidth}
\begin{tikzpicture}
\begin{axis}%
[
	width=\linewidth,
	xlabel={$N_{\rm dofs} \; (p=4)$},
	xmode=log,
	ymax=2,
	ymin=0,
	xmin=1.e3,
	xmax=1.e6
]
\addplot table[x=ndofs,y expr=\thisrow{est}/\thisrow{err}]{figures/PW/upwind/P3/curve_0001.txt}
	node[pos=0.1,pin=90:{$\nu = 1$}] {};

\addplot table[x=ndofs,y expr=\thisrow{est}/\thisrow{err}]{figures/PW/upwind/P3/curve_0002.txt};
\addplot table[x=ndofs,y expr=\thisrow{est}/\thisrow{err}]{figures/PW/upwind/P3/curve_0004.txt};
\addplot table[x=ndofs,y expr=\thisrow{est}/\thisrow{err}]{figures/PW/upwind/P3/curve_0008.txt};

\addplot table[x=ndofs,y expr=\thisrow{est}/\thisrow{err}]{figures/PW/upwind/P3/curve_0016.txt}
	node[pos=0.8,pin=-75:{$\nu=16$}] {};
\end{axis}
\end{tikzpicture}
\end{minipage}

\begin{minipage}{.49\linewidth}
\begin{tikzpicture}
\begin{axis}%
[
	width=\linewidth,
	xlabel={$N_{\rm dofs} \; (p=5)$},
	xmode=log,
	ymax=2,
	ymin=0,
	xmin=1.e3,
	xmax=1.e6
]
\addplot table[x=ndofs,y expr=\thisrow{est}/\thisrow{err}]{figures/PW/upwind/P4/curve_0001.txt}
	node[pos=0.1,pin=90:{$\nu = 1$}] {};

\addplot table[x=ndofs,y expr=\thisrow{est}/\thisrow{err}]{figures/PW/upwind/P4/curve_0002.txt};
\addplot table[x=ndofs,y expr=\thisrow{est}/\thisrow{err}]{figures/PW/upwind/P4/curve_0004.txt};
\addplot table[x=ndofs,y expr=\thisrow{est}/\thisrow{err}]{figures/PW/upwind/P4/curve_0008.txt};

\addplot table[x=ndofs,y expr=\thisrow{est}/\thisrow{err}]{figures/PW/upwind/P4/curve_0016.txt}
	node[pos=0.8,pin=-45:{$\nu=16$}] {};
\end{axis}
\end{tikzpicture}
\end{minipage}
\begin{minipage}{.49\linewidth}
\begin{tikzpicture}
\begin{axis}%
[
	width=\linewidth,
	xlabel={$N_{\rm dofs} \; (p=6)$},
	xmode=log,
	ymax=2,
	ymin=0,
	xmin=1.e3,
	xmax=1.e6
]
\addplot table[x=ndofs,y expr=\thisrow{est}/\thisrow{err}]{figures/PW/upwind/P5/curve_0001.txt}
	node[pos=0.1,pin=90:{$\nu = 1$}] {};

\addplot table[x=ndofs,y expr=\thisrow{est}/\thisrow{err}]{figures/PW/upwind/P5/curve_0002.txt};
\addplot table[x=ndofs,y expr=\thisrow{est}/\thisrow{err}]{figures/PW/upwind/P5/curve_0004.txt};
\addplot table[x=ndofs,y expr=\thisrow{est}/\thisrow{err}]{figures/PW/upwind/P5/curve_0008.txt};

\addplot table[x=ndofs,y expr=\thisrow{est}/\thisrow{err}]{figures/PW/upwind/P5/curve_0016.txt}
	node[pos=0.7,pin=-45:{$\nu=16$}] {};
\end{axis}
\end{tikzpicture}
\end{minipage}

\caption{Planewave example: efficiencies}
\label{figure_PW_effectivities}
\end{figure}
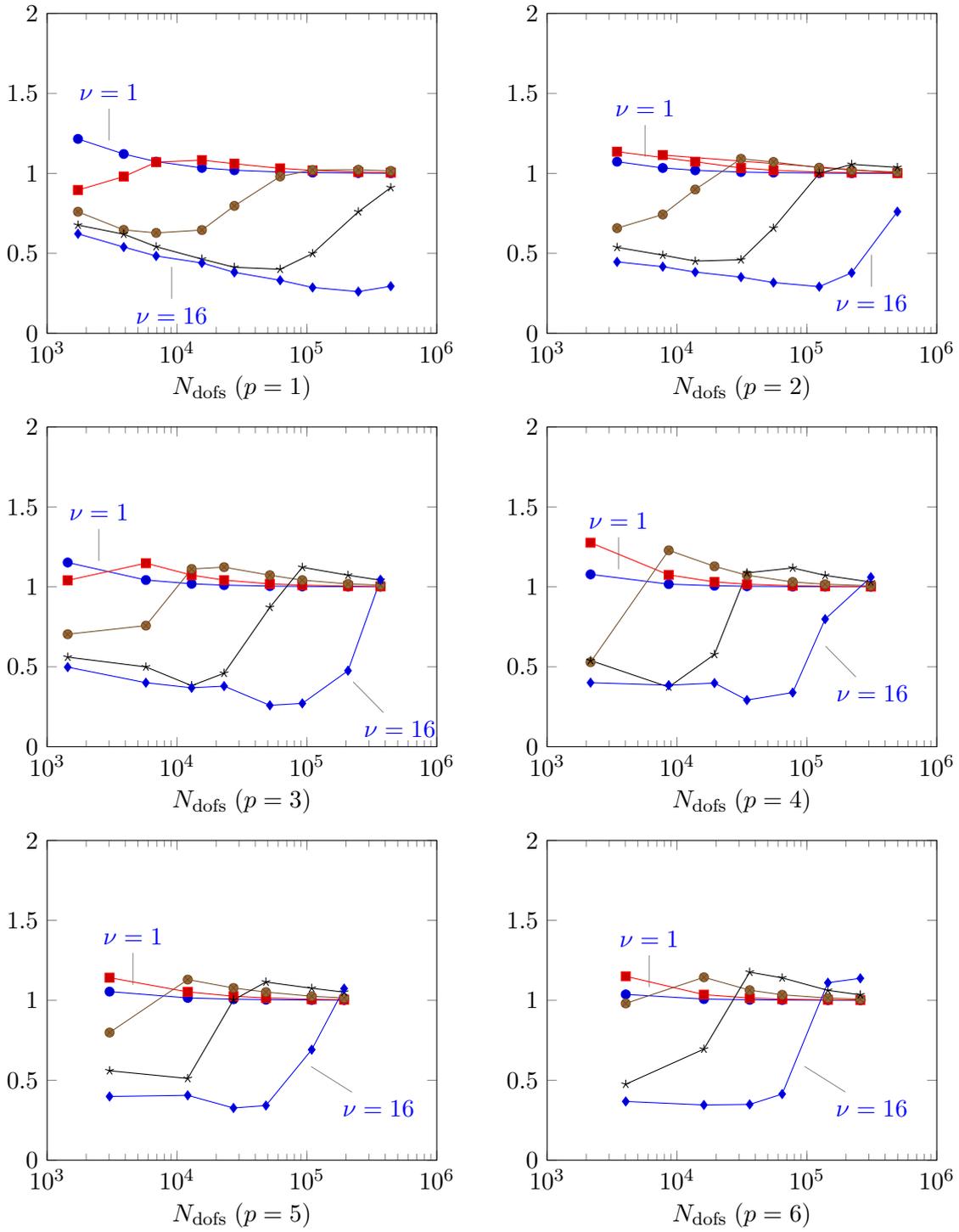

\subsection{Nearly resonant frequencies}

In this example, we consider the unit square $\Omega \eq (0,1)^2$ surrounded by PEC
boundary conditions (i.e., we do not use PMLs). The source term is $J \eq 1$, and
no incident fields are injected. The (semi-)analytical solution
is obtained via Fourier series as 
\begin{equation*}
E(\bx)
\eq
(16i\omega)
\sum_{n,m \text{ odd}}
\frac{1}{nm\pi^2}\frac{1}{(n^2+m^2)\pi^2 - \omega^2}\sin(n\pi \bx_1)\sin(m\pi \bx_2)
\end{equation*}
which we cut at $n,m \leq 500$. The magnetic field is obtained by (analytically) differentiating
$E$. Notice that the solution belongs to $H^3(\Omega)$, but not to $H^4(\Omega)$.

The first goal of this example is to highlight the behavior of the estimator when
approaching a resonance frequency. We focus on the resonance frequency
$\nu_{\rm r} \eq \sqrt{2}/2$, and we consider a sequence of frequencies
\begin{equation*}
\nu_\delta \eq (1+\delta) \frac{\sqrt{2}}{2} = \nu_{\rm r} + \frac{\sqrt{2}}{2}\delta,
\end{equation*}
for decreasing values of $\delta \in \{4^{-r}\}_{r=1}^5$. In contrast to the previous example,
we employ unstructured meshes here. Figure \ref{figure_res_errors} presents the behavior
of the error and the estimator for different values of $\delta$ and $p$ as the mesh is
refined, whereas effectivity indices are given in Figure \ref{figure_res_effectivities}.
The behavior is similar to the one observed in Figures \ref{figure_PW_errors} and
\ref{figure_PW_effectivities} and conforms to our theoretical prediction. Indeed, the
error is underestimated pre-asymptotically, and this effect is amplified when nearing
$\nu_{\rm r}$. The asymptotic range is achieved faster for higher polynomial degrees.
In Figure \ref{figure_res_errors}, we observe the optimal convergence rates for uniform
meshes, namely $N_{\rm dofs}^{\min(p,2)/2}$, given the finite regularity of the solution.
Notice also that the ``suboptimal'' convergence rates for $p=4$ are only seen ``late'' in
the convergence curves, which is in agreement with the regularity splitting results of
\cite{chaumontfrelet_nicaise_2018a}.
We also observed in Figure \ref{figure_res_effectivities} that the effectivity indices
asymptotically approach one for $p=1$ and $p=2$, which is coherent with Remark
\ref{remark_asymptotic_exactness} since $E \in H^3(\Omega)$ but $E \notin H^4(\Omega)$.

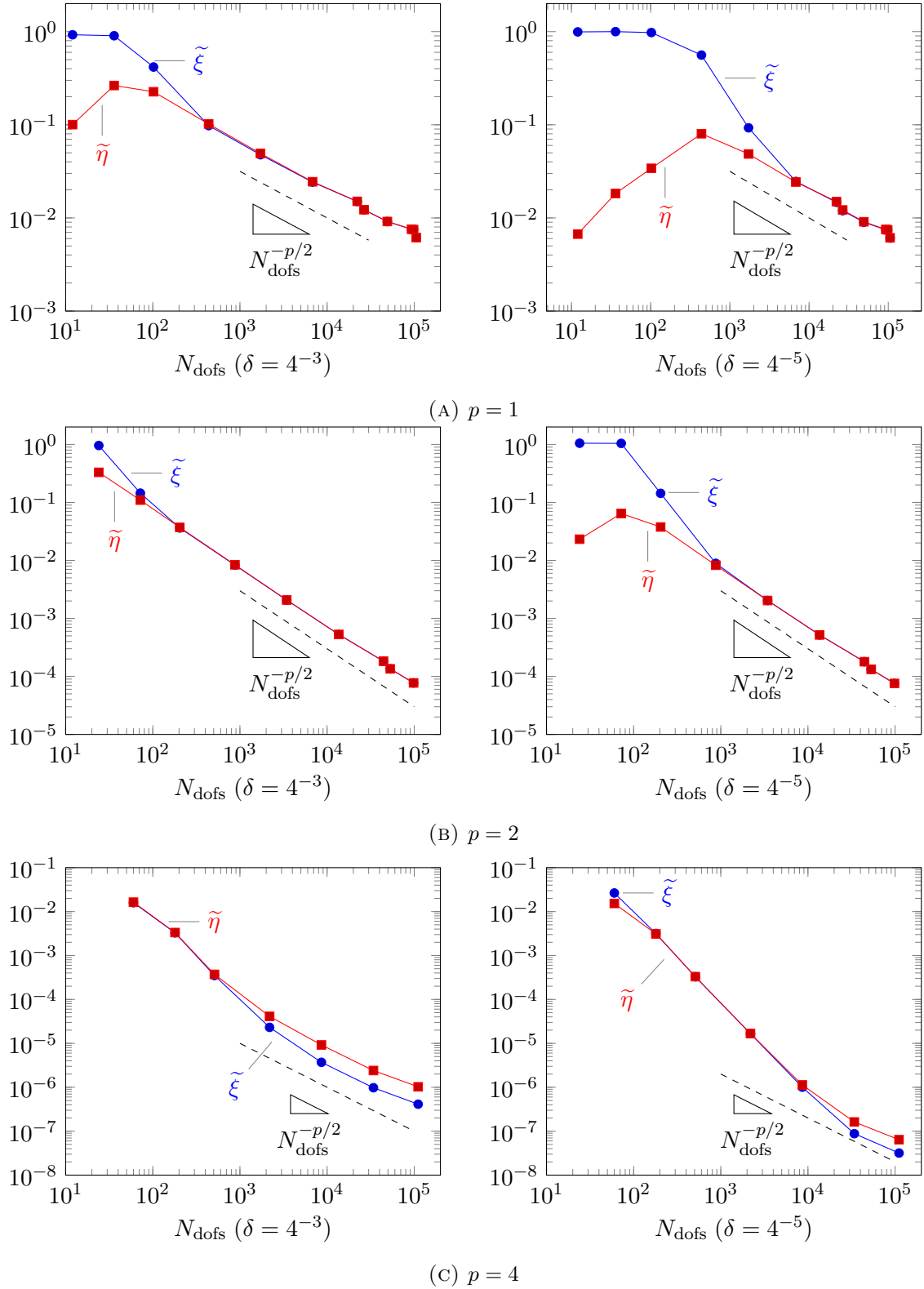
\begin{figure}
\subcaptionbox{$p=1$}{
\begin{minipage}{.49\linewidth}
\begin{tikzpicture}
\begin{axis}%
[
	width=\linewidth,
	xlabel={$N_{\rm dofs} \; (\delta=4^{-3})$},
	xmode=log,
	ymode=log,
	ymax=2,
	ymin=1.e-3,
	xmin=10,
	xmax=2.e5
]
\addplot table[x=ndofs,y expr=\thisrow{err}/52.2918]{figures/res/upwind/P0/curve_0064.txt}
	node[pos=0.2,pin=0:{$\widetilde \xi$}] {};
\addplot table[x=ndofs,y expr=\thisrow{est}/52.2918]{figures/res/upwind/P0/curve_0064.txt}
	node[pos=0.1,pin=-90:{$\widetilde \eta$}] {};

\addplot[black,dashed,domain=1e3:3.e4] {x^(-.5)};
\SlopeTriangle{.5}{-.15}{.25}{-.5}{$N_{\rm dofs}^{-p/2}$}{}

\end{axis}
\end{tikzpicture}
\end{minipage}
\begin{minipage}{.49\linewidth}
\begin{tikzpicture}
\begin{axis}%
[
	width=\linewidth,
	xlabel={$N_{\rm dofs} \; (\delta=4^{-5})$},
	xmode=log,
	ymode=log,
	ymax=2,
	ymin=1.e-3
]
\addplot table[x=ndofs,y expr=\thisrow{err}/830.808]{figures/res/upwind/P0/curve_1024.txt}
	node[pos=0.4,pin=0:{$\widetilde \xi$}] {};
\addplot table[x=ndofs,y expr=\thisrow{est}/830.808]{figures/res/upwind/P0/curve_1024.txt}
	node[pos=0.3,pin=-90:{$\widetilde \eta$}] {};

\addplot[black,dashed,domain=1e3:3.e4] {x^(-.5)};
\SlopeTriangle{.5}{-.15}{.25}{-.5}{$N_{\rm dofs}^{-p/2}$}{}

\end{axis}
\end{tikzpicture}
\end{minipage}
}

\subcaptionbox{$p=2$}{
\begin{minipage}{.49\linewidth}
\begin{tikzpicture}
\begin{axis}%
[
	width=\linewidth,
	xlabel={$N_{\rm dofs} \; (\delta=4^{-3})$},
	xmode=log,
	ymode=log,
	ymax=2,
	ymin=1.e-5,
	xmin=10,
	xmax=2.e5
]
\addplot table[x=ndofs,y expr=\thisrow{err}/52.2918]{figures/res/upwind/P1/curve_0064.txt}
	node[pos=0.1,pin=0:{$\widetilde \xi$}] {};
\addplot table[x=ndofs,y expr=\thisrow{est}/52.2918]{figures/res/upwind/P1/curve_0064.txt}
	node[pos=0.05,pin=-90:{$\widetilde \eta$}] {};

\addplot[black,dashed,domain=1e3:1.e5] {3*x^(-1)};
\SlopeTriangle{.5}{-.15}{.25}{-1}{$N_{\rm dofs}^{-p/2}$}{}

\end{axis}
\end{tikzpicture}
\end{minipage}
\begin{minipage}{.49\linewidth}
\begin{tikzpicture}
\begin{axis}%
[
	width=\linewidth,
	xlabel={$N_{\rm dofs} \; (\delta=4^{-5})$},
	xmode=log,
	ymode=log,
	ymax=2,
	ymin=1.e-5,
	xmin=10,
	xmax=2.e5
]
\addplot table[x=ndofs,y expr=\thisrow{err}/830.808]{figures/res/upwind/P1/curve_1024.txt}
	node[pos=0.25,pin=0:{$\widetilde \xi$}] {};
\addplot table[x=ndofs,y expr=\thisrow{est}/830.808]{figures/res/upwind/P1/curve_1024.txt}
	node[pos=0.20,pin=-90:{$\widetilde \eta$}] {};

\addplot[black,dashed,domain=1e3:1.e5] {3*x^(-1)};
\SlopeTriangle{.5}{-.15}{.25}{-1}{$N_{\rm dofs}^{-p/2}$}{}

\end{axis}
\end{tikzpicture}
\end{minipage}
}

\subcaptionbox{$p=4$}{
\begin{minipage}{.49\linewidth}
\begin{tikzpicture}
\begin{axis}%
[
	width=\linewidth,
	xlabel={$N_{\rm dofs} \; (\delta=4^{-3})$},
	xmode=log,
	ymode=log,
	ymax=1e-1,
	ymin=1.e-8,
	xmin=10.,
	xmax=2.e5
]
\addplot table[x=ndofs,y expr=\thisrow{err}/52.2918]{figures/res/upwind/P3/curve_0064.txt}
	node[pos=0.6,pin=-135:{$\widetilde \xi$}] {};
\addplot table[x=ndofs,y expr=\thisrow{est}/52.2918]{figures/res/upwind/P3/curve_0064.txt}
	node[pos=0.1,pin=0:{$\widetilde \eta$}] {};

\addplot[black,dashed,domain=1e3:1.e5] {.01*x^(-1)};
\SlopeTriangle{.6}{-.1}{.2}{-1}{$N_{\rm dofs}^{-p/2}$}{}

\end{axis}
\end{tikzpicture}
\end{minipage}
\begin{minipage}{.49\linewidth}
\begin{tikzpicture}
\begin{axis}%
[
	width=\linewidth,
	xlabel={$N_{\rm dofs} \; (\delta=4^{-5})$},
	xmode=log,
	ymode=log,
	ymax=1e-1,
	ymin=1.e-8,
	xmin=10.,
	xmax=2.e5
]
\addplot table[x=ndofs,y expr=\thisrow{err}/830.808]{figures/res/upwind/P3/curve_1024.txt}
	node[pos=0.0,pin=0:{$\widetilde \xi$}] {};
\addplot table[x=ndofs,y expr=\thisrow{est}/830.808]{figures/res/upwind/P3/curve_1024.txt}
	node[pos=0.20,pin=-135:{$\widetilde \eta$}] {};

\addplot[black,dashed,domain=1e3:1.e5] {.002*x^(-1)};
\SlopeTriangle{.5}{-.1}{.2}{-1}{$N_{\rm dofs}^{-p/2}$}{}

\end{axis}
\end{tikzpicture}
\end{minipage}
}

\caption{Nearly resonant example: error and estimator}
\label{figure_res_errors}
\end{figure}

\begin{figure}
\begin{minipage}{.49\linewidth}
\begin{tikzpicture}
\begin{axis}%
[
	width=\linewidth,
	xlabel={$N_{\rm dofs} \; (p=1)$},
	xmode=log,
	ymode=log,
	ymax=10,
	ymin=5.e-3,
	xmin=0,
	xmax=2.e5
]
\addplot table[x=ndofs,y expr=\thisrow{est}/\thisrow{err}]{figures/res/upwind/P0/curve_0004.txt}
	node[pos=0.1,pin={[pin distance=.3cm]90:{$\delta = 4^{-1}$}}] {};

\addplot table[x=ndofs,y expr=\thisrow{est}/\thisrow{err}]{figures/res/upwind/P0/curve_0016.txt};
\addplot table[x=ndofs,y expr=\thisrow{est}/\thisrow{err}]{figures/res/upwind/P0/curve_0064.txt};
\addplot table[x=ndofs,y expr=\thisrow{est}/\thisrow{err}]{figures/res/upwind/P0/curve_0256.txt};

\addplot table[x=ndofs,y expr=\thisrow{est}/\thisrow{err}]{figures/res/upwind/P0/curve_1024.txt}
	node[pos=0.5,pin={0:{$\delta = 4^{-5}$}}] {};
\end{axis}
\end{tikzpicture}
\end{minipage}
\begin{minipage}{.49\linewidth}
\begin{tikzpicture}
\begin{axis}%
[
	width=\linewidth,
	xlabel={$N_{\rm dofs} \; (p=2)$},
	xmode=log,
	ymode=log,
	ymax=10,
	ymin=5.e-3,
	xmin=0.,
	xmax=2.e5
]
\addplot table[x=ndofs,y expr=\thisrow{est}/\thisrow{err}]{figures/res/upwind/P1/curve_0004.txt}
	node[pos=0.1,pin={[pin distance=.3cm]90:{$\delta = 4^{-1}$}}] {};

\addplot table[x=ndofs,y expr=\thisrow{est}/\thisrow{err}]{figures/res/upwind/P1/curve_0016.txt};
\addplot table[x=ndofs,y expr=\thisrow{est}/\thisrow{err}]{figures/res/upwind/P1/curve_0064.txt};
\addplot table[x=ndofs,y expr=\thisrow{est}/\thisrow{err}]{figures/res/upwind/P1/curve_0256.txt};

\addplot table[x=ndofs,y expr=\thisrow{est}/\thisrow{err}]{figures/res/upwind/P1/curve_1024.txt}
	node[pos=0.3,pin=0:{$\delta = 4^{-5}$}] {};
\end{axis}
\end{tikzpicture}
\end{minipage}

\begin{minipage}{.49\linewidth}
\begin{tikzpicture}
\begin{axis}%
[
	width=\linewidth,
	xlabel={$N_{\rm dofs} \; (p=3)$},
	xmode=log,
	ymode=log,
	ymax=10,
	ymin=5.e-3,
	xmin=0.,
	xmax=2.e5
]
\addplot table[x=ndofs,y expr=\thisrow{est}/\thisrow{err}]{figures/res/upwind/P2/curve_0004.txt}
	node[pos=0.1,pin={[pin distance=.3cm]90:{$\delta = 4^{-1}$}}] {};

\addplot table[x=ndofs,y expr=\thisrow{est}/\thisrow{err}]{figures/res/upwind/P2/curve_0016.txt};
\addplot table[x=ndofs,y expr=\thisrow{est}/\thisrow{err}]{figures/res/upwind/P2/curve_0064.txt};
\addplot table[x=ndofs,y expr=\thisrow{est}/\thisrow{err}]{figures/res/upwind/P2/curve_0256.txt};

\addplot table[x=ndofs,y expr=\thisrow{est}/\thisrow{err}]{figures/res/upwind/P2/curve_1024.txt}
	node[pos=0.1,pin=-45:{$\delta=4^{-5}$}] {};
\end{axis}
\end{tikzpicture}
\end{minipage}
\begin{minipage}{.49\linewidth}
\begin{tikzpicture}
\begin{axis}%
[
	width=\linewidth,
	xlabel={$N_{\rm dofs} \; (p=4)$},
	xmode=log,
	ymode=log,
	ymax=10,
	ymin=5e-3,
	xmin=0.,
	xmax=2.e5
]
\addplot table[x=ndofs,y expr=\thisrow{est}/\thisrow{err}]{figures/res/upwind/P3/curve_0004.txt}
	node[pos=0.1,pin=90:{$\delta = 4^{-1}$}] {};

\addplot table[x=ndofs,y expr=\thisrow{est}/\thisrow{err}]{figures/res/upwind/P3/curve_0016.txt};
\addplot table[x=ndofs,y expr=\thisrow{est}/\thisrow{err}]{figures/res/upwind/P3/curve_0064.txt};
\addplot table[x=ndofs,y expr=\thisrow{est}/\thisrow{err}]{figures/res/upwind/P3/curve_0256.txt};

\addplot table[x=ndofs,y expr=\thisrow{est}/\thisrow{err}]{figures/res/upwind/P3/curve_1024.txt}
	node[pos=0.0,pin=-45:{$\delta=4^{-5}$}] {};
\end{axis}
\end{tikzpicture}
\end{minipage}

\begin{minipage}{.49\linewidth}
\begin{tikzpicture}
\begin{axis}%
[
	width=\linewidth,
	xlabel={$N_{\rm dofs} \; (p=5)$},
	xmode=log,
	ymode=log,
	ymax=10,
	ymin=5.e-3,
	xmin=0.,
	xmax=2.e5
]
\addplot table[x=ndofs,y expr=\thisrow{est}/\thisrow{err}]{figures/res/upwind/P4/curve_0004.txt}
	node[pos=0.1,pin={[pin distance=4]90:{$\delta = 4^{-1}$}}] {};

\addplot table[x=ndofs,y expr=\thisrow{est}/\thisrow{err}]{figures/res/upwind/P4/curve_0016.txt};
\addplot table[x=ndofs,y expr=\thisrow{est}/\thisrow{err}]{figures/res/upwind/P4/curve_0064.txt};
\addplot table[x=ndofs,y expr=\thisrow{est}/\thisrow{err}]{figures/res/upwind/P4/curve_0256.txt};

\addplot table[x=ndofs,y expr=\thisrow{est}/\thisrow{err}]{figures/res/upwind/P4/curve_1024.txt}
	node[pos=0.4,pin=-90:{$\delta=4^{-5}$}] {};
\end{axis}
\end{tikzpicture}
\end{minipage}
\begin{minipage}{.49\linewidth}
\begin{tikzpicture}
\begin{axis}%
[
	width=\linewidth,
	xlabel={$N_{\rm dofs} \; (p=6)$},
	xmode=log,
	ymode=log,
	ymax=10,
	ymin=5.e-3,
	xmin=0,
	xmax=2.e5
]
\addplot table[x=ndofs,y expr=\thisrow{est}/\thisrow{err}]{figures/res/upwind/P5/curve_0004.txt}
	node[pos=0.0,pin={[pin distance=2.8cm]0:{$\delta = 4^{-1}$}}] {};

\addplot table[x=ndofs,y expr=\thisrow{est}/\thisrow{err}]{figures/res/upwind/P5/curve_0016.txt};
\addplot table[x=ndofs,y expr=\thisrow{est}/\thisrow{err}]{figures/res/upwind/P5/curve_0064.txt};
\addplot table[x=ndofs,y expr=\thisrow{est}/\thisrow{err}]{figures/res/upwind/P5/curve_0256.txt};

\addplot table[x=ndofs,y expr=\thisrow{est}/\thisrow{err}]{figures/res/upwind/P5/curve_1024.txt}
	node[pos=0.2,pin=-45:{$\delta=4^{-5}$}] {};
\end{axis}
\end{tikzpicture}
\end{minipage}

\caption{Nearly resonant example: efficiencies with upwind fluxes}
\label{figure_res_effectivities}
\end{figure}
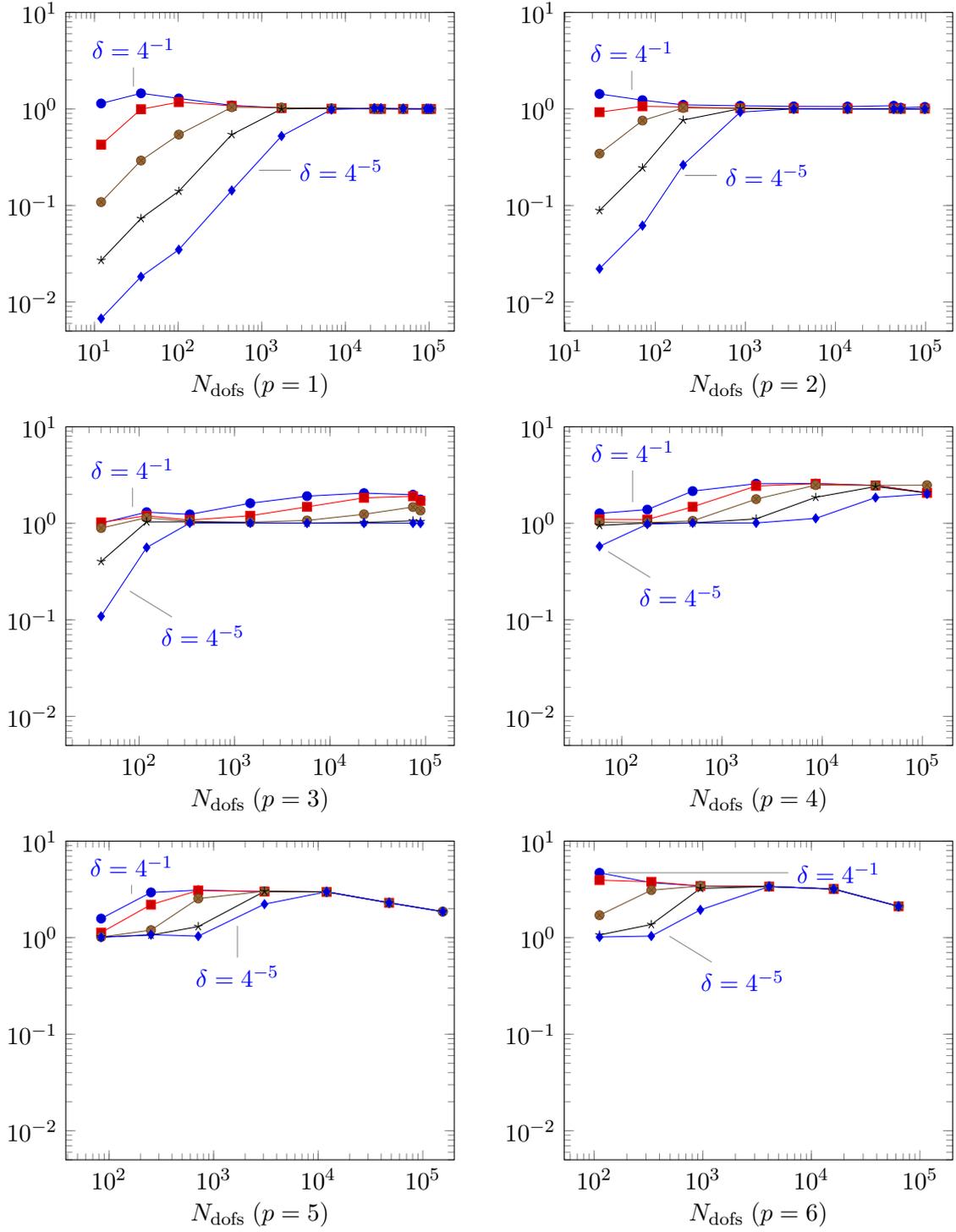

The next purpose of the present test is to analyze the ability of the estimator to drive
$h$-adaptive refinements. Results for $p=3$ and $4$ are presented for the frequencies
$\nu \eq 1.1$ and $5.1$ (respectively close to the resonance frequencies
$\nu_{\rm r} \eq \sqrt{1^2+2^2}/2 \sim 1.118$ and $\sqrt{6^2+8^2}/2 = 5$).
We start with a structured mesh for $p=3$ and an unstructured mesh for $p=4$.
Figure \ref{figure_res_hada_errors} shows the behavior of the error and estimators where our
previous comments on the pre-asymptotic underestimation apply. We further see that
in contrast with the uniform meshes, the optimal convergence rates are observed despite
the finite regularity of the solution, which is in particular due to the refinements close
to the domain's corners that can be seen in Figure \ref{figure_res_hada_img_nu11} when
$\nu \eq 1.1$. Notice that for the higher frequency $\nu \eq 5.1$, we do not observe
local refinements in Figure \ref{figure_res_hada_img_nu51}, which is again in agreement
with \cite{chaumontfrelet_nicaise_2018a} and the comment we made on Figure
\ref{figure_res_errors}. Finally, Figure \ref{figure_res_hada_effectivities} presents the
effectivity indices.  The usual pre-asymptotic underestimation is observed for the higher
frequency. Besides, the estimator seems asymptotic exact, which is in agreement with Remark
\ref{remark_asymptotic_exactness} since local refinements are performed here. We finally
note that Figures \ref{figure_res_hada_img_nu11} and \ref{figure_res_hada_img_nu51}
show an excellent agreement between the elementwise error and estimator.

\begin{figure}
\subcaptionbox{$p=3$}{
\begin{minipage}{.45\linewidth}
\begin{tikzpicture}
\begin{axis}[%
	width=\linewidth,
	xlabel={$N_{\rm dofs} \; (\nu=1.1)$},
	xmode=log,
	ymode=log,
	ymax=10,
	ymin=1.e-6
]

\addplot table[x=ndofs,y=err] {figures/res/upwind/hada/P2_F1.1.txt}
	node[pos=0.2,pin=-90:{$\widetilde \xi$}] {};
\addplot table[x=ndofs,y=est] {figures/res/upwind/hada/P2_F1.1.txt}
	node[pos=0.2,pin= 90:{$\widetilde \eta$}] {};

\addplot[black,dashed,domain=1e3:3.e4] {75*x^(-1.5)};
\SlopeTriangle{.5}{-.15}{.25}{-1.5}{$N_{\rm dofs}^{-p/2}$}{}

\end{axis}
\end{tikzpicture}
\end{minipage}
\begin{minipage}{.45\linewidth}
\begin{tikzpicture}
\begin{axis}[%
	width=\linewidth,
	xlabel={$N_{\rm dofs} \; (\nu=5.1)$},
	xmode=log,
	ymode=log,
	ymax=10,
	ymin=1.e-6
]

\addplot table[x=ndofs,y=err] {figures/res/upwind/hada/P2_F5.1.txt}
	node[pos=0.2,pin=0:{$\widetilde \xi$}] {};
\addplot table[x=ndofs,y=est] {figures/res/upwind/hada/P2_F5.1.txt}
	node[pos=0.2,pin=-90:{$\widetilde \eta$}] {};

\addplot[black,dashed,domain=5e3:8.e4] {8000*x^(-1.5)};
\SlopeTriangle{.6}{-.15}{.4}{-1.5}{$N_{\rm dofs}^{-p/2}$}{}

\end{axis}
\end{tikzpicture}
\end{minipage}
}

\subcaptionbox{$p=4$}{
\begin{minipage}{.45\linewidth}
\begin{tikzpicture}
\begin{axis}[%
	width=\linewidth,
	xlabel={$N_{\rm dofs} \; (\nu=1.1)$},
	xmode=log,
	ymode=log,
	ymax=10,
	ymin=1.e-6
]

\addplot table[x=ndofs,y=err] {figures/res/upwind/hada/P3_F1.1.txt}
	node[pos=0.2,pin=-90:{$\widetilde \xi$}] {};
\addplot table[x=ndofs,y=est] {figures/res/upwind/hada/P3_F1.1.txt}
	node[pos=0.2,pin= 90:{$\widetilde \eta$}] {};

\addplot[black,dashed,domain=7e2:1.e4] {8e2*x^(-2)};
\SlopeTriangle{.45}{-.15}{.25}{-2}{$N_{\rm dofs}^{-p/2}$}{}

\end{axis}
\end{tikzpicture}
\end{minipage}
\begin{minipage}{.45\linewidth}
\begin{tikzpicture}
\begin{axis}[%
	width=\linewidth,
	xlabel={$N_{\rm dofs} \; (\nu=5.1)$},
	xmode=log,
	ymode=log,
	ymax=10,
	ymin=1.e-6
]

\addplot table[x=ndofs,y=err] {figures/res/upwind/hada/P3_F5.1.txt}
	node[pos=.25,pin=0:{$\widetilde \xi$}] {};
\addplot table[x=ndofs,y=est] {figures/res/upwind/hada/P3_F5.1.txt}
	node[pos=0.3,pin=-90:{$\widetilde \eta$}] {};

\addplot[black,dashed,domain=2e3:2.e4] {2e5*x^(-2)};
\SlopeTriangle{.6}{-.15}{.45}{-2}{$N_{\rm dofs}^{-p/2}$}{}

\end{axis}
\end{tikzpicture}
\end{minipage}
}

\caption{Nearly resonant experiment: errors in $h$-adaptive refinements}
\label{figure_res_hada_errors}
\end{figure}
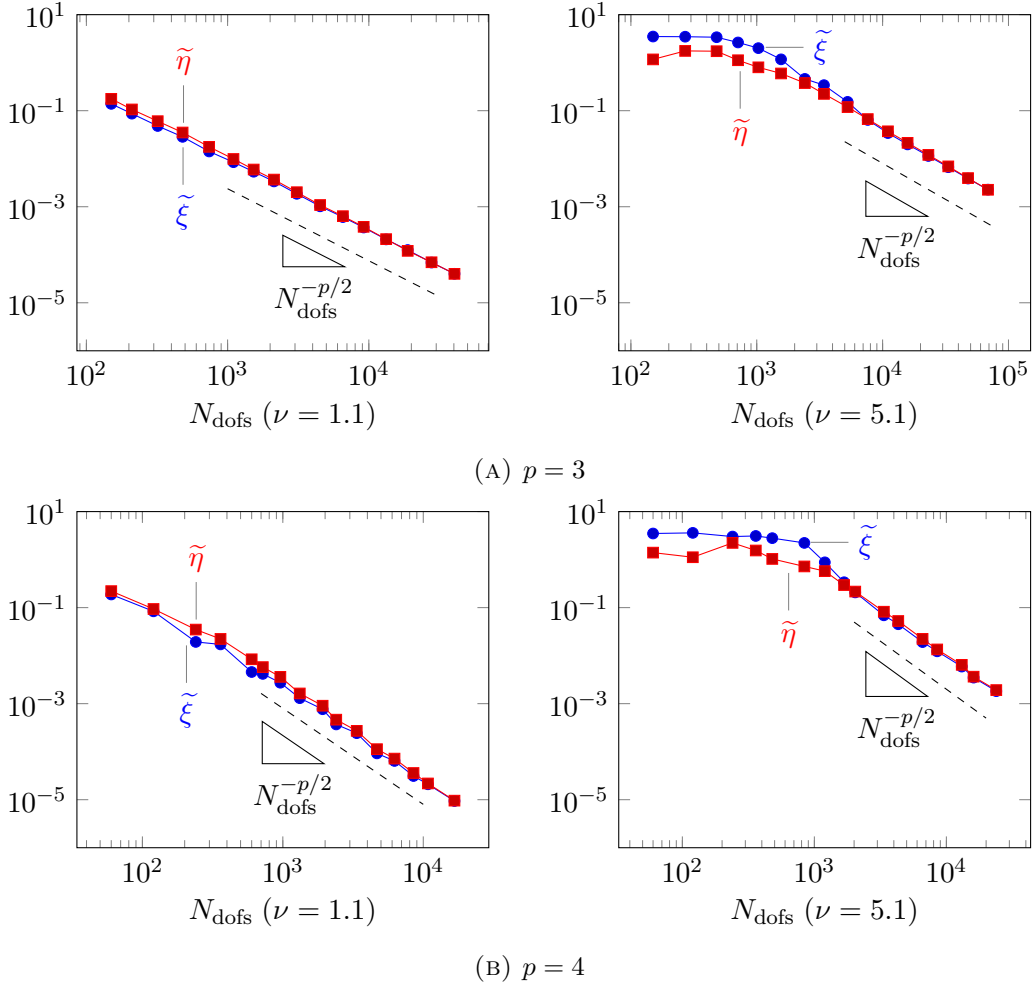

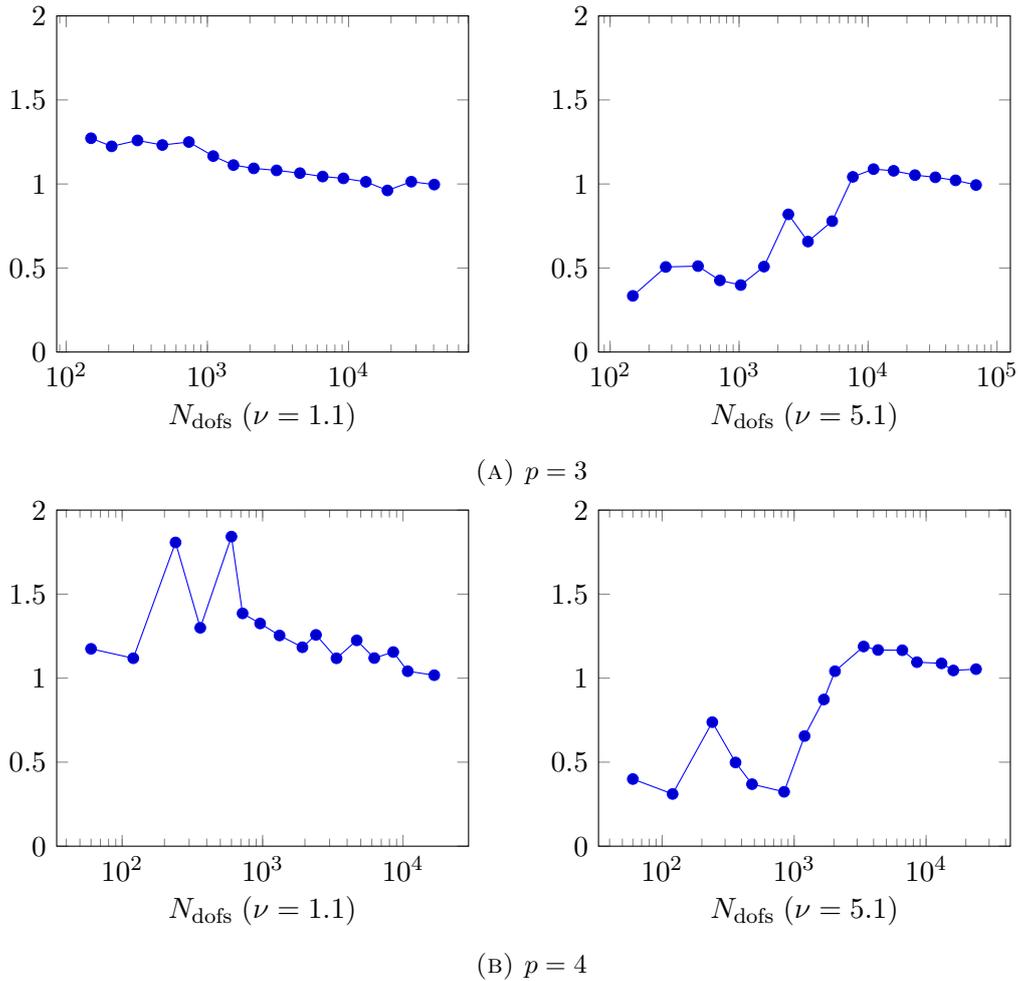
\begin{figure}
\subcaptionbox{$p=3$}{
\begin{minipage}{.45\linewidth}
\begin{tikzpicture}
\begin{axis}[%
	width=\linewidth,
	xlabel={$N_{\rm dofs} \; (\nu=1.1)$},
	xmode=log,
	ymax=2,
	ymin=0
]
\addplot table[x=ndofs,y expr=\thisrow{est}/\thisrow{err}] {figures/res/upwind/hada/P2_F1.1.txt};
\end{axis}
\end{tikzpicture}
\end{minipage}
\begin{minipage}{.45\linewidth}
\begin{tikzpicture}
\begin{axis}[%
	width=\linewidth,
	xlabel={$N_{\rm dofs} \; (\nu=5.1)$},
	xmode=log,
	ymax=2,
	ymin=0
]
\addplot table[x=ndofs,y expr=\thisrow{est}/\thisrow{err}] {figures/res/upwind/hada/P2_F5.1.txt};
\end{axis}
\end{tikzpicture}
\end{minipage}
}

\subcaptionbox{$p=4$}{
\begin{minipage}{.45\linewidth}
\begin{tikzpicture}
\begin{axis}[%
	width=\linewidth,
	xlabel={$N_{\rm dofs} \; (\nu=1.1)$},
	xmode=log,
	ymax=2,
	ymin=0
]
\addplot table[x=ndofs,y expr=\thisrow{est}/\thisrow{err}] {figures/res/upwind/hada/P3_F1.1.txt};
\end{axis}
\end{tikzpicture}
\end{minipage}
\begin{minipage}{.45\linewidth}
\begin{tikzpicture}
\begin{axis}[%
	width=\linewidth,
	xlabel={$N_{\rm dofs} \; (\nu=5.1)$},
	xmode=log,
	ymax=2,
	ymin=0
]
\addplot table[x=ndofs,y expr=\thisrow{est}/\thisrow{err}] {figures/res/upwind/hada/P3_F5.1.txt};
\end{axis}
\end{tikzpicture}
\end{minipage}
}

\caption{Nearly resonant experiment: effectivity indices in $h$-adaptive refinements}
\label{figure_res_hada_effectivities}
\end{figure}

\begin{figure}
\subcaptionbox{$p=3$}{\includegraphics[width=\linewidth]{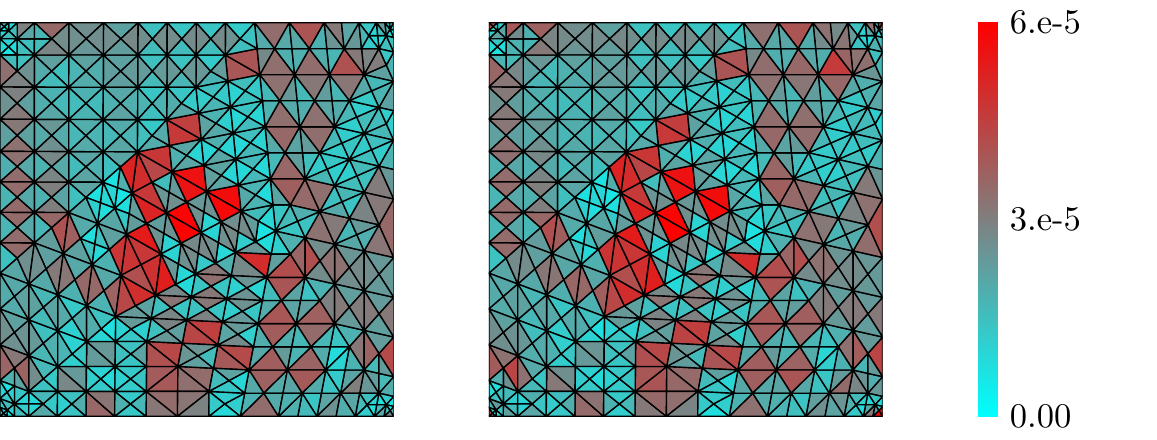}}
\subcaptionbox{$p=4$}{\includegraphics[width=\linewidth]{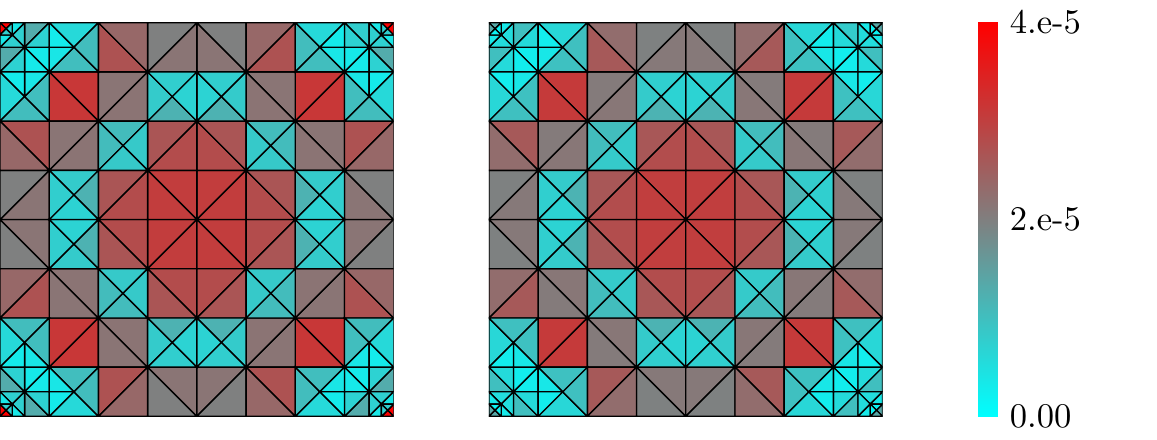}}
\caption{Nearly resonant experiment: actual (left) and estimated (error) at iteration \# 10
in the $h$-adpative example with $\nu=1.1$}
\label{figure_res_hada_img_nu11}
\end{figure}

\begin{figure}
\subcaptionbox{$p=3$}{\includegraphics[width=\linewidth]{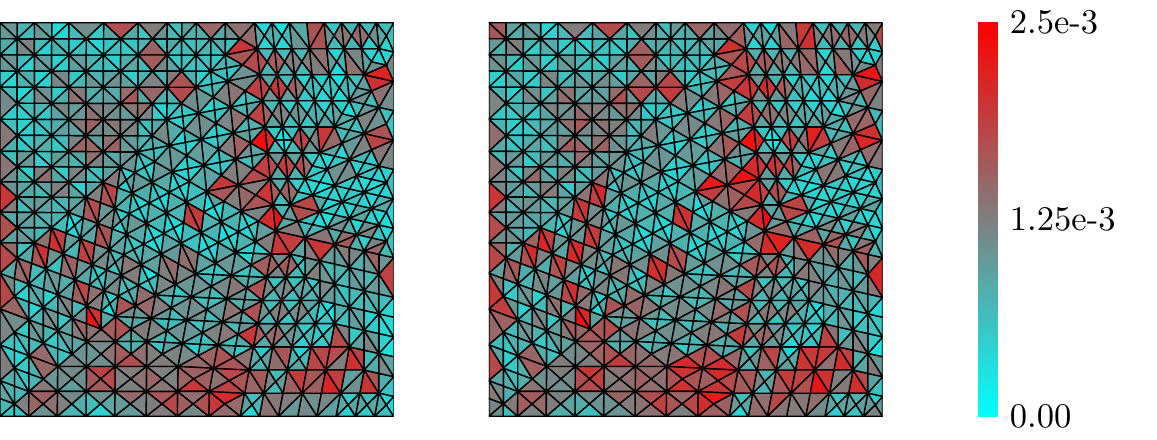}}
\subcaptionbox{$p=4$}{\includegraphics[width=\linewidth]{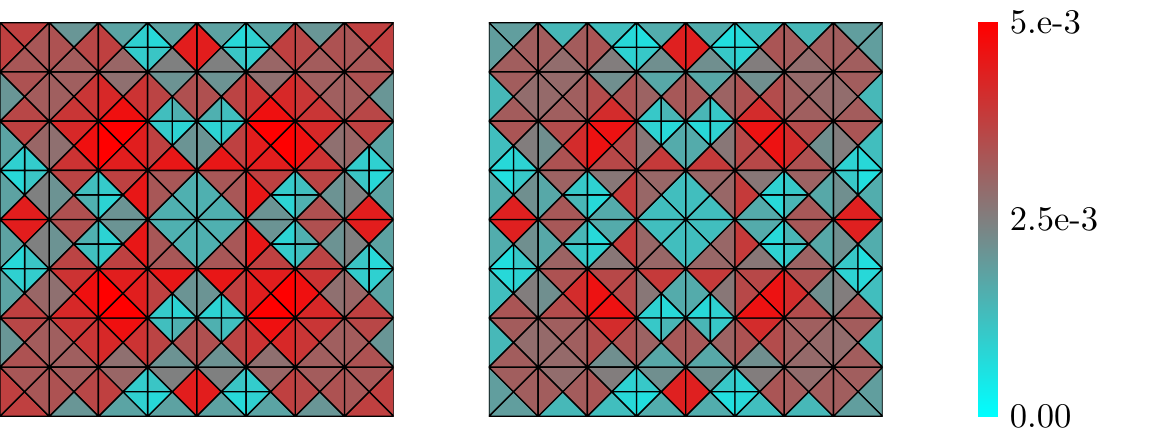}}
\caption{Nearly resonant experiment: actual (left) and estimated (error) at iteration \# 10
in the $h$-adpative example with $\nu=5.1$}
\label{figure_res_hada_img_nu51}
\end{figure}

\subsection{Reflection of a Gaussian beam}

In this example, we consider the reflection of a Gaussian beam by two
heterogeneous prisms modeled by heterogeneous coefficients. The domain
of interest $\Omega_0 \eq (-\ell,\ell)^2$, where $\ell \eq 1$, is surrounded by a PML of thickness
$\ell_{\rm PML} \eq 0.1$, so that $\Omega \eq (-(\ell+\ell_{\rm PML}),\ell+\ell_{\rm PML})^2$.
We set $J \eq 0$, and we define the TF region
$\Omega_{\rm TF} \eq (-\ell_{\rm TF},\ell_{\rm TF})^2$ with $\ell_{\rm TF} \eq 0.9$.
The incident field is the Gaussian beam
\begin{equation*}
E_{\rm inc}(\bx)
\eq
\frac{1}{\sqrt{4\alpha\beta(y_1)}}
\exp\left (-\frac{y_2^2}{4\beta(y_1)}\right )
\exp\left (-i\omega y_1\right ),
\quad
\beta(y_1) \eq \frac{1}{4\alpha} + i\frac{y_1}{2\omega},
\end{equation*}
with $\by = \bx-\bx_0$, which is the solution to paraxial Helhmoltz equation
\begin{equation*}
-\omega^2 u - i\omega \frac{\partial u}{\partial y_1} - \frac{\partial^2 u}{\partial y_2^2} = 0,
\qquad
u(0,y_2) = g(y_2)
\end{equation*}
with initial data $g(y_2) \eq \exp (-\alpha y_2)$, for $\bx_0 \eq (-1.0,0.3)$ and
$\alpha \eq w^{-2}$, $w = 0.1$, and we set $\BH_{\rm inc} \eq (i\omega)^{-1} \ccurl E_{\rm inc}$.
Notice that, strictly speaking, these incident fields do not enter our assumption in
\eqref{eq_assumption_incident_fields} because $E_{\rm inc}$ solves the paraxial Helmholtz equation,
which is only an approximation of the ``true'' Helmholtz equation. Nevertheless, the paraxial
approximation is fairly good for Gaussian beams, so the resulting discrepancy is not seen
numerically for the accuracy level we target. The prisms are the two triangles
\begin{equation*}
P_1 \eq \{( 0.3,0.0),(0.6, 0.6),( 0.0, 0.6)\}, \quad
P_2 \eq \{(-0.3,0.0),(0.0,-0.6),(-0.6,-0.6)\},
\end{equation*}
and the electric permittivity is defined by $\varepsilon \eq 1/10$ in $P_1 \cup P_2$ and
$\varepsilon \eq 1$ outside, and we apply the usual modifications in the PMLs.
We also select the frequency $\nu \eq 50$.
The whole setting is sketched on Figure \ref{figure_beam_setting}, where the
initial mesh is also represented.

We start the adaptive algorithm with the mesh represented in Figure \ref{figure_beam_setting},
and we employ the fixed polynomial degree $p \eq 6$. We run the adaptive loop for 30
iterations, and the final discrete solution is represented in Figure \ref{figure_beam_solution}.
In Figure \ref{figure_beam_convergence}, we plot the value of the estimator against the
number of degrees of freedom throughout the adaptive loop (we do not represent the true
error as it is not available here). We observe an initial stagnation, which is coherent
with other experiments of adaptivity for high-frequency waves \cite{chaumontfrelet_ern_vohralik_2019}.
This initial stagnation is actually expected since we start with a largely unresolved
mesh. We then see that the optimal convergence rate is asymptotically reached.

Figure \ref{figure_beam_images} depicts the solution, the mesh size, and the estimator
at various iterations of the adaptive loop. Interestingly, we see that the refinements
essentially follow the wavefront until the beam goes through the whole domain. More
uniform refinements then occur to capture the diffracted rays. We also observe that
the mesh is refined on some corners of the prisms, which agrees with the expected
presence of singularities at these points.

\begin{figure}

\begin{minipage}{.45\linewidth}
\begin{tikzpicture}[scale=2.5]
\draw (-1.1,-1.1) rectangle (1.1,1.1);

\draw (-1.1,-1.0) -- ( 1.1,-1.0);
\draw (-1.1, 1.0) -- ( 1.1, 1.0);

\draw (-1.0,-1.1) -- (-1.0, 1.1);
\draw ( 1.0,-1.1) -- ( 1.0, 1.1);

\draw (-0.9,-0.9) rectangle (0.9,0.9);

\draw[pattern=north west lines] ( 0.3,0.0) -- (0.6, 0.6) -- ( 0.0, 0.6) -- cycle;
\draw[pattern=north west lines] (-0.3,0.0) -- (0.0,-0.6) -- (-0.6,-0.6) -- cycle;
\end{tikzpicture}
\end{minipage}
\begin{minipage}{.45\linewidth}
\input{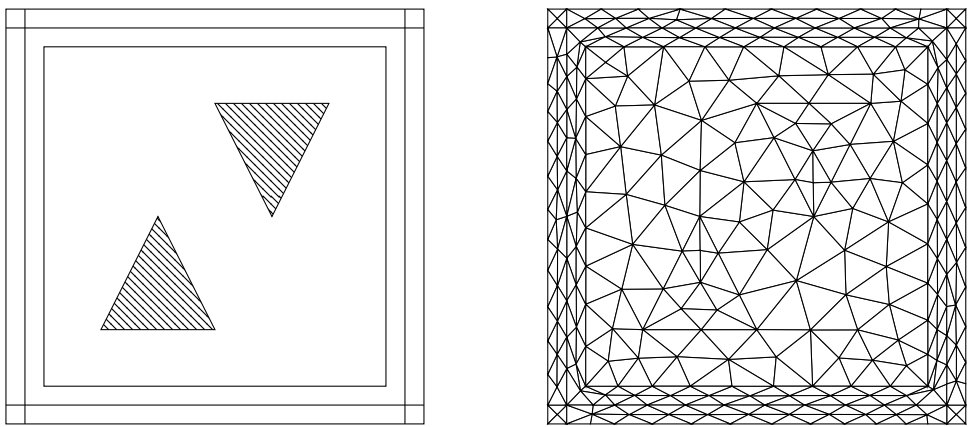}
\end{minipage}

\caption{Setting and initial mesh of the Gaussian beam experiment}
\label{figure_beam_setting}
\end{figure}

\begin{figure}

\begin{minipage}{.90\linewidth}
\scalebox{1}[-1]{\includegraphics[width=\linewidth]{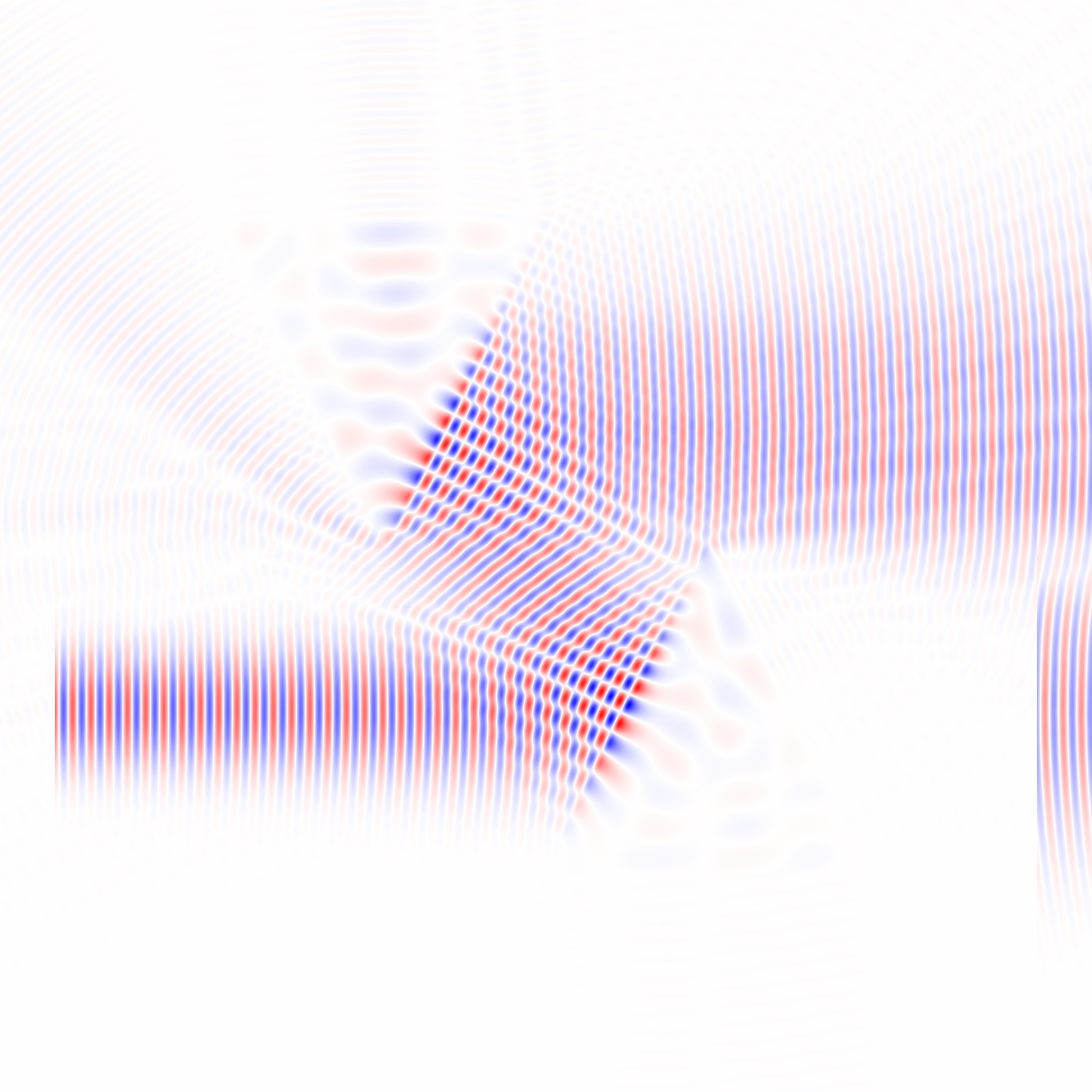}}
\end{minipage}
\begin{minipage}{.05\linewidth}
\begin{tikzpicture}
\draw (0,-2.50) node[anchor=south west,inner sep=0]%
{\includegraphics[width=0.25cm,height=5.0cm]{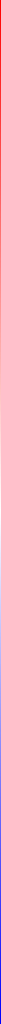}};
\draw[thick] (0,-2.50) rectangle (.25,2.50);

\draw (0.25,-2.50) node [anchor=west] {$-4$};
\draw (0.25, 2.50) node [anchor=west] {$\;4$};
\end{tikzpicture}
\end{minipage}
\caption{Real part of the solution in the Gaussian beam experiment}
\label{figure_beam_solution}
\end{figure}
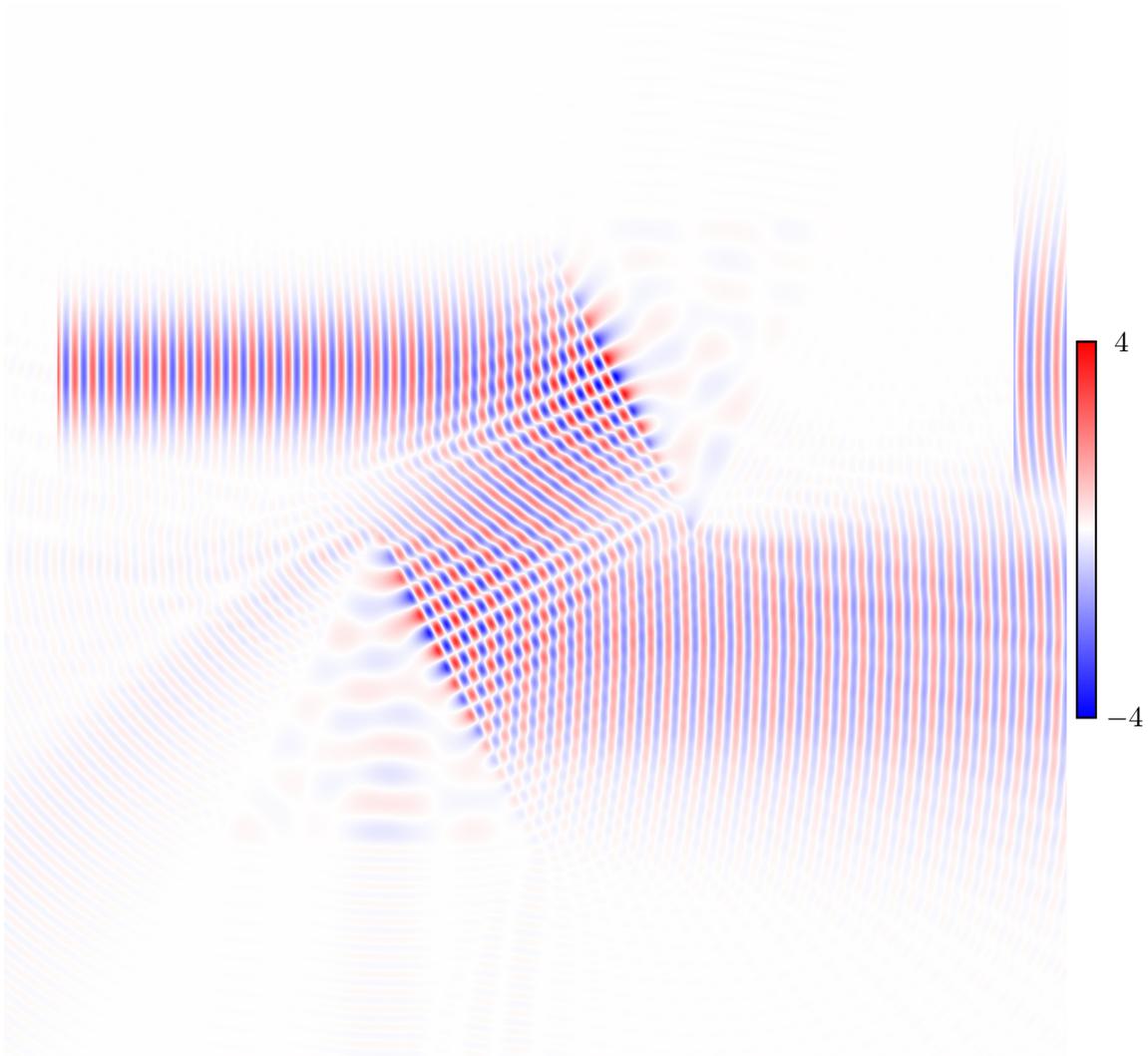

\begin{figure}
\begin{tikzpicture}
\begin{axis}
[
	width=\linewidth,
	xlabel={$N_{\rm dofs}$},
	ylabel={$\widetilde \eta$},
	xmode=log,
	ymode=log
]

\addplot[black,mark=o]
table[x = ndofs,y expr = \thisrow{est}/126.792]
{figures/beam/data.txt};

\plot[domain=2e5:2e6,dashed] {1e14*x^(-3)};

\SlopeTriangle{.8}{-.075}{.1}{-3}{$N_{\rm dofs}^{-3}$}{}

\end{axis}
\end{tikzpicture}

\caption{Convergence in the Gaussian beam experiment}
\label{figure_beam_convergence}
\end{figure}
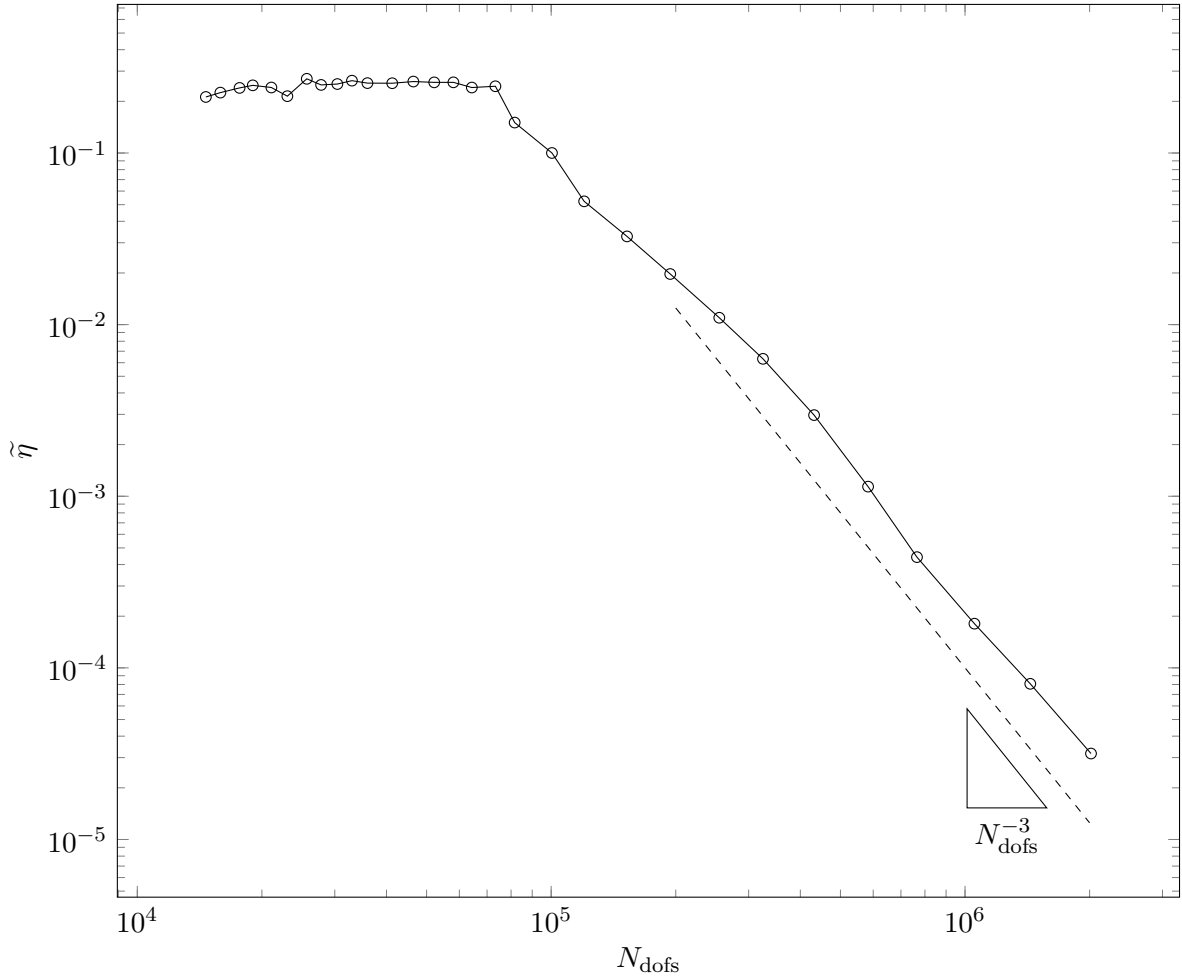

\begin{figure}

\begin{minipage}{.30\linewidth}
\scalebox{1}[-1]{\includegraphics[width=\linewidth]{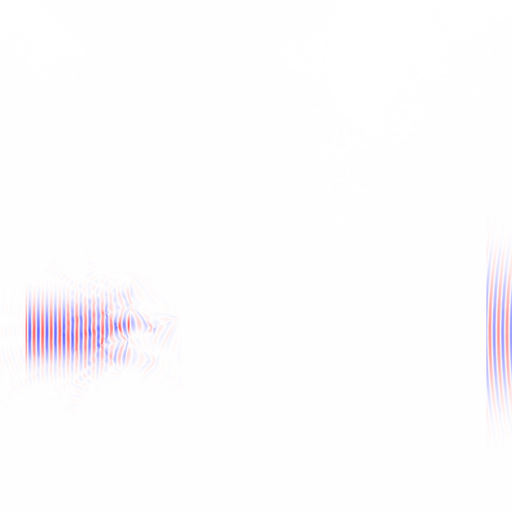}}
\end{minipage}
\begin{minipage}{.30\linewidth}
\scalebox{1}[-1]{\includegraphics[width=\linewidth]{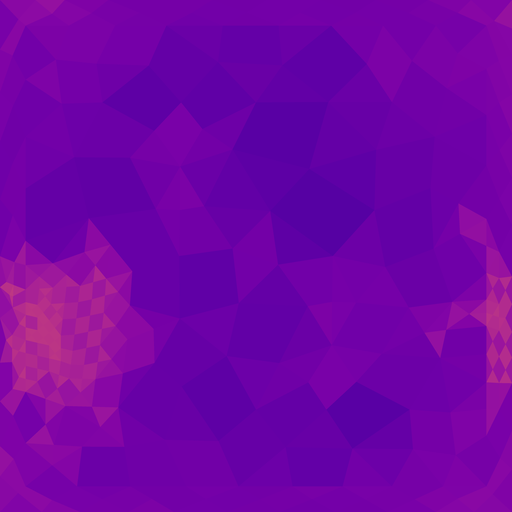}}
\end{minipage}
\begin{minipage}{.30\linewidth}
\scalebox{1}[-1]{\includegraphics[width=\linewidth]{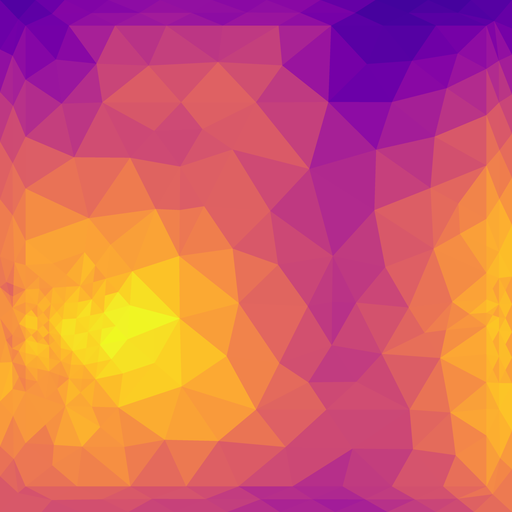}}
\end{minipage}
\begin{turn}{-90}
\hspace{-1cm}Iteration \#5
\end{turn}

\begin{minipage}{.30\linewidth}
\scalebox{1}[-1]{\includegraphics[width=\linewidth]{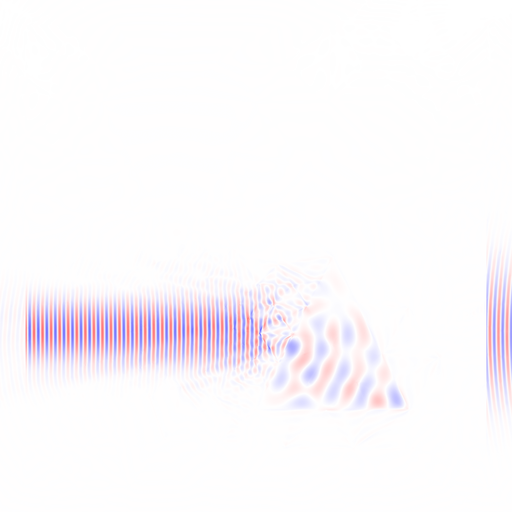}}
\end{minipage}
\begin{minipage}{.30\linewidth}
\scalebox{1}[-1]{\includegraphics[width=\linewidth]{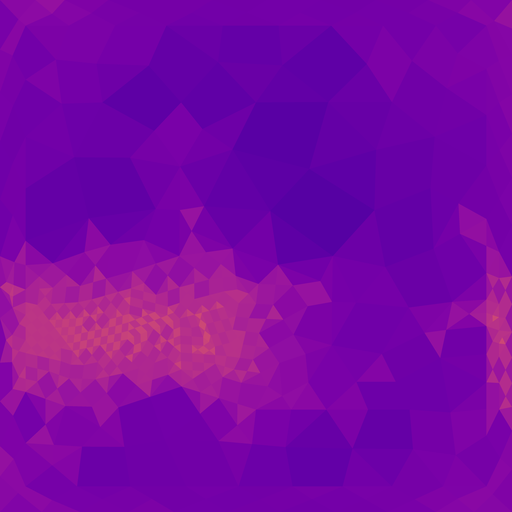}}
\end{minipage}
\begin{minipage}{.30\linewidth}
\scalebox{1}[-1]{\includegraphics[width=\linewidth]{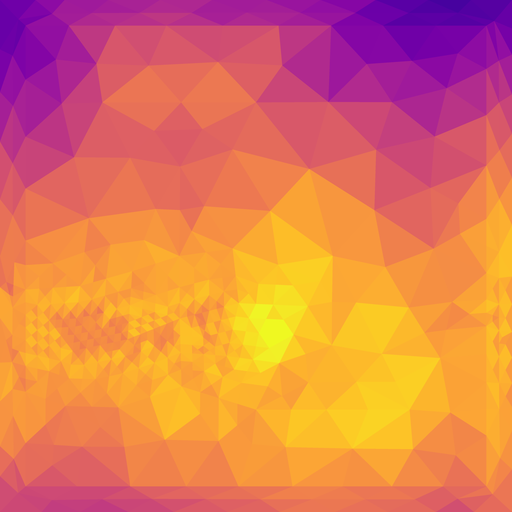}}
\end{minipage}
\begin{turn}{-90}
\hspace{-1cm}Iteration \#10
\end{turn}

\begin{minipage}{.30\linewidth}
\scalebox{1}[-1]{\includegraphics[width=\linewidth]{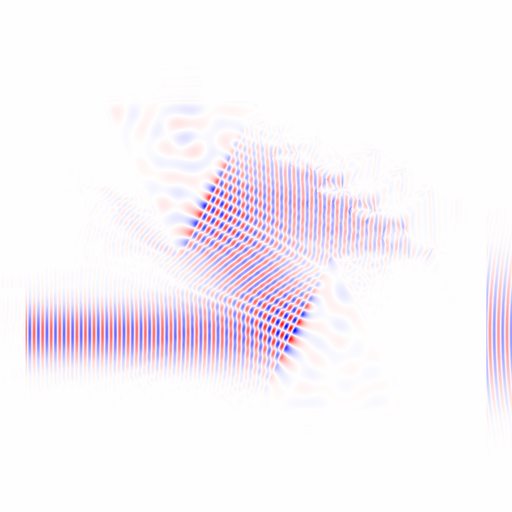}}
\end{minipage}
\begin{minipage}{.30\linewidth}
\scalebox{1}[-1]{\includegraphics[width=\linewidth]{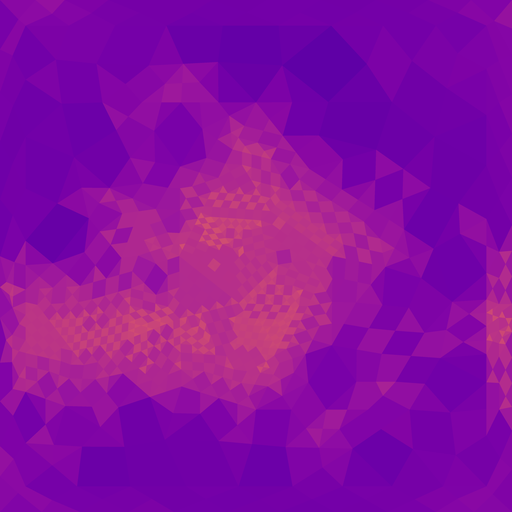}}
\end{minipage}
\begin{minipage}{.30\linewidth}
\scalebox{1}[-1]{\includegraphics[width=\linewidth]{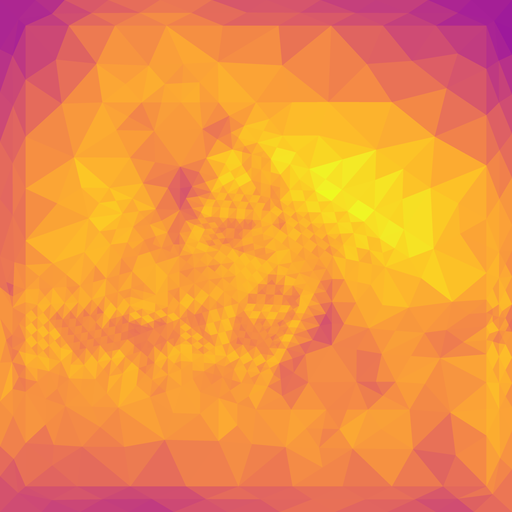}}
\end{minipage}
\begin{turn}{-90}
\hspace{-1cm}Iteration \#15
\end{turn}

%

\begin{minipage}{.30\linewidth}
\scalebox{1}[-1]{\includegraphics[width=\linewidth]{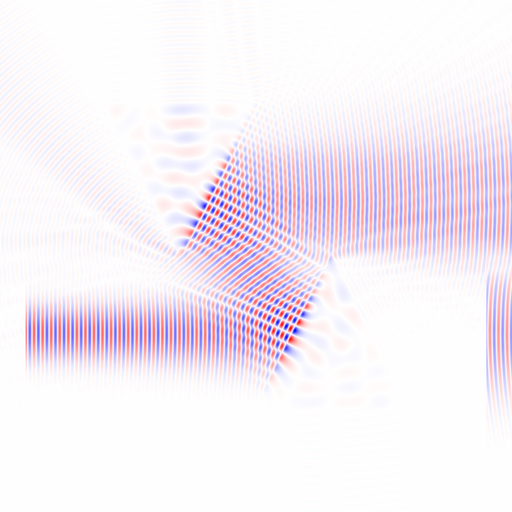}}
\end{minipage}
\begin{minipage}{.30\linewidth}
\scalebox{1}[-1]{\includegraphics[width=\linewidth]{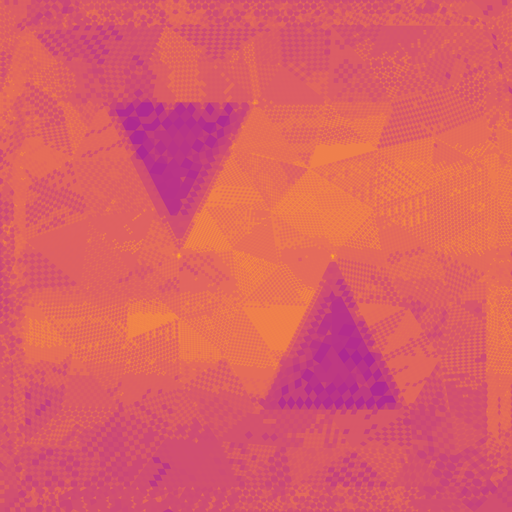}}
\end{minipage}
\begin{minipage}{.30\linewidth}
\scalebox{1}[-1]{\includegraphics[width=\linewidth]{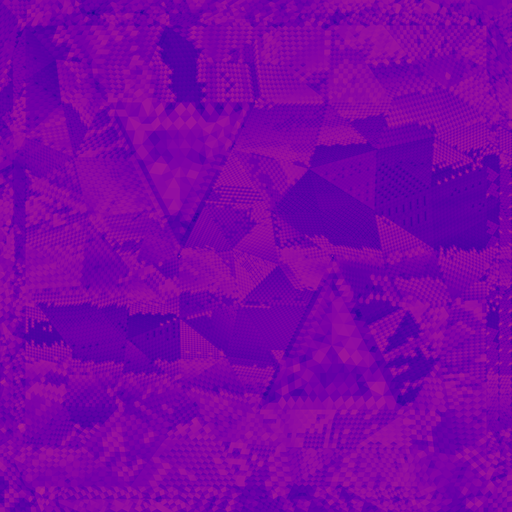}}
\end{minipage}
\begin{turn}{-90}
\hspace{-1cm}Iteration \#30
\end{turn}

\begin{minipage}{.30\linewidth}
\begin{tikzpicture}
\node[inner sep=0pt,anchor=south west] at (0cm,0cm) {\includegraphics[width=2.5cm,height=0.5cm]{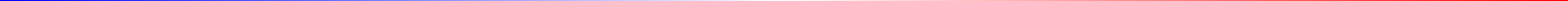}};

\draw[thick] (0.0cm,0.0cm) rectangle (2.5cm,0.5cm);

\draw[dashed] (1.25cm,0.0cm) -- (1.25cm,0.5cm);

\node[anchor=north] at (0.00cm,0.0cm) {$-1.7$};
\node[anchor=north] at (1.25cm,0.0cm) {$ 0.0$};
\node[anchor=north] at (2.50cm,0.0cm) {$ 1.7$};
\end{tikzpicture}
\end{minipage}
\begin{minipage}{.30\linewidth}
\begin{tikzpicture}
\node[inner sep=0pt,anchor=south west] at (0cm,0cm) {\includegraphics[width=2.5cm,height=0.5cm]{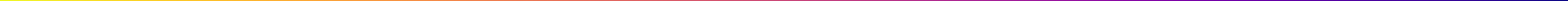}};

\draw[thick] (0.0cm,0.0cm) rectangle (2.5cm,0.5cm);

\draw[dashed] (0.83333cm,0.0cm) -- (0.83333cm,0.5cm);
\draw[dashed] (1.66666cm,0.0cm) -- (1.66666cm,0.5cm);

\node[anchor=north] at (0.00000cm,0.0cm) {$10^{-3}$};
\node[anchor=north] at (0.83333cm,0.0cm) {$10^{-2}$};
\node[anchor=north] at (1.66666cm,0.0cm) {$10^{-1}$};
\node[anchor=north] at (2.50000cm,0.0cm) {$1$};
\end{tikzpicture}
\end{minipage}
\begin{minipage}{.30\linewidth}
\begin{tikzpicture}
\node[inner sep=0pt,anchor=south west] at (0cm,0cm) {\includegraphics[width=2.5cm,height=0.5cm]{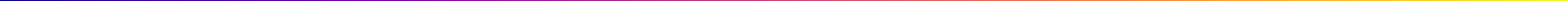}};

\draw[thick] (0.0cm,0.0cm) rectangle (2.5cm,0.5cm);

\draw[dashed] (0.83333cm,0.0cm) -- (0.83333cm,0.5cm);
\draw[dashed] (1.66666cm,0.0cm) -- (1.66666cm,0.5cm);

\node[anchor=north] at (0.00000cm,0.0cm) {$10^{-7}$};
\node[anchor=north] at (0.83333cm,0.0cm) {$10^{-5}$};
\node[anchor=north] at (1.66666cm,0.0cm) {$10^{-3}$};
\node[anchor=north] at (2.50000cm,0.0cm) {$1$};
\end{tikzpicture}
\end{minipage}

\caption{Solution mesh size and estimator in the Gaussian beam experiment}
\label{figure_beam_images}
\end{figure}
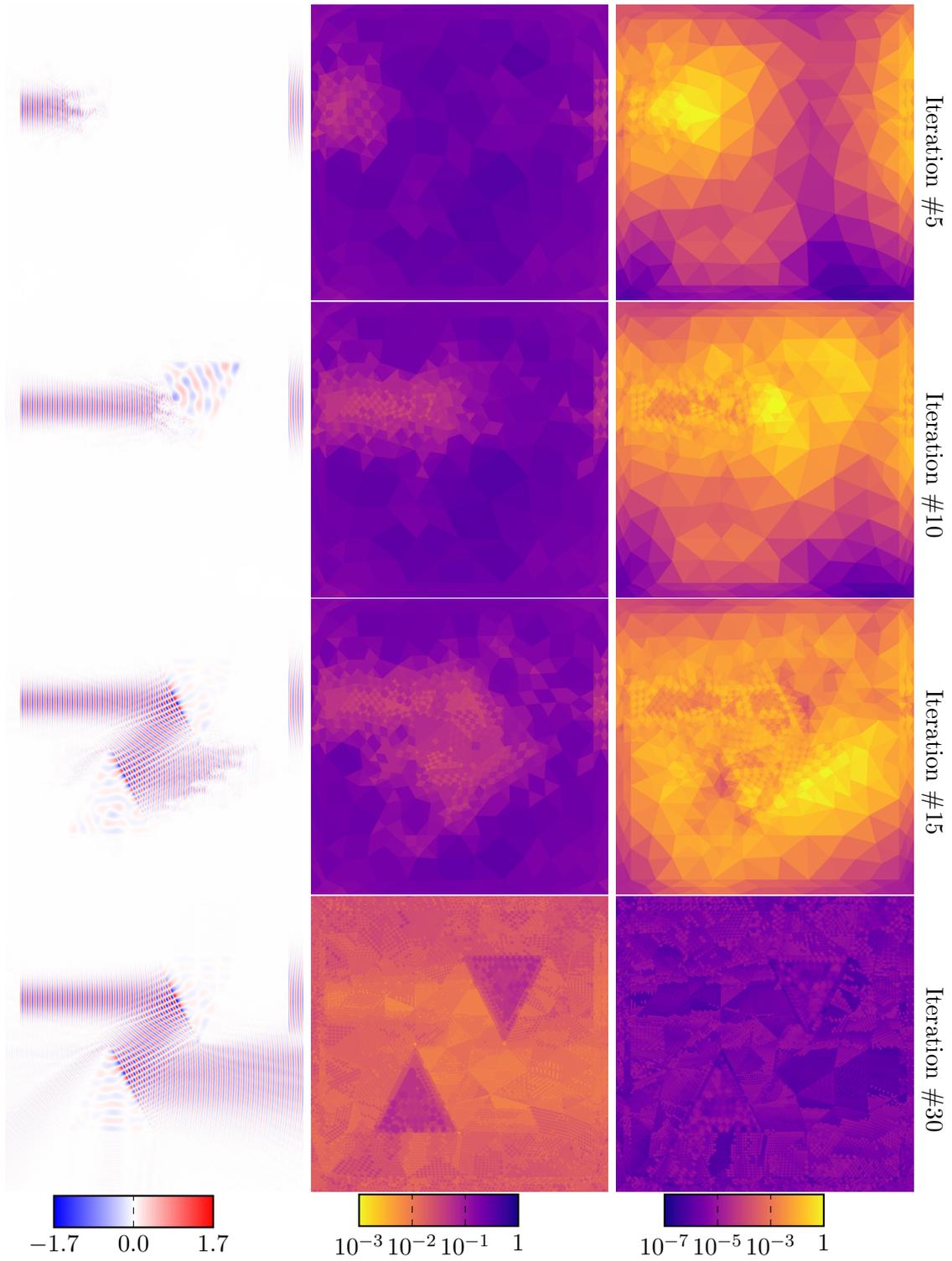

\subsection{Scattering by an aircraft}

Our final example is the scattering of planewave by an aircraft.
The incident field is given by $E_{\rm inc} \eq e^{i\omega\bd\cdot\bx}$
with $\bd = (\cos\phi,\sin\phi)$ and $\phi \eq 13\pi/12$ whereas
$\BH_{\rm inc} \eq (i\omega)^{-1}\ccurl E_{\rm inc}$. The domain of interest is
the square $\Omega_0 \eq (-\ell,\ell)^2 \setminus \overline{A}$, where
$A \subset\!\subset \Omega_0$ is an open Lipschitz polygon with 76 vertices
representing an aircraft (see Figure \ref{figure_aircraft_setting}) and $\ell \eq 40$,
the TF region is $\Omega_{\rm TF} \eq (-\ell_{\rm TF},\ell_{\rm TF})^2 \setminus A$ with
$\ell_{\rm TF}$ and we surround $\Omega_0$ with PMLs of length $\ell_{\rm PML} \eq 5$.
We impose the PEC condition on the boundary of
$\Omega \eq (-(\ell+\ell_{\rm PML}),\ell+\ell_{\rm PML})^2 \setminus \overline{A}$.
Figure \ref{figure_aircraft_setting} depicts the whole setting.

We employ the $hp$-adaptive algorithm presented in Section \ref{section_hp_adaptivity}
for $15$ iterations, starting with the mesh shown in Figure \ref{figure_aircraft_setting}
with a uniform polynomial degree distribution $p = 1$. We consider two frequencies:
$\nu \eq 0.1$ and $\nu \eq 0.5$. 
The solutions are represented in Figure \ref{figure_aircraft_solution},
whereas Figure \ref{figure_aircraft_convergence} shows the convergence history of the
$hp$-adaptive loop in both cases. We observe an exponential convergence rate for the
two frequencies (following \cite{melenk_wohlmuth_2001a}
and \cite[Theorem 4.63]{schwab_1998a}, we plot the estimator $\widetilde \eta$
in log-scale against $N_{\rm dofs}^{1/3}$ in linear-scale), indicating that the proposed
estimator is suited to drive $hp$-adaptive algorithms.

\begin{figure}

\begin{minipage}{.45\linewidth}
\begin{tikzpicture}[scale=.04]
\begin{scope}[yscale=-1,xscale=1]

\draw (-1.5*45,-1.5*45) rectangle (1.5*45,1.5*45);

\draw (-1.5*35,-1.5*35) rectangle (1.5*35,1.5*35);

\draw (-1.5*45,-1.5*40) -- ( 1.5*45,-1.5*40);
\draw (-1.5*45, 1.5*40) -- ( 1.5*45, 1.5*40);

\draw (-1.5*40,-1.5*45) -- (-1.5*40, 1.5*45);
\draw ( 1.5*40,-1.5*45) -- ( 1.5*40, 1.5*45);

\fill
( 14.1-14.1-45.0,38.4-38.4) -- 
( 15.2-14.1-45.0,37.2-38.4) -- 
( 25.4-14.1-45.0,34.4-38.4) -- 
( 40.0-14.1-45.0,34.4-38.4) -- 
( 48.3-14.1-45.0,25.5-38.4) -- 
( 44.4-14.1-45.0,25.5-38.4) -- 
( 44.4-14.1-45.0,22.0-38.4) -- 
( 48.4-14.1-45.0,22.2-38.4) -- 
( 48.5-14.1-45.0,22.7-38.4) -- 
( 50.6-14.1-45.0,22.8-38.4) -- 
( 58.5-14.1-45.0,14.4-38.4) -- 
( 54.4-14.1-45.0,14.5-38.4) -- 
( 54.5-14.1-45.0,11.1-38.4) -- 
( 58.5-14.1-45.0,11.2-38.4) -- 
( 58.5-14.1-45.0,11.7-38.4) -- 
( 60.7-14.1-45.0,11.7-38.4) -- 
( 69.8-14.1-45.0, 1.8-38.4) -- 
( 79.0-14.1-45.0, 1.7-38.4) -- 
( 75.0-14.1-45.0, 2.0-38.4) -- 
( 67.0-14.1-45.0,15.3-38.4) -- 
( 68.0-14.1-45.0,15.7-38.4) -- 
( 66.5-14.1-45.0,16.0-38.4) -- 
( 64.1-14.1-45.0,20.0-38.4) -- 
( 65.2-14.1-45.0,20.4-38.4) -- 
( 63.8-14.1-45.0,20.7-38.4) -- 
( 62.3-14.1-45.0,23.3-38.4) -- 
( 61.0-14.1-45.0,27.6-38.4) -- 
( 62.2-14.1-45.0,28.0-38.4) -- 
( 60.8-14.1-45.0,28.3-38.4) -- 
( 59.9-14.1-45.0,31.1-38.4) -- 
( 61.1-14.1-45.0,31.5-38.4) -- 
( 59.7-14.1-45.0,31.7-38.4) -- 
( 58.7-14.1-45.0,34.4-38.4) -- 
( 81.8-14.1-45.0,34.5-38.4) -- 
( 88.7-14.1-45.0,36.0-38.4) -- 
( 98.9-14.1-45.0,25.1-38.4) -- 
(102.3-14.1-45.0,24.9-38.4) -- 
( 98.5-14.1-45.0,37.9-38.4) -- 
(104.1-14.1-45.0,38.4-38.4) -- 
( 98.5-14.1-45.0,38.4-37.9) -- 
(102.3-14.1-45.0,38.4-24.9) -- 
( 98.9-14.1-45.0,38.4-25.1) -- 
( 88.7-14.1-45.0,38.4-36.0) -- 
( 81.8-14.1-45.0,38.4-34.5) -- 
( 58.7-14.1-45.0,38.4-34.4) -- 
( 59.7-14.1-45.0,38.4-31.7) -- 
( 61.1-14.1-45.0,38.4-31.5) -- 
( 59.9-14.1-45.0,38.4-31.1) -- 
( 60.8-14.1-45.0,38.4-28.3) -- 
( 62.2-14.1-45.0,38.4-28.0) -- 
( 61.0-14.1-45.0,38.4-27.6) -- 
( 62.3-14.1-45.0,38.4-23.3) -- 
( 63.8-14.1-45.0,38.4-20.7) -- 
( 65.2-14.1-45.0,38.4-20.4) -- 
( 64.1-14.1-45.0,38.4-20.0) -- 
( 66.5-14.1-45.0,38.4-16.0) -- 
( 68.0-14.1-45.0,38.4-15.7) -- 
( 67.0-14.1-45.0,38.4-15.3) -- 
( 75.0-14.1-45.0,38.4- 2.0) -- 
( 79.0-14.1-45.0,38.4- 1.7) -- 
( 69.8-14.1-45.0,38.4- 1.8) -- 
( 60.7-14.1-45.0,38.4-11.7) -- 
( 58.5-14.1-45.0,38.4-11.7) -- 
( 58.5-14.1-45.0,38.4-11.2) -- 
( 54.5-14.1-45.0,38.4-11.1) -- 
( 54.4-14.1-45.0,38.4-14.5) -- 
( 58.5-14.1-45.0,38.4-14.4) -- 
( 50.6-14.1-45.0,38.4-22.8) -- 
( 48.5-14.1-45.0,38.4-22.7) -- 
( 48.4-14.1-45.0,38.4-22.2) -- 
( 44.4-14.1-45.0,38.4-22.0) -- 
( 44.4-14.1-45.0,38.4-25.5) -- 
( 48.3-14.1-45.0,38.4-25.5) -- 
( 40.0-14.1-45.0,38.4-34.4) -- 
( 25.4-14.1-45.0,38.4-34.4) -- 
( 15.2-14.1-45.0,38.4-37.2) -- cycle;
\end{scope}
\end{tikzpicture}
\end{minipage}
\begin{minipage}{.45\linewidth}
\input{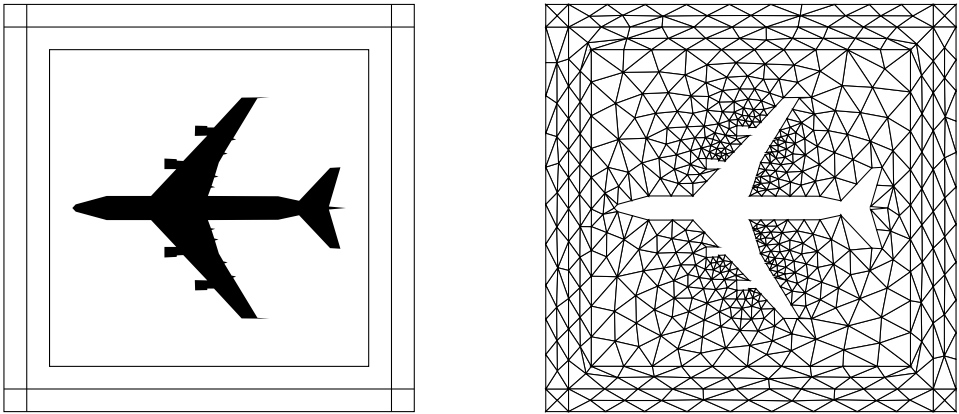}
\end{minipage}

\caption{Setting and initial mesh of the aircraft experiment}
\label{figure_aircraft_setting}
\end{figure}

\begin{figure}
\begin{minipage}{.45\linewidth}
\scalebox{1}[-1]{\includegraphics[width=5.5cm]{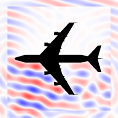}}
\subcaption{$\nu = 0.1$}
\end{minipage}
\begin{minipage}{.45\linewidth}
\scalebox{1}[-1]{\includegraphics[width=5.5cm]{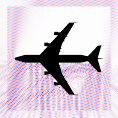}}
\subcaption{$\nu = 0.5$}
\end{minipage}
\begin{minipage}{.05\linewidth}
\vspace{-1cm}
\begin{tikzpicture}
\draw (0,-2.50) node[anchor=south west,inner sep=0]%
{\includegraphics[width=0.25cm,height=5.0cm]{figures/cbar_bwr}};
\draw[thick] (0,-2.50) rectangle (.25,2.50);

\draw (0.25,-2.50) node [anchor=west] {$-4$};
\draw (0.25, 2.50) node [anchor=west] {$\;4$};
\end{tikzpicture}
\end{minipage}
\caption{Real part of the solutions in the aircraft experiment}
\label{figure_aircraft_solution}
\end{figure}

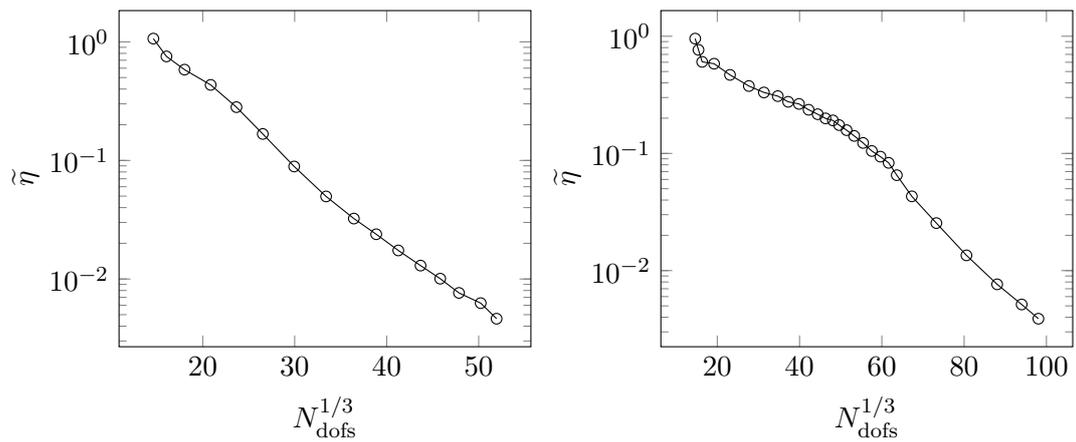
\begin{figure}
\begin{minipage}{.45\linewidth}
\begin{tikzpicture}
\begin{axis}
[
	width=\linewidth,
	xlabel={$N_{\rm dofs}^{1/3}$},
	ylabel={$\widetilde \eta$},
	ymode=log
]

\addplot[black,mark=o]
table[x expr = \thisrow{ndofs}^(1./3.),y expr = \thisrow{est}/57.7178]
{figures/aircraft/N1/data.txt};

\end{axis}
\end{tikzpicture}
\end{minipage}
\begin{minipage}{.45\linewidth}
\begin{tikzpicture}
\begin{axis}
[
	width=\linewidth,
	xlabel={$N_{\rm dofs}^{1/3}$},
	ylabel={$\widetilde \eta$},
	ymode=log
]

\addplot[black,mark=o]
table[x expr = \thisrow{ndofs}^(1./3.),y expr = \thisrow{est}/288.589]
{figures/aircraft/N5/data.txt};

\end{axis}
\end{tikzpicture}
\end{minipage}
\caption{Convergence in the aircraft experiment}
\label{figure_aircraft_convergence}
\end{figure}

\section{Conclusion}
\label{sec_conclusion}

We propose a new residual-based a posteriori error estimator
for discontinuous Galerkin discretizations (DG) of time-harmonic
Maxwell's equations in first-order form. Our estimator covers
a range of numerical DG fluxes, including the so-called ``central''
and ``upwind'' fluxes. We rigorously establish that the estimator
is reliable and efficient, with constants that are independent
of the frequency for sufficiently refined meshes. Besides, we show
that our estimator is asymptotically constant-free for a smooth solution.

We also present a set of numerical examples that highlights our key
findings. We find an excellent agreement between these examples
and the behavior predicted by the theory for our estimator. In
addition, we employ the estimator to drive $h$- and $hp$-adaptive
refinement processes. These examples suggest that the proposed estimator
is capable of driving such refinement processes since, in all cases, we
observe optimal convergence rates.

\section*{Funding}
PV was supported by the Chilean National Research and Development Agency (ANID) though the grant ANID FONDECYT No. 3220858.

\bibliographystyle{amsplain}
\bibliography{bibliography.bib}

\providecommand{\bysame}{\leavevmode\hbox to3em{\hrulefill}\thinspace}
\providecommand{\MR}{\relax\ifhmode\unskip\space\fi MR }
\providecommand{\MRhref}[2]{%
  \href{http://www.ams.org/mathscinet-getitem?mr=#1}{#2}
}
\providecommand{\href}[2]{#2}
\begin{thebibliography}{10}

\bibitem{adams_fournier_2003a}
R.~Adams and J.~Fournier, \emph{Sobolev spaces}, Academic Press, 2003.

\bibitem{arnold_brezzi_cockburn_marini_2002a}
D.~N. Arnold, F.~Brezzi, B.~Cockburn, and L.~D. Marini, \emph{Unified analysis
  of discontinuous {G}alerkin, methods for elliptic problems}, SIAM J. Numer.
  Anal. \textbf{39} (2002), 1749--1779.

\bibitem{beck_hiptmair_hoppe_wohlmuth_2000a}
R.~Beck, R.~Hiptmair, R.~H.~W. Hoppe, and B.~Wohlmuth, \emph{Residual based a
  posteriori error estimators for eddy current computation}, ESAIM Math. Model.
  Numer. Anal. \textbf{34} (2000), 159--182.

\bibitem{berenger_1994}
J.~P. B\'erenger, \emph{A perfectly matched layer for the absorption of
  electromagnetics waves}, J. Comput. Phys. \textbf{114} (1994), 185--200.

\bibitem{berenger_1996}
\bysame, \emph{Three-dimensional perfectly matched layer for the absorption of
  electromagnetic waves}, J. Comput. Phys. \textbf{127} (1996), 363--379.

\bibitem{bernkopf_sauter_torres_veit_2022a}
M.~Bernkopf, S.~Sauter, C.~Torres, and A.~Veit, \emph{Solvability of discrete
  {H}elmholtz equations}, arXiv:2105.02273v2, 2022.

\bibitem{bespalov_haberl_praetorius_2017a}
A.~Bespalov, A.~Haberl, and D.~Praetorius, \emph{Adaptive {FEM} with coarse
  initial mesh guarantees optimal convergence rates for compactly perturbed
  elliptic problems}, Comput. Methods Appl. Mech. Engrg. \textbf{317} (2017),
  318--340.

\bibitem{binev_dahmen_devore_2004a}
P.~Binev, W.~Dahmen, and R.~De Vore, \emph{Adaptive finite element methods with
  convergence rates}, Numer. Math. \textbf{97} (2004), 219--268.

\bibitem{bonito_nochetto_2010a}
A.~Bonito and R.~H. Nochetto, \emph{Quasi-optimal convergence rate of an
  adaptive discontinuous {G}alerkin method}, SIAM J. Numer. Anal. \textbf{48}
  (2010), 734--771.

\bibitem{buffa_ciarlet_2001a}
A.~Buffa and P.~Ciarlet Jr., \emph{On traces for functional spaces related to
  {M}axwell's equations part {I}: an integration by parts formula in
  {L}ipschitz polyhedra}, Math. Meth. Appl. Sci. \textbf{24} (2001), 9--30.

\bibitem{chaumontfrelet_2019a}
T.~Chaumont-Frelet, \emph{Mixed finite element discretization of acoustic
  {H}elmholtz problems with high wavenumbers}, Calcolo \textbf{56} (2019).

\bibitem{chaumontfrelet_ern_vohralik_2019}
T.~Chaumont-Frelet, A.~Ern, and M.~Vohral\'ik, \emph{On the derivation of
  guaranteed and p-robust a posteriori error estimates for the {H}elmholtz
  equation}, Numer. Math. \textbf{148} (2021), 525--573.

\bibitem{chaumontfrelet_nicaise_2018a}
T.~Chaumont-Frelet and S.~Nicaise, \emph{High-frequency behaviour of corner
  singularities in {H}elmholtz problems}, ESAIM Math. Model. Numer. Anal.
  \textbf{5} (2018), 1803--1845.

\bibitem{chaumontfrelet_nicaise_2019a}
\bysame, \emph{Wavenumber explicit convergence analysis for finite element
  discretizations of general wave propagation problems}, IMA J. Numer. Anal.
  \textbf{40} (2020), 1503--1543.

\bibitem{chaumontfrelet_vega_2020}
T.~Chaumont-Frelet and P.~Vega, \emph{Frequency-explicit a posteriori error
  estimates for finite element discretizations of {M}axwell's equations}, SIAM
  J. Numer. Anal. \textbf{60} (2022), 1774--1798.

\bibitem{chaumontfrelet_vega_2021a}
\bysame, \emph{Frequency-explicit approximability estimates for finite element
  discretizations of {M}axwell's equations}, Calcolo \textbf{59} (2022).

\bibitem{chen_xu_zou_2011}
J.~Chen, Y.~Xu, and J.~Zou, \emph{An adaptive edge element method and its
  convergence for a saddle-point problem from magnetostatics}, Numer. Methods
  Partial Differ. Equ. \textbf{28} (2011), 1643--1666.

\bibitem{ciarlet_2002a}
P.~G. Ciarlet, \emph{The finite element method for elliptic problems}, SIAM,
  2002.

\bibitem{costabel_dauge_nicaise_1999a}
M.~Costabel, M.~Dauge, and S.~Nicaise, \emph{Singularities of {M}axwell
  interface problems}, ESAIM Math. Model. Numer. Anal. \textbf{33} (1999),
  627--649.

\bibitem{demlow_hirani_2014}
A.~Demlow and A.~N. Hirani, \emph{A posteriori error estimates for finite
  element exterior calculus: the de {R}ham complex}, Found. Comput. Math.
  \textbf{14} (2014), 1337--1371.

\bibitem{mmg3d}
C.~Dobrzynski, \emph{{MMG3D}: User guide}, Tech. Report 422, Inria, 2012.

\bibitem{dorf_2006a}
R.~C. Dorf, \emph{Electronics, power electronics, optoelectronics, microwaves,
  electromagnetics and radar}, Taylor \& Francis, 2006.

\bibitem{dorfler_1996a}
W.~D\"orfler, \emph{A convergent adaptive algorithm for {P}oisson's equation},
  SIAM J. Numer. Anal. \textbf{33} (1996), 1106--1124.

\bibitem{dorfler_sauter_2013a}
W.~D\"orfler and S.~Sauter, \emph{A posteriori error estimation for highly
  indefinite {H}elmholtz problems}, Comput. Meth. Appl. Math. \textbf{13}
  (2013), 333--347.

\bibitem{feng_lu_xu_2016a}
X.~Feng, P.~Lu, and X.~Xu, \emph{A hybridizable discontinuous {G}alerkin method
  for the time-harmonic {M}axwell equations with high wave number}, Comput.
  Methods Appl. Math. \textbf{16} (2016), 429--445.

\bibitem{girault_raviart_1986a}
V.~Girault and P.~A. Raviart, \emph{Finite element methods for
  {N}avier-{S}tokes equations: theory and algorithms}, Springer-Verlag, 1986.

\bibitem{griffiths_1999a}
D.~J. Griffiths, \emph{Introduction to {E}lectrodynamics}, Prentice Hall, 1999.

\bibitem{hesthaven_warburton_2002a}
J.~S. Hesthaven and T.~Warburton, \emph{Nodal high-order methods on
  unstructured grids. {P}art {I}. {T}ime-domain solution of {M}axwell's
  equations}, J. Comput. Phys. \textbf{181} (2002), 1266--1288.

\bibitem{hiptmair_pechstein_2019}
R.~Hiptmair and C.~Pechstein, \emph{Discrete regular decompositions of
  tetrahedral discrete 1-forms}, in \textit{Maxwell’s Equations}, De Gruyter,
  Berlin, 2019, pp.~199--258.

\bibitem{li_lanteri_perrussel_2014a}
L.~Li, S.~Lanteri, and R.~Perrussel, \emph{A hybridizable discontinuous
  {G}alerkin method combined to a {S}chwarz algorithm for the solution of 3d
  time-harmonic {M}axwell's equations}, J. Comput. Phys. \textbf{256} (2014),
  563--581.

\bibitem{melenk_2005a}
J.~M. Melenk, \emph{{$hp$-interpolation of nonsmooth functions and an
  application to $hp$-a posteriori error estimation}}, SIAM J. Numer. Anal.
  \textbf{43} (2005), 127--155.

\bibitem{melenk_sauter_2011a}
J.~M. Melenk and S.~Sauter, \emph{Wavenumber explicit convergence analysis for
  {G}alerkin discretizations of the {H}elmholtz equation}, SIAM J. Numer. Anal.
  \textbf{49} (2011), 1210--1243.

\bibitem{melenk_sauter_2021a}
J.~M. Melenk and S.~A. Sauter, \emph{Wavenumber-explicit $hp$-{FEM} analysis
  for {M}axwell's equations with transparent boundary conditions}, Found.
  Comput. Math. \textbf{21} (2021), 125--241.

\bibitem{melenk_sauter_2022}
\bysame, \emph{Wavenumber-explicit $hp$-{FEM} analysis for {M}axwell's
  equations with impedance boundary conditions}, preprint
  \href{https://arxiv.org/abs/2201.02602}{arXiv:2201.02602}, 2022.

\bibitem{melenk_wohlmuth_2001a}
J.~M. Melenk and B.I. Wohlmuth, \emph{{On residual-based a posteriori
  estimation in $hp$-FEM}}, Adv. Comput. Math. \textbf{15} (2001), 311--331.

\bibitem{monk_2003a}
P.~Monk, \emph{Finite element methods for {M}axwell's equations}, Oxford
  University Press, New York, 2003.

\bibitem{nedelec_1986a}
J.~C. N\'ed\'elec, \emph{A new family of mixed finite elements in {$\mathbb
  R^3$}}, Numer. Math. \textbf{50} (1986), 57--81.

\bibitem{nguyen_peraire_cockburn_2011a}
N.~C. Nguyen, J.~Peraire, and B.~Cockburn, \emph{Hybridizable discontinuous
  {G}alerkin methods for the time-harmonic {M}axwell's equations}, J. Comput.
  Phy. \textbf{230} (2011), 7151--7175.

\bibitem{nicaise_creuse_2003a}
S.~Nicaise and E.~Creus\'e, \emph{A posteriori error estimation for the
  heterogeneous {M}axwell equations on isotropic and anisotropic meshes},
  Calcolo \textbf{40} (2003), 249--271.

\bibitem{nicaise_tomezyk_2019a}
S.~Nicaise and J.~Tomezyk, \emph{Convergence analysis of a $hp$-finite element
  approximation of the time-harmonic {M}axwell equations with impedance
  boundary conditions in domains with an analytic boundary}, Numer. Methods
  Partial Differ. Equ. \textbf{36} (2020), 1868--1903.

\bibitem{perugia_schotzau_2003}
I.~Perugia and D.~Sch\"otzau, \emph{The $hp$-local discontinuous {G}alerkin
  method for low-frequency time-harmonic {M}axwell equations}, Math. Comp.
  \textbf{72} (2003), 1179--1214.

\bibitem{sauter_zech_2015a}
S.~Sauter and J.~Zech, \emph{A posteriori error estimation of $hp$-dg finite
  element methods for highly indefinite {H}elmholtz problems}, SIAM J. Numer.
  Anal. \textbf{53} (2015), 2414--2440.

\bibitem{sauter_schwab_2010a}
S.A. Sauter and C.~Schwab, \emph{Boundary element methods}, Springer, 2010.

\bibitem{schoberl_2008}
J.~Sch\"oberl, \emph{A posteriori error estimates for {M}axwell equations},
  Math. Comp. \textbf{77} (2008), 633--649.

\bibitem{schwab_1998a}
C.~Schwab, \emph{$p-$ and $hp-$finite element methods}, Oxford Univ. Press,
  1998.

\bibitem{taflove_hagness_2005a}
A.~Taflove and S.C. Hagness, \emph{Computational electrodynamics the
  finite-difference time-domain method}, Artch house, 2005.

\bibitem{verfurth_1994}
R.~Verf{\"u}rth, \emph{A posteriori error estimation and adaptive
  mesh-refinement techniques}, J. Comput. Appl. Math. \textbf{50} (1994),
  67--83.

\bibitem{viquerat_2015a}
J.~Viquerat, \emph{Simulation of electromagnetic waves propagation in
  nano-optics with a high-order discontinuous {G}alerkin time-domain method},
  Ph.D. thesis, Universit\'e Nice Sophia-Antipolis and Inria project-team
  Nachos, 2015.

\end{thebibliography}

\end{document}